\documentclass[11pt,a4paper]{article}
\usepackage[utf8]{inputenc}
\usepackage[english]{babel}
\usepackage[T1]{fontenc}
\usepackage{amsmath, amsfonts, amssymb, amsthm}
\usepackage{lmodern}
\usepackage{fullpage}
\usepackage{microtype}
\usepackage[smalltableaux=true]{ytableau}
\usepackage[shortlabels]{enumitem}
\usepackage{mathtools}
\usepackage{tikz}
\usepackage[colorlinks=true, linkcolor=blue]{hyperref}

\numberwithin{equation}{section}

\newtheorem{theorem}[equation]{Theorem}
\newtheorem{proposition}[equation]{Proposition}
\newtheorem{lemma}[equation]{Lemma}
\newtheorem{corollary}[equation]{Corollary}
\newtheorem{theoremintro}{Theorem}

\theoremstyle{definition}
\newtheorem{definition}[equation]{Definition}

\theoremstyle{remark}
\newtheorem{remark}[equation]{Remark}
\newtheorem{example}[equation]{Example}

\newcommand\Z{\mathbb{Z}}
\newcommand\Ze{\Z/e\Z}
\newcommand\N{\mathbb{N}}
\newcommand\chargesmall{s}
\newcommand\charge{\mathsf\chargesmall}
\newcommand\chargez{\mathsf{0}}
\newcommand\chargeinteger[1]{\mathsf{#1}}

\newcommand\mcharge{{\tuple{\charge}}}
\newcommand\qcharge{b} 
\newcommand\bnumber[1]{\beta^{#1}}
\newcommand\bfayers{\mathfrak{b}}
\newcommand\xrostam{y}
\newcommand\xrostamod{x}
\newcommand\res[1]{\mathrm{res}^{#1}}
\newcommand\ci[2][i]{c^{#2}_{#1}}
\newcommand\cialpha[1][i]{c_{#1}}
\newcommand\nodes{\N^2}
\newcommand\mnodes[1][r]{\nodes \times \{1,\dots,#1\}}
\newcommand\boundcoreblock[1][r]{N_{#1,e}}
\newcommand\bestboundcoreblockgen[2]{N_{#1,#2}}
\newcommand\bestboundcoreblock[1][r]{\bestboundcoreblockgen{#1}{e}}
\newcommand\bestboundcoreblocke[1][e]{\bestboundcoreblockgen{r}{#1}}
\newcommand\idealbound[2]{N'_{#1,#2}}
\newcommand\qbound[2]{Q_{#1,#2}}

\newcommand\sfayers{\sigma}

\newcommand\Q{Q}
\newcommand\Qpos{Q_+}
\newcommand\Qcharge[1][\charge]{Q^{#1}}
\newcommand\Qmcharge{Q^\mcharge}
\newcommand\constQ{\delta}

\newcommand\spart[1]{\sigma #1}
\newcommand\tspart\spart

\newcommand\sblockprim[1][\ep]{\sigma}
\newcommand\sblock[1]{\sblockprim\cdot #1}
\newcommand\tuple\boldsymbol
\newcommand\blam{\tuple\lambda}
\newcommand\bmu{\tuple\mu}
\newcommand\bempty{\tuple\emptyset}

\newcommand\x[1][]{\xrostamod_{#1}}

\newcommand\ep{e'}
\newcommand\ip{i'}
\newcommand\ordere{p}
\newcommand\diffcont{z}

\newcommand\transpose\intercal
\newcommand\partse[1][e]{{\mathcal{P}_{#1}}}
\newcommand\partsekgen[2]{\mathcal{P}_{#1,#2}}
\newcommand\partsek[1][k]{\partsekgen{e}{#1}}
\newcommand\matq[1][e]{A_{#1}}
\newcommand\matqk[1][k]{A_{e,#1}}
\newcommand\qe[1][e]{q_{#1}}
\newcommand\qek[1][k]{q_{e,#1}}
\newcommand\mate[1][E]{\mathcal{#1}}

\newcommand\Qc{\Qcharge[\chargez]_*}
\newcommand\Qp{\Qcharge[\chargez]}
\newcommand\alphap{\alpha'}

\newcommand\w{\mathsf{w}}
\newcommand\we[1][e]{\w_{#1}}
\newcommand\wc[1][\charge]{\w^{#1}}
\newcommand\wmc[1][\mcharge]{\w^{#1}}

\newcommand\Hnr{\mathcal{H}_n^{\mcharge}}

\allowdisplaybreaks[3]

\newcounter{pos}
\newcommand{\scaletikz}{.6}
\newcommand{\scalelozenge}{.7}
\newcommand{\cabacus}[3]{
	\begin{center}
	\begin{tikzpicture}[scale=\scaletikz, baseline=(current bounding box.center)]
		\draw[very thick] (0, .5) -- (0, -.5);
		\setcounter{pos}{0}
		\foreach \b in {#2}
			{
			\ifnum \value{pos} = 0
				\draw (#1, 0) -- (#1+\b+2, 0) node[right] {$\cdots$};
				\foreach \i in {-1,...,\b}
					{
					\draw[very thin] (#1+\i+1, .1) -- (#1+\i+1, -.1);
					}
			\fi
			\stepcounter{pos}
			\fill (#1 + \b - \value{pos}, 0) circle (.2);
			}
			
		\foreach \i in {1,2}
			{
			\stepcounter{pos}
			\fill (#1 - \value{pos}, 0) circle (.2);
			}
			
		\foreach \i in {0,...,\value{pos}}
			{
			\draw[very thin] (#1-\i, .1) -- (#1-\i, -.1);
			}
		\stepcounter{pos}
		\draw (#1, 0) -- (#1-\value{pos}, 0) node[left]{$\cdots$};
		
	\end{tikzpicture} #3
	\end{center}
}

\newcommand{\ceabacus}[4]{
	\begin{center}
	\begin{tikzpicture}[scale=\scaletikz, baseline=(current bounding box.center),yscale=-1]
		\draw[very thick] (0, 0) -- (0, -#1 - 1);
		\setcounter{pos}{-#4}
		\foreach \b in {#2}
			{
			\stepcounter{pos}
			\ifnum \value{pos} = 1
				\pgfmathparse{(\b - \value{pos} - Mod(\b - \value{pos}, #1)) / #1}
				\foreach \i in {-1,...,-#1}
					{
					\draw (\pgfmathresult + 3, \i) -- (0, \i);
					\draw (\pgfmathresult + 3.6, \i - .03) node {$\dots$};
					\foreach \j in {-2,...,\pgfmathresult}
						\draw[very thin] (\j + 2, \i - .1) -- ++ (0, .2);
					}
			\fi
			\pgfmathparse{Mod(\b - \value{pos}, #1)}
			\fill (
					{(\b - \value{pos} - \pgfmathresult) / #1}
					,
					{- \pgfmathresult - 1}
					) circle (.2);
			}
			\foreach \i in {1,...,#1}
			{
			\stepcounter{pos}
			\pgfmathparse{Mod(-\value{pos}, #1) == #1 - 1}
			\ifnum \pgfmathresult = 0
				\pgfmathparse{Mod(-\value{pos}, #1)}
				\fill (
					{(-\value{pos} - \pgfmathresult) / #1}
					,
					{-\pgfmathresult - 1}
					) circle (.2);
			\else
				\breakforeach
			\fi
			}
			\foreach \k in {1,2}
			\foreach \i in {1,...,#1}
				{
				\pgfmathparse{Mod(-\value{pos}, #1)}
				\fill (
					{(-\value{pos} - \pgfmathresult) / #1}
					,
					{-\pgfmathresult - 1}
					) circle (.2);
				\stepcounter{pos}
				}

			\pgfmathparse{(-\value{pos} - Mod(-\value{pos}, #1)) / #1 + 1}
			\foreach \i in {-1,...,-#1}
				{
				\draw (\pgfmathresult - 1, \i) -- (0, \i);
				\draw (\pgfmathresult - 1.5, \i - .03) node {$\dots$};
			
				\foreach \j in {-\value{pos},...,0}
					{
					\ifnum \j > \pgfmathresult
						\draw[very thin] (\j, \i - .1) -- ++ (0, .2);
					\fi;
					}
				}
		
	\end{tikzpicture} #3
	\end{center}
}

\newcommand{\abacuscore}[4]{
	\begin{center}
	\begin{tikzpicture}[scale=\scaletikz, baseline=(current bounding box.center), yscale=-1]
		\draw[very thick] (0, 0) -- (0, -#1 - 1);
		\setcounter{pos}{-#4}
		\foreach \b in {#2}
			{
			\stepcounter{pos}
			\ifnum \value{pos} = 1
				\pgfmathparse{(\b - \value{pos} - Mod(\b - \value{pos}, #1)) / #1}
				\foreach \i in {-1,...,-#1}
					{
					\draw (\pgfmathresult + 3, \i) -- (0, \i);
					\draw (\pgfmathresult + 3.6, \i - .03) node {$\dots$};
					\foreach \j in {-2,...,\pgfmathresult}
						\draw[very thin] (\j + 2, \i - .1) -- ++ (0, .2);
					}
			\fi
			\pgfmathparse{Mod(\b - \value{pos}, #1)}
			\fill (
					{(\b - \value{pos} - \pgfmathresult) / #1}
					,
					{- \pgfmathresult - 1}
					) circle (.2);
			}
			\foreach \i in {1,...,#1}
			{
			\stepcounter{pos}
			\pgfmathparse{Mod(-\value{pos}, #1) == #1 - 1}
			\ifnum \pgfmathresult = 0
				\pgfmathparse{Mod(-\value{pos}, #1)}
				\fill (
					{(-\value{pos} - \pgfmathresult) / #1}
					,
					{-\pgfmathresult - 1}
					) circle (.2);
			\else
				\breakforeach
			\fi
			}
			\foreach \k in {1,2}
			\foreach \i in {1,...,#1}
				{
				\pgfmathparse{Mod(-\value{pos}, #1)}
				\fill (
					{(-\value{pos} - \pgfmathresult) / #1}
					,
					{-\pgfmathresult - 1}
					) circle (.2);
				\stepcounter{pos}
				}

			\pgfmathparse{(-\value{pos} - Mod(-\value{pos}, #1)) / #1 + 1}
			\foreach \i in {-1,...,-#1}
				{
				\draw (\pgfmathresult - 1, \i) -- (0, \i);
				\draw (\pgfmathresult - 1.5, \i - .03) node {$\dots$};
			
				\foreach \j in {-\value{pos},...,0}
					{
					\ifnum \j > \pgfmathresult
						\draw[very thin] (\j, \i - .1) -- ++ (0, .2);
					\fi;
					}
				}
				
		\setcounter{pos}{-#4}
		\foreach \b in {#2}
			{
			\stepcounter{pos}
			\pgfmathparse{Mod(\b - \value{pos}, #1)}
			\draw[red] (
					{(\b - \value{pos} - \pgfmathresult) / #1 + 1}
					,
					{- \pgfmathresult - 1}
					) node[scale=\scalelozenge] {$\blacklozenge$};
			}
			\foreach \i in {1,...,#1}
			{
			\stepcounter{pos}
			\pgfmathparse{Mod(-\value{pos}, #1) == #1 - 1}
			\ifnum \pgfmathresult = 0
				\pgfmathparse{Mod(-\value{pos}, #1)}
				\draw[red] (
					{(-\value{pos} - \pgfmathresult) / #1 + 1}
					,
					{-\pgfmathresult - 1}
					) node[scale=\scalelozenge] {$\blacklozenge$};
			\else
				\breakforeach
			\fi
			}
			\foreach \k in {1,2}
			\foreach \i in {1,...,#1}
				{
				\pgfmathparse{Mod(-\value{pos}, #1)}
				\draw[red] (
					{(-\value{pos} - \pgfmathresult) / #1 + 1}
					,
					{-\pgfmathresult - 1}
					) node[scale=\scalelozenge] {$\blacklozenge$};
				\stepcounter{pos}
				}

		\setcounter{pos}{-#4}
		\foreach \b in {#2}
			{
			\stepcounter{pos}
			\pgfmathparse{Mod(\b - \value{pos}, #1)}
			\fill (
					{(\b - \value{pos} - \pgfmathresult) / #1}
					,
					{- \pgfmathresult - 1}
					) circle (.2);
			}
			\foreach \i in {1,...,#1}
			{
			\stepcounter{pos}
			\pgfmathparse{Mod(-\value{pos}, #1) == #1 - 1}
			\ifnum \pgfmathresult = 0
				\pgfmathparse{Mod(-\value{pos}, #1)}
				\fill (
					{(-\value{pos} - \pgfmathresult) / #1}
					,
					{-\pgfmathresult - 1}
					) circle (.2);
			\else
				\breakforeach
			\fi
			}
			\foreach \k in {1,2}
			\foreach \i in {1,...,#1}
				{
				\pgfmathparse{Mod(-\value{pos}, #1)}
				\fill (
					{(-\value{pos} - \pgfmathresult) / #1}
					,
					{-\pgfmathresult - 1}
					) circle (.2);
				\stepcounter{pos}
				}
	\end{tikzpicture} #3
	\end{center}
}

\author{Salim \textsc{Rostam}\thanks{Univ Rennes, CNRS, IRMAR - UMR 6625, F-35000 Rennes, France}}
\title{Identifying Young diagrams\\ among residue multisets}

\begin{document}
\maketitle

\abstract{To any Young diagram we can associate the multiset of residues of all its nodes. This paper is concerned with the inverse problem: given a multiset of elements of $\Z/e\Z$, does it comes from a Young diagram? We give a full solution in level one and a partial answer in higher levels for Young multidiagrams, using Fayers's notions of core block and weight of a multipartition. We apply the result in level one to study a shift operation on partitions.}


\section{Introduction}

Studying the representation theory of a finite group $G$ over a field $k$ of characteristic zero reduces to find the irreducible representations of $G$. If $G = \mathfrak{S}_n$ is the symmetric group on $n$ letters, the irreducible representations are indexed by the \emph{partitions} of $n$, that is, non-increasing sequences of positive integers $\lambda = (\lambda_1 \geq \dots \geq \lambda_h > 0)$ with sum $\lvert \lambda \rvert = n$. If the field $k$ is of positive characteristic $p$, some representations may not be written as a direct sum of irreducible ones. Hence, we are also interested in the \emph{blocks} of the group algebra, that is, indecomposable two-sided ideals. These blocks are parametrised by the \emph{$p$-cores} of the partitions of $n$, in particular, any block is uniquely determined by its $p$-core and its \emph{$p$-weight}. Note that the set of irreducible representations can be parametrised by the \emph{$p$-regular} partitions of $n$.

More generally, we can replace the symmetric group $\mathfrak{S}_n$ by a \emph{complex reflection group}. The set of \emph{irreducible} complex reflection groups consists of an infinite family $\{G(r,p,n)\}_{r,p,n}$ where $r, p, n$ are positive integers with $p \mid r$, and also a finite number of exceptions (see~\cite{shephard-todd}). The complex reflection group $G(r,1,n)$ is isomorphic to $(\Z/r\Z)\wr\mathfrak{S}_n \simeq (\Z/r\Z)^n \rtimes \mathfrak{S}_n$ and can be seen as the set of $n\times n$ monomial matrices with non-zero entries in the set of complex $r$-roots of unity, while $G(r,p,n)$ is a certain subgroup of $G(r,1,n)$ of index $p$. We can then study the representation theory of a \emph{Hecke algebra} $\Hnr(q)$ of $G(r,1,n)$, where $\mcharge  = (\charge_1,\dots,\charge_r)\in \Z^r$ is a multicharge (\cite{ariki-koike,broue-malle,bmr}). The algebra $\Hnr(q)$ is  a particular deformation of the group algebra of $G(r,1,n)$ and $q \in k$ is the deformation parameter. Assume that $q \in k\setminus\{0\}$ and let $e \geq 0$ be its multiplicative order (with $e \coloneqq 0$ if $q$ is of infinite order). The representation theories of $\Hnr(q)$ and $G(r,1,n)$ are deeply linked. For instance, if $r = 1$ then $G(1,1,n) \simeq \mathfrak{S}_n$ and the situation is the following: if $\Hnr(q)$ is semisimple then its irreducible representations are indexed by the partitions of $n$, otherwise they are parametrised by the $e$-regular partitions of $n$ and the blocks are parametrised by the $e$-cores of partitions of~$n$.
In the general case $r \geq 1$, if $\Hnr(q)$ is semisimple then its irreducible representations are indexed by the \emph{$r$-partitions} of $n$, that is, $r$-tuples $\blam = \bigl(\lambda^{(1)},\dots,\lambda^{(r)}\bigr)$ with $\lvert \blam \rvert = \lvert\lambda^{(1)}\rvert + \dots + \lvert\lambda^{(r)}\rvert = n$.  On the contrary, if $\Hnr(q)$ is non-semisimple then the situation is more complex. First, its irreducible representations can be indexed by a non-trivial generalisation of $e$-regular partitions, known as \emph{Kleshchev} $r$-partitions (see~\cite{ariki:classification,ariki-mathas}). Similarly, the naive generalisation of $e$-cores to $r$-partitions, the \emph{$e$-multicores}, do not parametrise in general the blocks of $\Hnr(q)$. Namely, Lyle and Mathas~\cite{lyle-mathas:blocks} proved that the blocks of $\Hnr(q)$ are parametrised by the multisets of $\mcharge$-residues modulo $e$ of the $r$-partitions of $n$. To make things more explicit, let $\Q = \oplus_{i \in \Ze} \Z \alpha_i$ be a free abelian group. If $\blam$ is a multipartition, for any $i \in \Ze$ we denote by $\ci{\mcharge}(\blam)$ the number of $i$-nodes of $\blam$. Then the block \emph{corresponding} to $\blam$ is
\[
\alpha^\mcharge(\blam) \coloneqq \sum_{i \in \Ze} \ci{\mcharge}(\blam) \alpha_i \in \Q,
\]
and we say that $\blam$ \emph{lies} in $\alpha^\mcharge(\blam)$.
During their proof, Lyle and Mathas used a generalisation of the $e$-weight to $r$-partitions, introduced by Fayers~\cite{fayers:weights}. If $\blam = \bigl(\lambda^{(1)},\dots,\lambda^{(r)}\bigr)$ is an $r$-partition, its \emph{$\mcharge$-weight} is given by
\[
\wmc(\blam) \coloneqq \sum_{j = 1}^r \ci[\charge_j]{\mcharge}(\blam) - \frac{1}{2}\sum_{i \in \Ze}\left(\ci{\mcharge}(\blam) - \ci[i+1]{\mcharge}(\blam)\right)^2.
\]
In particular, two $r$-partitions that lie in the same block (\textit{i.e.} $\alpha^\mcharge(\blam)=\alpha^\mcharge(\bmu)$) have the same $\mcharge$-weight.
Fayers proved that if $r = 1$ then this definition of $\charge$-weight coincides with the usual notion of $e$-weight for partitions. For instance, if the partition $\lambda$ is an $e$-core we have:
\[
\cialpha[0](\lambda) = \frac{1}{2}\sum_{i \in \Ze} \left(\cialpha(\lambda) - \cialpha[i+1](\lambda)\right)^2
\]
 (in level $1$ we may omit to write the multicharge, assuming, without loss of generality, that $\charge = \chargez$). Moreover, while the fact that the $e$-weight of a partition  is non-negative follows from the definition, this is a non-trivial statement for the $\mcharge$-weight in higher levels (see~\cite{fayers:weights}).

In level one, a partition is an $e$-core if and only if its $e$-weight is zero. In higher levels, we still have an implication: if an $r$-partition has $\mcharge$-weight at most $r-1$ then it is an $e$-multicore. However, we can find $e$-multicores of arbitrary large $\mcharge$-weight. Moreover, we can have two $r$-partitions $\blam$ and $\bmu$ lying in the same block, thus having the same $\mcharge$-weight, but with $\blam$ (respectively $\bmu$) being (resp. not being) an $e$-multicore.  In order to obtain a more satisfying notion of core of a multipartition, Fayers~\cite{fayers:core} introduced the notion of \emph{core block} and \emph{reduced $e$-multicore} (the latter expression is due to~\cite{lyle-mathas:blocks}): an $e$-multicore $\blam$ is \emph{$\mcharge$-reduced} if any multipartition $\bmu$ such that $\alpha^\mcharge(\bmu) = \alpha^\mcharge(\blam)$ is an $e$-multicore, in which case the block associated with $\blam$ is a \emph{core block}. If $r = 1$ then every $e$-core is $\charge$-reduced and  every block is a core block. Now if $r > 1$, to any multipartition we can still associate a unique core block, however there is no canonical choice for a reduced $e$-multicore inside this core block. Note that Jacon--Lecouvey~\cite{jacon-lecouvey} managed to define what they have called the \emph{$(e, \mcharge)$-core} of a multipartition (they also use the notion of \emph{reduced $(e, \mcharge)$-core}, which is different from the notion of reduced $e$-multicore that we use here).  The $(e,\mcharge)$-core of an $r$-partition is again an $r$-partition, and the situation is then entirely similar to the level one case, namely, for the combinatorics of blocks. However, the $(e, \mcharge)$-core of a multipartition is associated with a possibly different multicharge, which depends on the multipartition.

Now let $p \mid r$. We can use Clifford theory to study the representation theory of $G(r,p,n)$, and this involves the following natural \emph{shift} operation of order $p$ on $r$-partitions:
\[
\blam = \bigl(\lambda^{(1)},\dots,\lambda^{(r)}\bigr) \longmapsto \prescript{\sigma}{}{\blam} \coloneqq \bigl(\lambda^{(r-d+1)},\dots,\lambda^{(r)},\lambda^{(1)},\dots,\lambda^{(r-d)}\bigr),
\]
 where $d \coloneqq \frac{r}{p}$. Typically, if $S^{\blam}$ is an irreducible $G(r,1,n)$-module then the number of irreducible constituents of $S^{\blam}$ seen as an $G(r,p,n)$-module  will only depend on the cardinality of the orbit of $\blam$ under the shift operation (see, for instance, \cite{ariki:representation,chlouveraki-jacon,genet-jacon,hu-mathas:decomposition}).
Now take $e \geq 2$ and assume that $p$ also divides $e$. There is a map $\sigma : \Q \to \Q$ of order $p$ such that if  $\mcharge \in \Z^r$ is a multicharge satisfying some compatibility conditions, we have the relation
\begin{equation}
\label{equation:intro_shift_block_parts}
\alpha^{\mcharge}\bigl(\prescript{\sigma}{}{\blam}) = \sigma\cdot \alpha^\mcharge(\blam).
\end{equation}
In~\cite{rostam:cyclotomic}, the author proved that if $r \geq 2$ then a block $\alpha \in \Qmcharge$ that is \emph{stuttering}, that is $\sblock\alpha=\alpha$, always corresponds to a \emph{stuttering} multipartition, that is, a multipartition $\blam$ satisfying $\prescript{\sigma}{}{\blam} = \blam$.

\bigskip

The aim of this paper is to study the natural generalisation of Fayers's weight function to $\Q = \oplus_{i \in \Ze} \Z\alpha_i$, and then apply this result in level one to define a shift operation on partitions so that a relation such as~\eqref{equation:intro_shift_block_parts} holds. More precisely, let $\mcharge \in \Z^r$ be a multicharge and for any $\alpha = \sum_{i \in \Ze} \cialpha\alpha_i \in \Q$ define its $\mcharge$-weight by:
\[
\wmc(\alpha) \coloneqq\sum_{j = 1}^r \cialpha[\charge_j] - \frac{1}{2}\sum_{i \in \Ze}\left(\cialpha - \cialpha[i+1]\right)^2 \in \Z,
\]
so that $\wmc\bigl(\alpha^\mcharge(\blam)\bigr) = \wmc(\blam)$ for any $r$-partition $\blam$. We also define the set
\[
\Qmcharge \coloneqq \bigl\{\alpha^\mcharge(\blam) : \blam \text{ is an } r\text{-partition}\bigr\} \subseteq \Q.
\]
Since $\wmc(\blam) \geq 0$ for any $r$-partition $\blam$, we have an inclusion $\Qmcharge \subseteq \{\alpha \in Q : \wmc(\alpha) \geq 0\}$. The paper is mainly concerned in studying a reverse inclusion. More precisely, we prove the following results.

\begin{theoremintro}[Corollaries~\ref{corollary:Qmcharge_supset_superlevel_weight} and~\ref{corollary:sharpest_upper_bound_bestboundcoreblock}]
\label{theoremintro:C}
Let $r \geq 2$ and $e > 0$. We have
\[
\Qmcharge \supseteq \bigl\{\alpha \in \Q : \wmc(\alpha) > \idealbound r e - r\bigr\},
\]
where $\idealbound r e \coloneqq \left\lfloor\frac{r^2}{2e}\left\lfloor\frac{e^2}{4}\right\rfloor\right\rfloor$.
\end{theoremintro}

\begin{theoremintro}[Propositions~\ref{proposition:partitions_wgeq0} and~\ref{proposition:Qmcharge_semialg_particular_case}]
\label{theoremintro:0}
Let $r \geq 1$ and $e \geq 0$. Assume that $r = 1$ or $(r, e) \in  \bigl\{(2,2), (2,3), (3,2)\bigr\}$. Then
\[
\Qmcharge = \bigl\{\alpha \in \Q : \wmc(\alpha) \geq 0\bigr\}.
\]
In particular, given a multiset of elements of $\mathbb{Z}/e\mathbb{Z}$, we can easily determine whether it comes from a Young diagram by computing its weight.
\end{theoremintro}

Note that Theorem~\ref{theoremintro:C} does not hold when $e = 0$, whereas Theorem~\ref{theoremintro:0} does (when $r = 1$).  If $e > 0$, a set of the form $\{\alpha \in \Q : \wmc(\alpha) \geq C\bigr\}$ for $C \geq 0$ is never empty (and even always infinite), so that the inclusion of Theorem~\ref{theoremintro:C} is not trivially true. The main idea of the proof of Theorem~\ref{theoremintro:C} is to prove that the weight function is bounded above by a constant $\bestboundcoreblock$ on the set of \emph{reduced} $e$-multicores. Note that this assertion is wrong if we only consider $e$-multicores, which can have arbitrarily large weight.  With some calculations involving binary matrices, we then give a sharp estimation of the bound $\bestboundcoreblock$. Interestingly, as it is mentioned in the statement, Theorem~\ref{theoremintro:0} gives a simple criterion to determine whether an element of $Q$, in particular, a multiset of residues, actually comes from a partition. To the author's knowledge, such a result was not already known and is a non-trivial generalisation of an old result of Robinson--Thrall~\cite{robinson-thrall} from $e = 0$ to $e \geq 0$. We will also give a procedure to compute all the corresponding partitions.

Now assume that $r = 1$, take $p$ dividing $e$ and let $\lambda$ be a partition. Shifting the components of the \emph{$e$-quotient} of $\lambda$, we define another partition $\sigma\lambda$. Using Theorem~\ref{theoremintro:0}, we prove the following analogue of~\eqref{equation:intro_shift_block_parts} in level one. 

\begin{theoremintro}[Corollary~\ref{corollary:case_spart_lambda_good}]
Let $\lambda$ be a partition. With $\ep\coloneqq\frac{e}{p}$, we have
\[
\alpha(\sigma\lambda) = \sigma\cdot\alpha(\lambda) \iff \lvert \sigma\lambda\rvert = \lvert \lambda \rvert \iff \cialpha[0](\lambda) = \cialpha[\ep](\lambda).
\]
\end{theoremintro}
In particular, if $\lambda$ is a partition of $n$ then $\spart\lambda$ is a partition of $m$ with  possibly $m \neq n$.  Finally, we complete the study of stuttering blocks initiated in~\cite{rostam:stuttering}, by giving its analogue in level one. It turns out that a similar equivalence holds, now with an additional condition on the weight of the block.

\begin{theoremintro}[Lemma~\ref{lemma:stuttering_partition} and Corollary~\ref{corollary:stuttering_blocks}]
\label{theoremintro:stuttering}
Assume that $r = 1$ and let $\alpha \in \Qcharge[\chargez]$. The block $\alpha$ corresponds to a partition $\lambda$ satisfying $\spart\lambda=\lambda$ if and only if $\sblock\alpha=\alpha$ and $p \mid \w(\alpha)$.
\end{theoremintro}

\bigskip

We now give a brief overview of the paper. Section~\ref{section:bacgroud_core_blocks} introduces the necessary material to define core blocks and give their proprieties as stated in~\cite{fayers:core}. In particular, in~\textsection\ref{subsection:partitions} we define the weight of any element of $Q = \oplus_{i \in \Ze} \Z \alpha_i$ and in~\textsection\ref{subsection:abaci} we define the abaci representations of a partition and recall how to recover the block associated with a partition from its abacus. We use these results to prove  that an element $\alpha \in \Q$ corresponds to an $e$-core if and only if $\alpha$ has weight $0$ (Lemma~\ref{lemma:cores_w=0}) and we deduce the case $r = 1$ of Theorem~\ref{theoremintro:0} (Proposition~\ref{proposition:partitions_wgeq0}). We deduce a simple criterion to determine whether a given collection of residues actually comes from a Young diagram (Corollary~\ref{corollary:existence_young_diagram_given_residues}). In~\textsection\ref{subsection:multipartitions} we recall Fayers's definition of weight for multipartitions, and we reprove that the weight of a multipartition is non-negative using abaci. We also recall an important result from~\cite{fayers:weights} (Proposition~\ref{proposition:weight_min_two_sets}), which expresses the weight of a multipartition as the minimum of the cardinalities of two sets. Finally, in~\textsection\ref{subsection:blocks_multipartitions} we recall from~\cite{fayers:core} the notion of core block and reduced multicore, and we prove that any element of $Q$ is canonically associated with a core block (Lemma~\ref{lemma:core_block_associated_with_any_alpha}). This leads to the definition of the \emph{$\mcharge$-core} of any element of $\Q$ (Definition~\ref{definition:s-core}). We conclude this section by recalling a key result from~\cite{fayers:core} (Proposition~\ref{proposition:multicore_equivalent_statements}), which characterises the abaci of reduced $e$-multicores.

Section~\ref{section:implicit_description_blocks} is the heart of the paper. In~\textsection\ref{subsection:weight_binary} we show that the weights of reduced $e$-multicores can be obtained via a simple functional on binary matrices (Lemma~\ref{lemma:weight_subsets}), where the columns are seen as characteristic vectors of subsets of $\{1,\dots,e\}$. In~\textsection\ref{subsection:bound_weight_core_blocks} we prove that for fixed $r$ and $e$ the weight of a reduced $e$-multicore is bounded above by a constant~$\bestboundcoreblock$  (Theorem~\ref{theorem:bound_core_blocks}). We then deduce the first part of Theorem~\ref{theoremintro:C}, stating that $\Qmcharge$ contains an (infinite) superlevel set for the weight function on $Q$ (Corollary~\ref{corollary:Qmcharge_supset_superlevel_weight}).  We also give the second part of Theorem~\ref{theoremintro:0}, using some results of the next subsection (Proposition~\ref{proposition:Qmcharge_semialg_particular_case}).
In~\textsection\ref{subsection:optimality} the aim is to give sharp bounds for the constant $\bestboundcoreblock$. We first compute $\bestboundcoreblock$ for $r = 2$ and $e = 2$ (Propositions~\ref{proposition:bound_r=2} and~\ref{proposition:bound_e=2}), and we give a lower bound for $\bestboundcoreblock$ using the superadditivity of the sequence $(\bestboundcoreblock)_{e\geq 1}$ (Corollary~\ref{corollary:lower_bound_bestboundcoreblock}). The computation of  a sharp upper bound is more elaborate. We use the fact that the computation of $\bestboundcoreblock$ reduces to maximising a certain quadratic form with integer coefficients on the $(2^e-1)$-sphere for the $1$-norm (Proposition~\ref{proposition:boundcoreblock_maxqr}). Namely, using a calculation taken from graph theory, we compute the eigenvalues of the matrix $\matqk = (\lvert E \cap F\rvert)_{E, F}$, where $E, F$ run over the subsets of $\{1,\dots,e\}$ of cardinality $k$ (Lemma~\ref{lemma:spectra_Aeq}). Note that the matrix $\matqk$ often appears in the literature (see, for instance,  \cite{ryser1,ryser2}). We then deduce an upper bound for $\bestboundcoreblock$ (Corollary~\ref{corollary:sharpest_upper_bound_bestboundcoreblock}), and we  compute other values of $\bestboundcoreblock$ for small $r$ or $e$, proving that $\bestboundcoreblock = \idealbound r e$ (where $\idealbound r e$ is the constant appearing in Theorem~\ref{theoremintro:C}) for (at least) $r \in \{2,4\}$ and $e \in \{2,\dots,6\}$ (Proposition~\ref{proposition:upper_bound_exact_small_cases}). We conclude the section by a quick study of the asymptotic behaviour of $\bestboundcoreblock$, for $e \to \infty$ (see~\eqref{equation:e_infinite}) and $r \to \infty$ (Corollary~\ref{corollary:r_infinite}).

Finally, Section~\ref{section:application} is devoted to an application of Theorem~\ref{theoremintro:0} in level one.  In~\textsection\ref{subsection:shifting_blocks} we define a shift operation $\alpha\mapsto\sblock\alpha$ on $\Q$ of order $p$ where $p$ divides $e$,  and we characterise the blocks $\alpha \in \Q$ such that both $\alpha$ and $\sblock\alpha$ correspond to a partition (Corollary~\ref{corollary:shifting_blocks}). In~\textsection\ref{subsection:shifting_partitions} we propose a definition of a shift operation $\lambda\mapsto\spart\lambda$ on the set of partitions, uniquely defined on the set of $e$-cores by the following requirement: if $\lambda$ is an $e$-core then $\alpha(\spart\lambda)$ is the core of $\sblock\alpha(\lambda)$. In particular, if $\sblock\alpha(\lambda) = \alpha(\mu)$ for some partition $\mu$ then $\spart\lambda$ is the $e$-core of $\mu$. Moreover, we prove that the two shift operations, on blocks and on partitions, are compatible under some conditions, giving an analogue of~\eqref{equation:intro_shift_block_parts} in level one (Corollary~\ref{corollary:case_spart_lambda_good}).
 Finally, in~\textsection\ref{subsection:some_properties} we give some properties of the map $\alpha(\lambda) \mapsto \sigma\cdot\alpha(\lambda)$. We first give some properties involving $e$-cores and $\ep$-cores, where $\ep \coloneqq \frac{e}{p}$ (Propositions~\ref{proposition:shift_share_same_core} and~\ref{proposition:e_ehat_core}). Then, as in~\cite{rostam:stuttering} we study the existence of a stuttering partition inside a stuttering block. More precisely, given a partition~$\lambda$ such that $\sigma\cdot \alpha(\lambda) =  \alpha(\lambda)$, we prove that there exists a partition~$\mu$ satisfying $\spart\mu=\mu$ and $\alpha(\mu) = \alpha(\lambda)$ if and only if $p$ divides the $e$-weight of $\lambda$ (Lemma~\ref{lemma:stuttering_partition} and Corollary~\ref{corollary:stuttering_blocks}).

\section{Background on core blocks}
\label{section:bacgroud_core_blocks}

We recall here the combinatorics of blocks of Ariki--Koike algebras. We will write $\N$ for $\Z_{\geq 0}$.  Let $e \in\N$. If $e > 0$ we identify $\Z/e\Z$ and $\{0, \dots, e-1\}$, and if $e = 0$ then $\Z/e\Z \simeq \Z$.

\subsection{Partitions}
\label{subsection:partitions}

Let $n \in \N$. Let $\lambda$ be a \emph{partition} of $n$, that is, a non-increasing sequence of positive integers $\lambda = (\lambda_1 \geq \dots \geq \lambda_h > 0)$ with sum $n$. We will write $\lvert \lambda\rvert \coloneqq n$ and $h(\lambda) \coloneqq h$. We denote by $\emptyset$ the empty partition, which satisfies $\lvert \emptyset\rvert = h(\emptyset) = 0$.  The Young diagram associated with $\lambda$ is the region of $\nodes$ given by
\[
\mathcal{Y}(\lambda) \coloneqq \left\{(a, b) \in \nodes : 1 \leq a \leq h(\lambda) \text{ and } 1 \leq b \leq \lambda_a\right\}.
\]
An element of $\mathcal{Y}(\lambda)$ is a \emph{node of $\lambda$}. More generally, we will call \emph{node}   any element of $\nodes$. A node $\gamma \notin \mathcal{Y}(\lambda)$ is \emph{addable} (respectively, $\gamma \in \mathcal{Y}(\lambda)$ is \emph{removable}) if $\mathcal{Y}(\lambda) \cup \{\gamma\}$ (resp. $\mathcal{Y}\setminus\{\gamma\}$) is the Young diagram of a partition.
A \emph{rim hook} of $\lambda$ is a subset of $\mathcal{Y}(\lambda)$ of the following form:
\[
r^\lambda_{(a, b)} \coloneqq \left\{ (a',b') \in \mathcal{Y}(\lambda) : a' \geq a,\, b' \geq b \text{ and } (a'+1,b'+1) \notin \mathcal{Y}(\lambda)\right\},
\]
where $(a, b) \in \mathcal{Y}(\lambda)$. We say that $r^\lambda_{(a, b)}$ is an \emph{$l$-rim hook} if it has cardinality $l$. Note that $1$-rim hooks are exactly removable nodes. 
The set $\mathcal{Y}(\lambda) \setminus r^{\lambda}_{(a, b)}$ is the Young diagram of a certain partition $\mu$, obtained by \emph{unwrapping} or \emph{removing} the rim hook $r^{\lambda}_{(a, b)}$ from $\lambda$.
Conversely, we say that $\lambda$ is obtained from~$\mu$ by \emph{wrapping on} or \emph{adding} the rim hook $r^{\lambda}_{(a, b)}$. We say that $\lambda$ is an \emph{$e$-core}  if~$\lambda$ has no $e$-rim hook. Note that if $e = 0$ then every partition is an $e$-core, and if $e = 1$ then the empty partition $\emptyset$ is the only $e$-core. In particular, the combinatorics of $1$-cores is very easy, and all the future statements in the paper can easily be proven for $e = 1$.

\begin{example}
\label{exemple:ruban}
We consider the partition $\lambda \coloneqq (3,2,2,1)$. An example of a $3$-rim hook is
\[
r^\lambda_{(3,1)} =
\begin{ytableau}
{}&{}&{}
\\
{}&{}
\\
\times&\times
\\
\times
\end{ytableau}\, ,
\]
and a $4$-rim hook is for instance
\[
r^\lambda_{(2,1)} = 
\begin{ytableau}
{} &{} & {}
\\
 {}&\times
\\
 \times& \times
\\ 
\times
\end{ytableau}\, .
\]
We can check that $\lambda$ has no $5$-rim hooks so it is a $5$-core. We will see in \textsection\ref{subsection:abaci} how to use abaci to easily know whether a partition is an $e$-core.
\end{example}

More generally, we can successively remove $e$-rim hooks to the partition $\lambda$ until we reach an $e$-core partition. This $e$-core is uniquely defined from $\lambda$, in particular it does not depend on the order we chose to remove the $e$-rim hooks (see Lemma~\ref{lemma:remove_hook_abacus}). We denote by $\overline\lambda$ the $e$-core of $\lambda$.

\begin{definition}
\label{definition:e-weight}
The \emph{$e$-weight} of $\lambda$, denoted by $\we(\lambda)$, is the number of $e$-rim hooks that we remove to $\lambda$ to obtain its $e$-core $\overline{\lambda}$.
 \end{definition}
 
  Note that $\lvert \lambda \rvert = \lvert \overline\lambda\rvert + \we(\lambda)e$ if $e > 0$ and $\we[0](\lambda) = 0$.
Now let $\charge \in \Z$ be a \emph{charge}. Given $i \in \Ze$, an \emph{$(i, \charge)$-node} or simply  \emph{$i$-node} is a node of $\charge$-residue $i$, where the \emph{$\charge$-residue} of the node $(a, b) \in \nodes$ is $\res{\charge}(a, b) \coloneqq b - a + \charge \pmod{e}$. More generally, if $i \in \Z$ then an $i$-node will be a node $\gamma \in \nodes$ such that $\res{\charge}(\gamma) = i \pmod{e}$. We denote by $\ci{\charge}(\lambda)$ the number of $(i,\charge)$-nodes of $\lambda$.

\begin{example}
Take $\charge=-1$ and $e = 4$. For the partition $\lambda = (5,3,3,1)$, in each box of $\mathcal{Y}(\lambda)$ we place the value of the  corresponding $\charge$-residues:
\[
\ytableaushort{30123,230,123,0}\,,
\]
thus:
\[
\ci[0]{\charge}(\lambda) = 3, \quad \ci[1]{\charge}(\lambda) = 2, \quad \ci[2]{\charge}(\lambda) = 3, \quad \ci[3]{\charge}(\lambda) = 4.
\]
\end{example}

 Note the following simple equality, where~$\chargez$ denotes the charge $0 \in \Z$:
\begin{equation}
\label{equation:charges_charge0}
\ci{\charge}(\lambda) = \ci[i-\charge]{\chargez}(\lambda).
\end{equation}

\begin{remark}
\label{remark:ci}
Let $i \in \Z$. The integer  $\ci{\charge}(\lambda)$ does not depend on $e$ if $e \geq \max\bigl(\lvert \lambda \rvert , \lvert i - \charge\rvert\bigr)$, in which case the value of $\ci{\charge}(\lambda)$ is the one for $e = 0$. In particular, we will sometimes deal with the case $e= 0$ by taking $e \gg 0$. Besides, we have $\ci{\charge}(\lambda) = 0$ as soon as $e \geq \lvert i-\charge\rvert \geq\lvert\lambda\rvert$ (as soon as $\lvert i  - \charge\rvert \geq \lvert\lambda\rvert$ if $e = 0$).
\end{remark}

A simple consequence of the definition of a rim hook is the following.

\begin{lemma}
\label{lemma:add_rim_hook_residues}
Let $\lambda$ be a partition. Any $e$-rim hook of $\lambda$ has exactly one node of residue $i$ for each $i \in \Ze$.
\end{lemma}

Fayers~\cite{fayers:weights} introduced another weight function on partitions: the \emph{$\charge$-weight} of $\lambda$, given by
\begin{equation}
\label{equation:def_wc_partitions}
\wc(\lambda) \coloneqq \ci[\charge]{\charge}(\lambda) - \frac{1}{2}\sum_{i \in \Ze} \left(\ci{\charge}(\lambda) - \ci[i+1]{\charge}(\lambda)\right)^2.
\end{equation}
It is well-defined for $e = 0$ by Remark~\ref{remark:ci}.
We will see that the $\charge$-weight coincide with the $e$-weight, however these weights will have different generalisations to multipartitions (see~\textsection\ref{subsection:multipartitions}).

Let $\Q$ be a free $\Z$-module with a basis indexed by $\Ze$ that we denote by $\{\alpha_i\}_{i \in \Ze}$. We have $\Q = \oplus_{i \in \Ze} \Z\alpha_i$ and we define $\Qpos \coloneqq \oplus_{i \in \Ze} \mathbb{N}\alpha_i$. We will also use the element $\constQ \in \Qpos$ defined by
\[
\constQ = \begin{cases}
\sum_{i \in \Ze} \alpha_i,&\text{if } e > 0,
\\
0, &\text{if } e = 0.
\end{cases}
\]
 For any $\alpha \in \Q$, we denote by $\lvert\alpha\rvert \in \Z$ the sum of its coordinates in the basis $\{\alpha_i\}_{i \in \Ze}$. For any partition $\lambda$, we define
\[
\alpha^\charge(\lambda) \coloneqq \sum_{\gamma \in \mathcal{Y}(\lambda)} \alpha_{\res{\charge}(\gamma)} = \sum_{i \in \Ze} \ci{\charge}(\lambda) \alpha_i \in \Qpos.
\]
Note that $\lvert \alpha^\charge(\lambda)\rvert = \lvert \lambda\rvert$. We will say that the partition $\lambda$ \emph{lies} in $\alpha \in \Q$ if $\alpha^\charge(\lambda) = \alpha$. 
Lemma~\ref{lemma:add_rim_hook_residues} can be reformulated as follows.

\begin{lemma}
\label{lemma:block_add_rim_hook_partition}
Let $\lambda$ and $\mu$ be two partitions such that $\mu$ is obtained from $\lambda$ by adding an $e$-rim hook. Then $\alpha^\charge(\mu) = \alpha^\charge(\lambda)+\constQ$.
\end{lemma}

We extend the definition of the $\charge$-weight to $\Q$ by setting
\[
\wc(\alpha) \coloneqq \cialpha[\charge] - \frac{1}{2}\sum_{i \in \Ze} \left(\cialpha - \cialpha[i+1]\right)^2 \in \Z,
\]
for any $\alpha = \sum_{i \in \Ze} \cialpha \alpha_i \in \Q$. Note that $\wc(\alpha) \in \Z$ indeed, since $\sum_{i \in \Ze}\bigl(\cialpha-\cialpha[i+1]\bigr) = 0$ and thus $\sum_{i \in \Ze}\bigl(\cialpha-\cialpha[i+1]\bigr)^2$ is even. Moreover, we have
\begin{equation}
\label{equation:wclambda=wcalphalambda}
\wc(\lambda) = \wc\bigl(\alpha^\charge(\lambda)\bigr),
\end{equation}
 and if $e > 0$ then
\begin{equation}
\label{equation:wcalpha+h}
\wc(\alpha + h\constQ) = \wc(\alpha) + h,
\end{equation}
for any  $h \in \Z$. Finally, we define
\[
\Qcharge \coloneqq \left\{ \alpha^\charge(\lambda) : \lambda \text{ partition of } n \text{ for some } n \in \mathbb{N}\right\}.
\]
In~\textsection\ref{subsection:abaci} we will give a parametrisation of $\Qcharge$ using abaci, while in~\textsection\ref{subsection:blocks_partitions} we will give an implicit description of $\Qcharge$, using a superlevel set for the $\charge$-weight function.

\subsection{Abaci}
\label{subsection:abaci}

Let $\lambda$ be a partition.
The \emph{charged beta-number} associated with the partition $\lambda$ and the charge~$\charge$ is the sequence $\bnumber{\charge}(\lambda) = \bigl(\bnumber{\charge}(\lambda)_a\bigr)_{a \geq 1}$ defined by
\[
\bnumber{\charge}(\lambda)_a \coloneqq \charge + \lambda_a - a \in \Z,
\]
for any $a \geq 1$. It is a (strictly) decreasing sequence of integers. 
For $a > h(\lambda)$ we have $\bnumber{\charge}(\lambda)_a = \charge - a$, and conversely if $\beta = (\beta_a)_{a \geq 1}$ is a decreasing sequence of integers such that $\beta_a = \charge-a$ for $a \gg 0$ then $\beta = \bnumber{\charge}(\lambda)$ for some partition $\lambda$. Representing $\bnumber{\charge}(\lambda)$ by a copy of~$\Z$ where we put a bead on position $i \in \Z$ if $i \in \bnumber{\charge}(\lambda)$ and a gap otherwise gives the \emph{$\charge$-abacus} associated with $\lambda$.

\begin{example}
The $\chargez$-beta-number associated with the empty partition is $(-1,-2,-3,\dots)$ and the associated charged abacus is \cabacus{0}{0}{.}
\end{example}

\begin{example}
\label{example:abacus}
The $\chargeinteger{1}$-beta-number associated with the partition $(4,3,3,1)$ is $(4, 2, 1, -2,-4,-5,\dots)$ and the associated charged abacus is \cabacus{1}{4,3,3,1}{.}
\end{example}

 The next result is well-known, see for instance~\cite[Proposition (1.8)]{olsson}.

\begin{lemma}
\label{lemma:removal_hook_beta_number}
Let $l \in \N$. The partition $\lambda$ has an $l$-rim hook if and only if there is an element $b \in \beta(\lambda)$ such that $b-l\notin\beta(\lambda)$. In that case, if $\mu$ is the partition that we obtain by removing this $l$-rim hook, then $\beta(\mu)$ is obtained by replacing $b$ by $b-l$ in $\beta(\lambda)$ and then sorting in decreasing order.
\end{lemma}

Unless mentioned otherwise, we now assume that $e > 0$. 
For each $i \in \{0, \dots, e-1\}$
 we can define $\bfayers_i^{\charge}(\lambda) \in \mathbb{Z}$ to be the largest element of $\bnumber{\charge}(\lambda)$ congruent to $i$ modulo $e$.
 By~\cite[Lemma 3.2]{fayers:core}, if $\lambda$ is an $e$-core we have the following relation:
\begin{equation}
\label{equation:sumb_charge}
\charge = \frac{e+1}{2} + \frac{1}{e} \sum_{i \in \Z/e\Z} \bfayers^\charge_i(\lambda).
\end{equation}
Setting $\xrostam^\charge_i(\lambda) \coloneqq \frac{1}{e} \bigl(\bfayers^\charge_i(\lambda) - i\bigr) + 1 \in \Z$ for each $i \in \{0, \dots, e-1\}$, the relation~\eqref{equation:sumb_charge} becomes
\begin{equation}
\label{equation:sumx_charge}
\charge = \sum_{i =0}^{e-1} \xrostam^\charge_i(\lambda).
\end{equation}

\begin{proposition}
\label{proposition:1-1_correspondance_cores_abaci_x}
Let $\xrostam_0, \dots, \xrostam_{e-1} \in \Z$. There exists an $e$-core $\lambda$ such that $\xrostam_i = \xrostam_i^\charge(\lambda)$ for all $i \in \{0, \dots, e-1\}$ if and only if $\xrostam_0 + \dots + \xrostam_{e-1} = \charge$.
\end{proposition}

\begin{proof}
The direct implication follows from~\eqref{equation:sumx_charge}. Now assume that $\xrostam_0 + \dots + \xrostam_{e-1} = \charge$ and define $\bfayers_i \coloneqq e (\xrostam_i-1)+i$ for all $i \in \{0, \dots, e-1\}$. The (unique) decreasing sequence of integers $\beta = (\beta_a)_{a \geq 1}$  defined by:
\begin{itemize}
\item for all $i \in \{1, \dots,e-1\}$, the integer $\bfayers_i$ is the largest element of $\beta$ congruent to $i$ modulo~$e$;
\item for all $h \in \beta$ we have $h-e\in \beta$;
\end{itemize}
is the charged beta-number associated with some partition $\lambda$ and charge $\charge' \in \Z$. By Lemma~\ref{lemma:removal_hook_beta_number} the partition $\lambda$ is an $e$-core and by~\eqref{equation:sumx_charge} we have $\charge'=\charge$. This concludes the proof.
\end{proof}

The abacus representation of a partition that we now use has been first introduced by James~\cite{james:some}; we use here the setting of~\cite{lyle-mathas:blocks}.
We dispose the elements of $\bnumber{\charge}(\lambda)$ on an abacus with now $e$ runners, the \emph{charged $e$-abacus of $\lambda$}, each runner being a horizontal copy of $\Z$ and displayed in the following way: the $0$-th runner is on the bottom and the origins of each copy of $\Z$ are aligned with respect to a vertical line. We record the elements of $\bnumber{\charge}(\lambda)$ on this abacus according to the following rule: there is a bead at position $j \in \Z$ on the runner $i \in \{0, \dots, e-1\}$ if and only if there exists $a \geq 1$ such that $\bnumber{\charge}(\lambda)_a = i + je$. For each $i \in \{0,\dots,e-1\}$, the $i$-th runner corresponds to the charged abacus (as defined at the beginning of~\textsection\ref{subsection:abaci}) of a certain partition $\lambda^{[i]}$. The $e$-tuple $\bigl(\lambda^{[0]},\dots,\lambda^{[e-1]}\bigr)$ is the \emph{$e$-quotient} of $\lambda$. In particular, any partition is uniquely determined by its $e$-core and its $e$-quotient.

\begin{remark}
The $e$-quotient of $\lambda$ depends on the charge $\charge$ only up to a shift: if $\bigl(\lambda^{[0]},\dots,\lambda^{[e-1]}\bigr)$ is the $e$-quotient of $\lambda$ computed on the $\charge$-charged $e$-abacus then $\bigl(\lambda^{[e-1]},\lambda^{[0]},\dots,\lambda^{[e-2]}\bigr)$ is the $e$-quotient computed on the $(\charge+1)$-charged $e$-abacus. 
\end{remark}

\begin{example}
\label{example:e_abacus}
We take the charge $\charge=\chargeinteger{1}$ and we consider the partition $\lambda \coloneqq (4,3,3,1)$ as in Example~\ref{example:abacus}. The associated charged $3$-abacus is
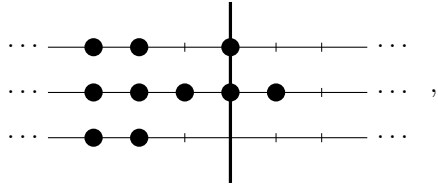
\ceabacus{3}{4,3,3,1}{,}{1}
and the associated charged $6$-abacus is
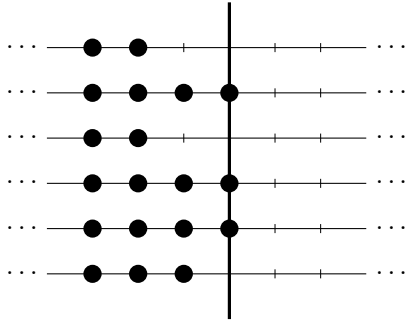
\ceabacus{6}{4,3,3,1}{.}{1}
Counting the number of gaps down each bead (continuing counting on the left starting from the top runner when reaching the bottom one) recovers the underlying partition. The $3$-quotient of $\lambda$ is $\bigl(\emptyset,\emptyset,(1)\bigr)$ and its $6$-quotient has only empty partitions.
\end{example}

The $e$-abacus representation is particularly adapted to the addition and deletion of rim hooks: we give two particular cases, as corollaries of Lemma~\ref{lemma:removal_hook_beta_number}.

\begin{lemma}
\label{lemma:remove_node_abacus}
Let $\lambda$ be a partition and $i \in \{0, \dots, e-1\}$.
\begin{itemize}
\item The partition $\lambda$ has a removable $i$-node if and only if we can move a bead from runner~$i$ to runner $i-1$ (to runner $e-1$ if $i = 0$), keeping the same position $j \in \Z$ (from position~$j$ to $j-1$ if $i = 0$).
\item The partition $\lambda$ has an addable $i$-node if and only if we can move a bead from runner~$i$ to runner $i+1$ (to runner $0$ if $i = e-1$), keeping the same position $j \in \Z$ (from position~$j$ to $j+1$ if $i = e-1$).
\end{itemize}
\end{lemma}

\begin{lemma}
\label{lemma:remove_hook_abacus}
Let $\lambda$ be a partition.
\begin{itemize}
\item
The partition $\lambda$ has an $e$-rim hook if and only if on some runner we can slide a bead at some position $j \in \Z$ to the previously free position $j-1$. Hence, the partition $\lambda$ is an $e$-core if and only if its associated $e$-abacus has no gaps, that is, there are no gaps between two beads on a same runner.
\item The partition $\lambda$ has an addable $e$-rim hook if and only if on some runner we can slide a bead at some position $j \in \Z$ to the previously free position $j+1$. In particular, to any partition we can add at least $e$ different $e$-rim hooks.
\end{itemize}
\end{lemma}

\begin{remark}
\label{remark:lozenge}
Let $\lambda$ be an $e$-core. For any $i \in \{0,\dots,e-1\}$, the integer $\xrostam^\charge_i(\lambda)$ corresponds to the position of the first gap on runner $i$ (pictured by \scalebox{\scalelozenge}{$\color{red}{\blacklozenge}$} in the abacus below). For instance, by Lemma~\ref{lemma:remove_hook_abacus} the partition $\lambda = (4,3,3,1)$ of Example~\ref{example:e_abacus} is a $6$-core and we have $\xrostam^\charge(\lambda) = (0, 1, 1, -1,1,-1)$. In particular, note that~\eqref{equation:sumx_charge} is satisfied indeed. \abacuscore{6}{4,3,3,1}{}{1}
\end{remark}

Note that Lemma~\ref{lemma:remove_hook_abacus} implies
\begin{equation}
\label{equation:eweight_sum_length_equotient}
\we(\lambda) = \sum_{i = 0}^{e-1} \left\lvert \lambda^{[i]}\right\rvert.
\end{equation}
The next result follows from Lemma~\ref{lemma:remove_node_abacus} (see, for instance, \cite[Lemma 2.11]{rostam:stuttering}).

\begin{proposition}
\label{proposition:xi_ci_ci+1}
Assume that $\lambda$ is an $e$-core.
For any $i \in \{0,\dots,e-1\}$ we have:
\[
\xrostam^\charge_i(\lambda) =  \ci{\charge}(\lambda) - \ci[i+1]{\charge}(\lambda) + \xrostam^\charge_i(\emptyset).
\]
\end{proposition}

Note that we can compute the value of $\xrostam^\charge_i(\emptyset)$.

\begin{lemma}
\label{lemma:x_vide}
Write $\charge = \qcharge e + \charge'$, where $\qcharge \in \Z$ and $\charge'\in \{0, \dots, e-1\}$.
For any $i \in \{0, \dots, e-1\}$ we have:
\[
\xrostam^\charge_i(\emptyset) = \begin{cases}
\qcharge + 1, &\text{if } i \in \{0, \dots, \charge' - 1\},
\\
\qcharge, &\text{otherwise}.
\end{cases}
\]
\end{lemma}

\begin{proof}
We have $\bnumber{\charge}(\emptyset)_a = \charge - a$ for all $a \geq 1$. Thus, for any $a \in \{1, \dots, e\}$, if $i \in \{0, \dots, e-1\}$ is the residue modulo $e$ of  $\charge - a$ then $\bfayers^\charge_{i}(\emptyset) = \charge - a$. Thus, if $a$ describes $\{1, \dots, \charge'\}$ then $i \coloneqq \charge'-a$ describes $\{0,\dots,\charge'-1\}$, moreover $i = \charge-a \pmod e$ thus  $\bfayers^{\charge}_i(\emptyset) = \charge-a= \qcharge e + i$ and $\xrostam^\charge_i(\emptyset) = \frac{1}{e}\bigl(\bfayers^{\charge}_i(\emptyset) -i\bigr)+1= \qcharge + 1$. Similarly, if $a \in \{\charge' + 1, \dots, e\}$ then with $i \coloneqq \charge' - a + e \in \{\charge',\dots,e-1\}$ we have $i = \charge-a\pmod{e}$ thus $\bfayers^{\charge}_i(\emptyset) = \qcharge e + i - e$ and $\xrostam^\charge_i(\emptyset) =\qcharge$.
\end{proof}

For any $i \in \Ze$, define
\[
\xrostamod^{\charge}_i(\lambda) \coloneqq \xrostam^\charge_i(\lambda) - \xrostam^\charge_i(\emptyset).
\]
By Proposition~\ref{proposition:xi_ci_ci+1} (and~\eqref{equation:sumx_charge}), for an $e$-core $\lambda$ we have
\begin{align}
\xrostamod^\charge_i(\lambda) &= \ci{\charge}(\lambda) - \ci[i+1]{\charge}(\lambda), &&\text{for all } i \in \Ze,
\label{equation:yi_ci_ci+1}
\\
\sum_{i \in \Ze} \xrostamod^\charge_i(\lambda) &= 0.
\label{equation:sumy_0}
\end{align}

As for the $e$-quotient, the family $\xrostamod^\charge(\lambda) = \bigl(\xrostamod^\charge_i(\lambda)\bigr)_{i \in \Ze}$ depends on the charge $\charge$ only up to a shift (by~\eqref{equation:charges_charge0} and~\eqref{equation:yi_ci_ci+1}).
The following proposition is a direct consequence of~\eqref{equation:sumx_charge} and Proposition~\ref{proposition:1-1_correspondance_cores_abaci_x}.

\begin{proposition}
\label{proposition:1-1_correspondance_cores_abaci_y}
Let $\charge \in \Z$ and $\xrostamod_0, \dots, \xrostamod_{e-1} \in \Z$. There exists an $e$-core $\lambda$ such that $\xrostamod_i = \xrostamod_i^\charge(\lambda)$ for all $i \in \{0, \dots, e-1\}$ if and only if $\xrostamod_0 + \dots + \xrostamod_{e-1} = 0$.
\end{proposition}


For completeness, we also give the version of Proposition~\ref{proposition:1-1_correspondance_cores_abaci_y} in the case $e = 0$. We recall that every partition $\lambda$ is a $0$-core, and we define the integers $\xrostamod^\charge_i(\lambda)$ for any $i \in \Z$ using~\eqref{equation:yi_ci_ci+1}. By Remark~\ref{remark:ci} the family $\bigl(\xrostamod^\charge_i(\lambda)\bigr)_{i \in \Z}$ has finite support, equality~\eqref{equation:sumy_0} is still satisfied and by~\eqref{equation:yi_ci_ci+1} we have
\begin{equation}
\label{equation:einfinite_x01}
\xrostamod^\charge_i(\lambda) \in \begin{cases}
\{0, 1\}, &\text{if } i \geq \charge,
\\
\{0, -1\},&\text{if } i < \charge.
\end{cases}
\end{equation}

\begin{proposition}[\protect{\cite[(3.9)]{robinson-thrall}}]
\label{proposition:1-1_correspondance_cores_abaci_einfinite}
Let $(\xrostamod_i)_{i \in \Z} \in \Z^{\Z}$ with finite support. There exists a partition $\lambda$ such that $\xrostamod_i = \xrostamod^\charge_i(\lambda)$ for all $i \in \Z$ if and only if
\begin{align}
\xrostamod_i &\in \begin{cases}
\{0, 1\}, &\text{if } i \geq \charge,
\\
\{0, -1\},&\text{if } i < \charge,
\end{cases}
\label{equation:proof_cond_xrostamod_einfinite}
\\
\intertext{for all $i \in \Z$, and}
\sum_{i \in \Z} \xrostamod_i &= 0.
\notag
\end{align}
\end{proposition}
 
\begin{proof}
We have just seen that these conditions are necessary. Now let $(\xrostamod_i)_{i \in \Z} \in \Z^{\Z}$ with finite support satisfy the conditions of Proposition~\ref{proposition:1-1_correspondance_cores_abaci_einfinite} and let $N > 0$ such that $\xrostamod_i = 0$ if $\lvert i \rvert \geq N$. We define a family $(c_i)_{i \in \Z}$ by setting $c_N \coloneqq 0$ and
\[
c_i \coloneqq \begin{cases}
c_{i+1} +\xrostamod_i, &\text{if } i < N,
\\
c_{i-1} - \xrostamod_{i-1}, &\text{if } i > N.
\end{cases}
\]
By assumption we have $c_i = 0$ if $\lvert i \rvert \geq N$, moreover $c_i - c_{i+1} = \xrostamod_i$ for all $i \in \Z$. By~\eqref{equation:proof_cond_xrostamod_einfinite}, it is clear that we can construct a partition $\lambda$ such that $c_i = \ci{\charge}(\lambda)$ for all $i \in \Z$. 
\end{proof}

The following result shows that the $e$-weight~$\we$ (defined in Definition~\ref{definition:e-weight})  and the $\charge$-weight~$\wc$ (defined in~\eqref{equation:def_wc_partitions}) of a partition coincide.

\begin{proposition}[\protect{\cite[Proposition 2.1]{fayers:weights}}]
\label{proposition:we=wc}
Let $e \in \N$.
We have $\we(\lambda) = \wc(\lambda)$.
\end{proposition}

We are thus able to recover the integers $\ci{\charge}(\lambda)$ from $\xrostamod^\charge(\lambda)$.

\begin{corollary}
\label{corollary:expression_ci_in_terms_of_yi}
Let $e \in \N$.
\begin{itemize}
\item 
Assume that $e > 0$.
If $\lambda$ is an $e$-core, then
\begin{align*}
\ci[\charge]{\charge}(\lambda) &= \frac{1}{2}\lVert \xrostamod^\charge(\lambda)\rVert^2 =  \frac12 \sum_{i = 0}^{e-1}  \xrostamod^\charge_i(\lambda)^2,
\\
\ci[\charge+i]{\charge}(\lambda) &= \frac{1}{2}\lVert \xrostamod^\charge(\lambda)\rVert^2 - \xrostamod^\charge_{\charge}(\lambda) - \dots - \xrostamod^\charge_{\charge+i-1}(\lambda),
\end{align*}
for all $i \in \{1, \dots, e-1\}$.

\item Assume that $e = 0$. We have
\begin{align*}
\ci[\charge]{\charge}(\lambda) &= \frac{1}{2}\lVert \xrostamod^\charge(\lambda)\rVert^2 =  \frac12 \sum_{i \in \Z}  \xrostamod^\charge_i(\lambda)^2,
\\
\ci[\charge+i]{\charge}(\lambda) &= \frac{1}{2}\lVert \xrostamod^\charge(\lambda)\rVert^2 - \xrostamod^\charge_{\charge}(\lambda) - \dots - \xrostamod^\charge_{\charge+i-1}(\lambda),
\\
\ci[\charge-i]{\charge}(\lambda) &= \frac{1}{2}\lVert \xrostamod^\charge(\lambda)\rVert^2 + \xrostamod^\charge_{\charge-1}(\lambda) + \dots + \xrostamod^\charge_{\charge-i+1}(\lambda),
\end{align*}
for all $i \in \Z_{>0}$.
\end{itemize}
\end{corollary}

\subsection{Blocks for partitions}
\label{subsection:blocks_partitions}

Let $e \in \N$. We first recall some standard facts concerning blocks associated with partitions. 

\begin{lemma}
\label{lemma:nakayamas_conjecture}
If $\lambda$ and $\mu$ are two $e$-cores with $\alpha^\charge(\lambda) = \alpha^\charge(\mu)$ then $\lambda = \mu$.
\end{lemma}

\begin{proof}
The statement is clear if $e = 0$. If $e > 0$, by~\eqref{equation:yi_ci_ci+1} we have $\xrostamod^\charge(\lambda) = \xrostamod^\charge(\mu)$. We thus conclude that $\xrostam^\charge_i(\lambda) = \xrostam^\charge_i(\mu)$ and thus $\bfayers^\charge_i(\lambda)=\bfayers_i^\charge(\mu)$ for all $i \in \{0,\dots,e-1\}$ whence $\lambda = \mu$.
\end{proof}

Recall from the introduction that if $\lambda$ and $\mu$ are two partitions, we say that $\lambda$ and $\mu$ \emph{lie in the same block} if $\alpha^\charge(\lambda) = \alpha^\charge(\mu)$. For instance, Lemma~\ref{lemma:nakayamas_conjecture} says that two different $e$-cores cannot lie in the same block.
Together with Lemma~\ref{lemma:block_add_rim_hook_partition}, \eqref{equation:wclambda=wcalphalambda} and Proposition~\ref{proposition:we=wc}, we deduce the next lemma, which can be seen as a combinatorial version of the so-called ``Nakayama's Conjecture''. 

\begin{lemma}[\protect{\cite[Theorem 2.7.41]{james-kerber}}]
\label{lemma:block_common_core}
Two partitions that lie in the same block share the same $e$-core.
\end{lemma}

In particular, the elements $\alpha^\charge(\lambda)$ encode the blocks of the associated \emph{Iwahori--Hecke algebra}, see for instance~\cite[5.38 Corollary]{mathas:iwahori}.
Finally, by~\eqref{equation:wclambda=wcalphalambda}, Proposition~\ref{proposition:we=wc} and Lemma~\ref{lemma:block_common_core}, we have the following result.

\begin{lemma}
\label{lemma:block_core_unique_partition}
Let $\alpha \in \Qcharge$ with $\wc(\alpha) = 0$. Then there is a unique partition lying in $\alpha$, and this partition is an $e$-core.
\end{lemma}

\begin{remark}
We will see in Lemma~\ref{lemma:cores_w=0} that the statement of Lemma~\ref{lemma:block_core_unique_partition} remains true if we replace $\alpha\in \Qcharge$ with $\alpha \in \Q$.
\end{remark}

We can now give a $1{:}1$-parametrisation of the set $\Qcharge = \{\alpha^\charge(\lambda) : \lambda$ partition of $n$ for some $n \in \mathbb{N}\}$ by $\Z^{e-1} \times \mathbb{N}$, if $e > 0$.
By Lemmas~\ref{lemma:nakayamas_conjecture} and~\ref{lemma:block_core_unique_partition}, for any $\alpha \in \Qcharge$ we can denote by~$\lambda_\alpha$ the (unique) common $e$-core of the partitions lying in $\alpha$.

\begin{proposition}
\label{proposition:parametrisation_Qcharge}
Assume that $e > 0$. The map
\[
\Qcharge \longrightarrow \bigl\{ \xrostamod = (\xrostamod_0, \dots, \xrostamod_{e-1}) \in \Z^e : \xrostamod_0 + \dots + \xrostamod_{e-1} = 0\bigr\} \times \mathbb{N},
\]
given by $\alpha \mapsto \bigl(\xrostamod^\charge(\lambda_\alpha), \wc(\alpha)\bigr)$, is a bijection. Its inverse is given by 
\[
(x, w) \longmapsto \left(\frac{1}{2} \lVert \xrostamod\rVert^2 + w\right)\sum_{i = 0}^{e-1} \alpha_i - \sum_{i = 1}^{e-1} (\xrostamod_{\charge} + \dots + \xrostamod_{\charge + i-1})\alpha_{\charge + i},
\]
where the indices of $\xrostamod = (\xrostamod_0,\dots,\xrostamod_{e-1})$ are taken modulo $e$.
\end{proposition}

\begin{remark}
Using Proposition~\ref{proposition:1-1_correspondance_cores_abaci_einfinite} and Corollary~\ref{corollary:expression_ci_in_terms_of_yi}, we can also give a version of Proposition~\ref{proposition:parametrisation_Qcharge} in the case $e = 0$, using the map $\alpha \mapsto \xrostamod^\charge(\lambda_\alpha)$. Note that in this case, an element  $\alpha = \sum_{i \in\Z}\cialpha\alpha_i \in Q$ is in $\Qcharge$ if and only if
\[
\cialpha-\cialpha[i+1] \in  \begin{cases}
\{0, 1\}, &\text{if } i \geq \charge,
\\
\{0, -1\},&\text{if } i < \charge.
\end{cases}
\]
\end{remark}

We now want to give an implicit description of $\Qcharge$. Note that in the following results we consider the $\charge$-weight inside $Q$ and not $\Qpos$.

\begin{lemma}
\label{lemma:cores_w=0}
We have $\bigl\{ \alpha^{\charge}(\lambda) : \lambda \text{ is an } e\text{-core} \bigr\} = \{\alpha \in Q : \wc(\alpha) = 0\}$.
\end{lemma}

\begin{proof}
We first assume that $e > 0$.
The inclusion $\subseteq$ follows from~\eqref{equation:wclambda=wcalphalambda} and Proposition~\ref{proposition:we=wc}. Now   let $\alpha = \sum_{i \in \Ze} \cialpha \alpha_i \in Q$ such that $\wc(\alpha) = 0$ and define $\xrostamod_i \coloneqq \cialpha - \cialpha[i+1]$ for all $i \in \Ze$. By Proposition~\ref{proposition:1-1_correspondance_cores_abaci_y} we know that $\xrostamod_i = \xrostamod^\charge_i(\lambda)$ for some $e$-core $\lambda$. By Corollary~\ref{corollary:expression_ci_in_terms_of_yi} and since $\wc(\alpha) = 0$ we have
\begin{align*}
\ci[\charge]{\charge}(\lambda)
&=
\frac{1}{2}\sum_{i \in \Ze} \xrostamod^\charge_i(\lambda)^2
\\
&=
\frac{1}{2}\sum_{i \in \Ze} \xrostamod_i^2
\\
&=
\cialpha[\charge],
\end{align*}
and
\begin{align*}
\ci[\charge + i]{\charge}(\lambda)
&=
\ci[\charge]{\charge}(\lambda) - \xrostamod^\charge_\charge(\lambda) - \dots - \xrostamod^\charge_{\charge+i-1}(\lambda)
\\
&=
\ci[\charge]{\charge} - \xrostamod^\charge_\charge - \dots - \xrostamod^\charge_{\charge+i-1}
\\
&=
\cialpha[\charge+i],
\end{align*}
thus $\alpha^\charge(\lambda) = \alpha$.

We now assume that $e = 0$. Again $\subseteq$ follows from~\eqref{equation:wclambda=wcalphalambda} and Proposition~\ref{proposition:we=wc}, thus let $\alpha = \sum_{i \in \Z}\cialpha\alpha_i \in Q$ such that $\wc(\alpha) = 0$. We can repeat the proof of the case $e >0$, using Proposition~\ref{proposition:1-1_correspondance_cores_abaci_einfinite} instead of Proposition~\ref{proposition:1-1_correspondance_cores_abaci_y}. However, we have to ensure that $\cialpha - \cialpha[i+1] \in \begin{cases}
\{0,1\}, &\text{if } i \geq \charge,
\\
\{0,-1\},&\text{if } i < \charge.
\end{cases}$ Since $\wc(\alpha) = 0$ we have
\begin{align*}
2\cialpha[\charge]
&=
\sum_{i \in \Z} (\cialpha - \cialpha[i+1])^2
\\
&=
\sum_{i \geq \charge} (\cialpha-\cialpha[i+1])^2
+
\sum_{i < \charge} (\cialpha-\cialpha[i+1])^2
\\
&\geq
\sum_{i \geq \charge} (\cialpha-\cialpha[i+1])
+
\sum_{i < \charge} (\cialpha[i+1]-\cialpha)
\\
&=
\cialpha[\charge] + \cialpha[\charge],
\end{align*}
thus we deduce that
\[
(\cialpha-\cialpha[i+1])^2 = \begin{cases}
\cialpha-\cialpha[i+1], &\text{if } i \geq \charge,
\\
\cialpha[i+1] - \cialpha,&\text{if } i < \charge,
\end{cases}
\]
for all $i \in \Z$, thus we obtain the desired result. Another way to see the result for $e = 0$ is to consider $e >0$ large enough so that we can apply the previous case (using Remark~\ref{remark:ci}).
\end{proof}

\begin{proposition}
\label{proposition:partitions_wgeq0}
We have $\Qcharge = \{\alpha \in Q : \wc(\alpha) \geq 0\}$.
\end{proposition}

\begin{proof}
Assume first that $e > 0$.
Again $\subseteq$ follows from~\eqref{equation:wclambda=wcalphalambda} and Proposition~\ref{proposition:we=wc}. Let $\alpha = \sum_{i \in \Ze} \cialpha\alpha_i \in Q$ such that $\wc(\alpha) \geq 0$. By~\eqref{equation:wcalpha+h}, setting $\hat\alpha \coloneqq \alpha - \wc(\alpha)\constQ \in Q$ we have $\wc(\hat\alpha) = 0$ thus by Lemma~\ref{lemma:cores_w=0} we have $\hat\alpha = \alpha^\charge(\lambda)$ for some $e$-core $\lambda$. Now if $\mu$ is any partition obtained from $\lambda$ by adding $\wc(\alpha)$ times an $e$-rim hook we have $\alpha^\charge(\mu) = \alpha^\charge(\lambda) + \wc(\alpha)\constQ =  \hat\alpha + \wc(\alpha)\constQ  = \alpha$ which concludes the proof in the case $e > 0$.

Now if $e = 0$, if $\alpha = \sum_{i \in \Z}\cialpha\alpha_i \in Q$ is such that $\wc(\alpha) \geq 0$ then as in the proof of Lemma~\ref{lemma:cores_w=0} we have
\[
2\cialpha[\charge] \geq \sum_{i \in \Z} (\cialpha-\cialpha[i+1])^2 \geq \sum_{i \geq \charge}(\cialpha-\cialpha[i+1]) + \sum_{i < \charge}(\cialpha[i+1]-\cialpha) = 2\cialpha[\charge],
\]
thus we have in fact $\wc(\alpha) = 0$ and we apply Lemma~\ref{lemma:cores_w=0}.
\end{proof}

An interesting consequence of Proposition~\ref{proposition:partitions_wgeq0} is the following result, which was, up to the author's knowledge, not already known. It allows to determine whether a given multiset of residues actually comes from a Young diagram and non-trivially extends a result of Robinson--Thrall~\cite{robinson-thrall} for $e  = 0$ to $e \geq 0$.

\begin{corollary}
\label{corollary:existence_young_diagram_given_residues}
Let $e \in \N$ and let $(\cialpha)_{i \in \Ze}$ be a sequence of integers indexed by $\Ze$ (with finite support if $e = 0$). There exists a partition $\lambda$ satisfying $\ci\charge(\lambda) = \cialpha$ for all $i \in \Ze$ if and only if $\sum_{i \in \Ze}  \left(\cialpha-\cialpha[i+1]\right)^2 \leq 2 \cialpha[\charge]$.
\end{corollary}

\begin{remark}
If a partition as in Corollary~\ref{corollary:existence_young_diagram_given_residues} exists, Lemma~\ref{lemma:cores_w=0} and the proof of Proposition~\ref{proposition:partitions_wgeq0} show how to compute the associated core, and thus all the partitions with the same multiset of residues (that is, lying in the same block).
\end{remark}

The aim of Section~\ref{section:implicit_description_blocks} is to study an analogue of Proposition~\ref{proposition:partitions_wgeq0} in higher levels.
By~\eqref{equation:wcalpha+h} we obtain the following corollary of Proposition~\ref{proposition:partitions_wgeq0}.

\begin{corollary}
\label{corollary:partition_any_alpha+h_block}
Assume $e > 0$. 
For any $\alpha \in \Q$ and any $h \in \mathbb{Z}$ we have:
\[
\alpha + h\constQ \in \Qcharge \iff h  \geq -\wc(\alpha).
\]
\end{corollary}

\begin{example}
\label{example:alpha=halpha0}
Assume $e \geq 2$ and take $\charge=\chargez$. Let $n \in \mathbb{N}$ and define $\alpha \coloneqq n\alpha_0 \in \Qpos$. We have $\wc[\chargez](\alpha) = n-n^2 \leq 0$. If $n \in \{0,1\}$ then $\wc[\chargez](\alpha) = 0$ thus $\alpha \in \Qcharge$, and indeed if $n = 0$ then $\alpha = 0 = \alpha^{\chargez}(\emptyset)$ and if $n = 1$ then  $\alpha = \alpha_0 = \alpha^{\chargez}\bigl((1)\bigr)$.

Thus, we now assume $n \geq 2$, in particular $\wc[\chargez](\alpha) = n - n^2 < 0$. We have 
\[
\wc[\chargez]\bigl(\alpha + (n^2-n)\constQ\bigr) = 0.
\]
Thus, by Lemma~\ref{lemma:cores_w=0} there is an $e$-core $\lambda$ such that
\[
\alpha^{\chargez}(\lambda) = \alpha+(n^2-n)\delta =  n^2\alpha_0 + (n^2-n)\sum_{i = 1}^{e-1}\alpha_i,
\]
in other words $\ci[0]{\chargez}(\lambda) = n^2$ and $\ci{\chargez}(\lambda) = n^2-n$ for all $i \in \{1, \dots, e-1\}$. In particular, we have:
\[
|\lambda| = n^2 + (e-1)(n^2-n) = en^2 - (e-1)n.
\]
 By~\eqref{equation:yi_ci_ci+1}, the $e$-core $\lambda$ satisfies $\xrostamod^\chargez_0(\lambda) = -\xrostamod^\chargez_{e-1}(\lambda) = n$, with the other integers $\xrostamod^\chargez_i(\lambda)$ being $0$. For instance, with $e = 4$ and $n = 2$ then by Remark~\ref{remark:lozenge} the charged $4$-abacus of $\lambda$ is:
\abacuscore{4}{5,2,1,1,1}{,}{0}
thus $\lambda = (5,2,1,1,1)$, with Young diagram (with residues) \ytableaushort{01230,30,2,1,0}\, thus $\ci[0]{\chargez}(\lambda) = 4$ and $\ci{\chargez}(\lambda) = 2$ for all $i \in \{1,2,3\}$ indeed.
\end{example}

\begin{example}
Assume $e \geq 2$ and take $\charge=\chargez$. Let $n \in \mathbb{N}$, fix $i_0 \in \{1, \dots, e-1\}$ and define $\alpha \coloneqq n\alpha_{i_0} \in \Qpos$. We have $\wc[\chargez](\alpha) = -n^2 \leq 0$. If $n = 0$ then $\alpha = 0 = \alpha^{\chargez}(\emptyset)$, thus we now assume that $n \geq 1$. We have $\wc[\chargez](\alpha) =  - n^2 < 0$ and
\[
\wc[\chargez]\bigl(\alpha + n^2\constQ\bigr) = 0,
\]
thus there is an $e$-core $\lambda$ such that
\[
\alpha^{\chargez}(\lambda) = \bigl(n^2+n\bigr)\alpha_{i_0} + n^2\sum_{\substack{i \in \Ze \\ i \neq i_0}}\alpha_i,
\]
in other words $\ci[i_0]{\chargez}(\lambda) = n^2+n$ and $\ci{\chargez}(\lambda) = n^2$ for all $i \neq i_0$. In particular, we have:
\[
|\lambda| = n^2+n + (e-1)n^2 = en^2 + n.
\]
 The $e$-core $\lambda$ satisfies $\xrostamod^\chargez_{i_0}(\lambda) = -\xrostamod^\chargez_{i_0-1}(\lambda) = n$, with the other integers $\xrostamod^\chargez_i(\lambda)$ being $0$. Note that for (the excluded case) $i_0 = 0$ we obtain the same $e$-tuple $\xrostamod^\chargez(\lambda)$ as in Example~\ref{example:alpha=halpha0}, although the value of $|\lambda|$ is different. For  $e = 4$ and $n = 2$ and $i_0 = 1$, the charged $4$-abacus of $\lambda$ is:
\abacuscore{4}{6,3,2,2,2,1,1,1}{,}{0}
thus $\lambda = (6,3,2,2,2,1,1,1)$, with Young diagram (with residues) \ytableaushort{012301,301,23,12,01,3,2,1}\, thus $\ci[1]{\chargez}(\lambda) = 6$ and $\ci{\chargez}(\lambda) =4$ for all $i \neq 1$ indeed.
\end{example}

\subsection{Multipartitions}
\label{subsection:multipartitions}

Let $r \geq 1$ and $e \in \N$. Let $\blam$ be an \emph{$r$-partition} (or \emph{multipartition}) of $n$, that is, an $r$-tuple $\blam = \bigl(\lambda^{(1)},\dots,\lambda^{(r)}\bigr)$ of partitions such that $\bigl\lvert\lambda^{(1)}\bigr\rvert + \dots + \bigl\lvert\lambda^{(r)}\bigr\rvert = n$. 
The Young diagram of the $r$-partition $\blam$ is the subset of $\mnodes$ defined by
\[
\mathcal{Y}(\blam) \coloneqq \bigcup_{j = 1}^r \mathcal{Y}\bigl(\lambda^{(j)}\bigr) \times \{j\}.
\]
An element of $\mathcal{Y}(\blam)$ is a node of $\blam$, more generally any element of $\mnodes$ is a node. An $e$-rim hook of $\blam$ is an $e$-rim hook of $\lambda^{(j)}$ for some $j \in \{1,\dots,r\}$. We say that $\blam$ is an \emph{$e$-multicore} if $\lambda^{(j)}$ is an $e$-core for all $j \in \{1,\dots,r\}$, and the \emph{$e$-multicore} of $\blam$ is the $r$-partition $\overline\blam$ given by $\overline\blam \coloneqq \bigl(\overline{\lambda^{(1)}},\dots,\overline{\lambda^{(r)}}\bigr)$.

Now let $\mcharge = (\charge_1,\dots,\charge_r) \in \Z^r$ be a \emph{multicharge}. The $\mcharge$-residue $\res{\mcharge}(\gamma)$ of a node $\gamma = (a,b,j)$ is the $\charge_j$-residue of the node $(a, b)$, in other words
\[
\res{\mcharge}(a, b, j) = \res{\charge_j}(a, b) = b - a + \charge_j \pmod{e}.
\]
Again, given $i \in \Ze$ an $(i, \mcharge)$-node (or simply $i$-node) is a node of $\mcharge$-residue $i$. We denote by $\ci{\mcharge}(\blam)$ the number of $(i, \mcharge)$-nodes of $\blam$. We define
\[
\alpha^{\mcharge}(\blam) \coloneqq \sum_{\gamma \in \mathcal{Y}(\blam)} \alpha_{\res{\mcharge}(\gamma)} = \sum_{i \in \Ze} \ci{\mcharge}(\blam)\alpha_i \in \Qpos.
\]


%

Following~\cite{fayers:weights}, the \emph{$\mcharge$-weight} (or simply \emph{weight}) of an $r$-partition $\blam$ is
\[
\wmc(\blam) = \sum_{j = 1}^r \ci[\charge_j]{\mcharge}(\blam) - \frac{1}{2}\sum_{i \in \Ze}\left(\ci{\mcharge}(\blam) - \ci[i+1]{\mcharge}(\blam)\right)^2.
\]
Note that we recover the corresponding definition given at~\textsection\ref{subsection:partitions} when $r = 1$.

\begin{remark} (See, for instance, \cite[\textsection 3.1]{jacon-lecouvey}\footnote{In this remark, the reference~\cite{jacon-lecouvey} refers  to the \texttt{arXiv} version, since a part of the exposition has been removed from the published version.})
This definition of $\mcharge$-weight, which can be directly computed from the Young diagram, can be interpreted in Lie-theoretic terms as follows. The Fock space $\mathcal{F}_{\mcharge} = \oplus_{\blam} \mathbb{Q}(v)\blam$, where $v$ is an indeterminate and $\blam$ runs over all the $r$-partitions, has a structure of an integrable $\mathcal{U}_v(\widehat{\mathfrak{sl}}_e)$-module. For the corresponding weight space decomposition, each $r$-partition $\blam$ is a weight vector and applying Kac's bilinear form on the corresponding weight gives a scalar $\|\blam\|^{\mcharge}$. We then have the relation $\wmc(\blam) = \|\bempty\|^{\mcharge} - \|\blam\|^{\mcharge}$ (see, for instance, \cite[Proposition 3.5]{jacon-lecouvey}).
\end{remark}

We extend the definition of the $\mcharge$-weight to $Q$ by setting
\[
\wmc(\alpha)
\coloneqq \sum_{j = 1}^r \cialpha[\charge_j] - \frac{1}{2}\sum_{i \in \Ze} \bigl(\cialpha - \cialpha[i+1]\bigr)^2 \in \Z,
\]
for any $\alpha = \sum_{i \in \Ze}\cialpha\alpha_i \in Q$. Note that $\wmc(\alpha) \in \Z$ for the same reason as in level one, and
\[
\wmc(\blam) = \wmc\bigl(\alpha^\mcharge(\blam)\bigr).
\]

%

We now recall some results from~\cite{fayers:weights}. The first one follows from Lemma~\ref{lemma:add_rim_hook_residues}.

\begin{lemma}[\protect{\cite[Corollary 3.4]{fayers:weights}}]
\label{lemma:weight_add_rim_hook}
Assume that $e > 0$ and let $\blam$ be an $r$-partition obtained from $\bmu$ by adding an $e$-rim hook. We have $\wmc(\blam) = \wmc(\bmu) + r$.
\end{lemma}

More generally, if $e > 0$ then for any $\alpha \in \Q$ and any $h \in \Z$ we have
\begin{equation}
\label{equation:wmcalpha+h}
\wmc(\alpha + h\constQ) = \wmc(\alpha) + rh.
\end{equation}
If $\blam$ is an $e$-multicore, for any $i \in \Ze$ and $j \in \{1, \dots, r\}$ we define
\begin{align*}
\bfayers^\mcharge_{ij}(\blam) &\coloneqq \bfayers^{\charge_j}_i\bigl(\lambda^{(j)}\bigr),
\\
\xrostam^\mcharge_{ij}(\blam
) &\coloneqq \xrostam^{\charge_j}_i\bigl(\lambda^{(j)}\bigr),
\\
\xrostamod^\mcharge_{ij}(\blam
) &\coloneqq \xrostamod^{\charge_j}_i\bigl(\lambda^{(j)}\bigr).
\end{align*}
Moreover, if $e > 0$ then for any $j, k\in \{1, \dots, r\}$ we define
\[
u_{j,k}(\blam) \coloneqq
- \xrostamod^\mcharge_{\charge_j, j}(\blam) - \dots - \xrostamod^\mcharge_{\charge_j + i - 1,j}(\blam),
\]
if $\charge_j \neq \charge_k$, 
where $i \geq 1$ is any integer  such that $i = \charge_k - \charge_j \pmod{e}$ (recall~\eqref{equation:sumy_0}). We define $u_{j,k}(\blam) \coloneqq 0$ if $\charge_j = \charge_k$. The point of this definition is that
\begin{equation}
\label{equation:ci_ujk}
\ci[\charge_k]{\charge_j}(\lambda^{(j)}) = \frac{1}{2}\bigl\lVert \xrostamod^{\charge_j}\bigl(\lambda^{(j)}\bigr)\bigr\rVert^2 + u_{j,k}(\blam),
\end{equation}
for all $j, k \in \{1, \dots,r\}$, by Corollary~\ref{corollary:expression_ci_in_terms_of_yi}

\begin{lemma}[\protect{\cite[Proposition 3.5]{fayers:weights}}]
\label{lemma:weight_sum_r=2}
Write $\blam = \bigl(\lambda^{(1)},\dots,\lambda^{(r)}\bigr)$. If $\blam$ is a multicore then
\[
\wmc(\blam) = \sum_{1 \leq j < k \leq r} \wmc[(\charge_j,\charge_k)]\bigl(\lambda^{(j)},\lambda^{(k)}\bigr).
\]
\end{lemma}

Note that Lemma~\ref{lemma:weight_sum_r=2} can be proved by a direct calculation using the elements $\xrostamod^\mcharge_{ij}(\blam)$ together with~\eqref{equation:yi_ci_ci+1} and~\eqref{equation:ci_ujk}. Using the same idea, we will give a proof of the next proposition. 

\begin{proposition}[\protect{\cite[Corollary 3.9]{fayers:weights}}]
\label{proposition:weight_positive}
For any $r$-partition $\blam$ we have $\wmc(\blam) \geq 0$.
\end{proposition}

\begin{proof}
By Lemma~\ref{lemma:weight_add_rim_hook} it suffices to consider the case where $\blam$ is a multicore, moreover by Lemma~\ref{lemma:weight_sum_r=2} it suffices to consider the case $r = 2$. Besides, by Remark~\ref{remark:ci} we can assume $e > 0$. By~\eqref{equation:charges_charge0} and since $\wmc[(\charge_1,\charge_2)](\lambda^{(1)},\lambda^{(2)}) = \wmc[(\charge_2,\charge_1)](\lambda^{(2)},\lambda^{(1)})  $, we can further assume that $\mcharge = (0,\charge)$ with $\charge \in \{0,\dots,e-1\}$. We have, using~\eqref{equation:yi_ci_ci+1} and~\eqref{equation:ci_ujk},
\begin{align*}
\wmc(\blam)
&=
\ci[0]{\mcharge}(\blam) + \ci[\charge]{\mcharge}(\blam) - \frac{1}{2}\sum_{i \in \Ze}\left(\ci{\mcharge}(\blam) - \ci[i+1]{\mcharge}(\blam)\right)^2
\\
&=
\sum_{j = 1}^2\left(\ci[0]{\charge_j}(\lambda^{(j)}) + \ci[\charge]{\charge_j}(\lambda^{(j)})\right) - \frac{1}{2}\sum_{i \in \Ze} \left(\sum_{j = 1}^2\left(\ci{\charge_j}(\lambda^{(j)}) - \ci[i+1]{\charge_j}(\lambda^{(j)})\right)\right)^2
\\
&=
\left(\bigl\lVert \xrostamod^{\chargez}(\lambda^{(1)})\bigr\rVert^2 + u_{1,2}(\blam)\right) + \left(\bigl\lVert \xrostamod^{\charge}(\lambda^{(2)})\bigr\rVert^2 + u_{2,1}(\blam)\right) - \frac{1}{2}\sum_{i \in \Ze} \left(\xrostamod^\mcharge_{i1}(\blam) + \xrostamod^\mcharge_{i2}(\blam)\right)^2
\\
&=
\bigl\lVert \xrostamod^{\chargez}(\lambda^{(1)})\bigr\rVert^2  +  \bigl\lVert \xrostamod^{\charge}(\lambda^{(2)})\bigr\rVert^2  + u_{1,2}(\blam) + u_{2,1}(\blam)  - \frac{1}{2}\sum_{i \in \Ze}\left(\xrostamod^\mcharge_{i1}(\blam)^2 + \xrostamod^\mcharge_{i2}(\blam)^2 + 2\xrostamod^\mcharge_{i1}(\blam)\xrostamod^\mcharge_{i2}(\blam)\right)
\\
&=
\frac{1}{2}\left(\lVert \xrostamod^\chargez(\lambda^{(1)})\rVert^2 + \lVert \xrostamod^\charge(\lambda^{(2)})\rVert^2 \right) + u_{1,2}(\blam) + u_{2,1}(\blam) - \sum_{i \in \Ze} \xrostamod^\mcharge_{i1}(\blam)\xrostamod^\mcharge_{i2}(\blam)
\\
&=
\frac{1}{2}\left(\lVert \xrostamod^\chargez(\lambda^{(1)})\rVert^2 + \lVert \xrostamod^\charge(\lambda^{(2)})\rVert^2 \right) + u_{1,2}(\blam) + u_{2,1}(\blam) - \sum_{i \in \Ze} \xrostamod^\chargez_i(\lambda^{(1)})\xrostamod^\charge_i(\lambda^{(2)})
\\
&=
\frac{1}{2} \lVert\xrostamod^\chargez(\lambda^{(1)}) - \xrostamod^{\charge}(\lambda^{(2)})\rVert^2 + u_{1,2}(\blam) + u_{2,1}(\blam).
\end{align*}
If $\charge = 0$ we have $u_{1,2}(\blam) = u_{2,1}(\blam) = 0$ so that $\wmc(\blam) \geq 0$ as desired, otherwise we have $\charge \in \{1,\dots,e-1\}$ so that
\begin{align*}
u_{1,2}(\blam)
&= - \xrostamod^\mcharge_{0,1}(\blam) - \dots - \xrostamod^\mcharge_{\charge-1,1}(\blam),
\\
u_{2,1}(\blam)
&= -\xrostamod^\mcharge_{\charge,2}(\blam) - \dots - \xrostamod^\mcharge_{e- 1,2}(\blam).
\end{align*}
Now by~\eqref{equation:sumy_0} we have
\[
u_{2,1}(\blam) = \xrostamod^\mcharge_{0,2}(\blam) + \dots + \xrostamod^\mcharge_{\charge-1,2}(\blam).
\]
Hence, setting $z \coloneqq \xrostamod^\chargez(\lambda^{(1)}) - \xrostamod^\charge(\lambda^{(2)})$ we obtain
\[
\wmc(\blam) = \frac{1}{2}\lVert z \rVert^2 - z_0 - \dots - z_{\charge - 1},
\]
thus we conclude by the below Lemma~\ref{lemma:quadratic_positive}, recalling~\eqref{equation:sumy_0}.
\end{proof}

\begin{lemma}
\label{lemma:quadratic_positive}
Assume that $e > 0$. Let $z = (z_1, \dots, z_e) \in \Z^e$ such that $z_1 + \dots + z_e =0$. Then $\lVert z \rVert ^2 \geq 2(z_1 + \dots + z_\charge)$ for any $\charge \in \{1, \dots, e\}$.
\end{lemma}

\begin{proof}
First, note that for any integer $t \in \Z$ we have $t^2 \geq 2t$ unless if $t = 1$. We define the following two sets:
\begin{align*}
I &\coloneqq \left\{i \in \{1, \dots, \charge\} : z_i = 1\right\},
\\
J &\coloneqq \left\{i \in \{1,\dots, e \} : z_i \leq 0\right\},
\end{align*}
in particular $I \cap J = \emptyset$.
We have
\[
0 = z_1 + \dots + z_e \geq \sum_{i \in J} z_i  + \sum_{i \in I} z_i,
\]
so that $\lvert I\rvert + \sum_{i \in J} z_i \leq 0$. For any $i \in J$ we have $z_i \in \Z_{\leq 0}$ thus
\[
z_i^2 - 2z_i \geq z_i^2 \geq -z_i,
\]
hence we obtain
\begin{align*}
\lVert z \rVert^2 - 2(z_1 + \dots + z_\charge)
&\geq \sum_{i \in  I} (z_i^2 - 2z_i) + \sum_{i \in J} (-z_i)
\\
&= \sum_{i \in I} (-1) - \sum_{i \in J} z_i
\\
&= -\lvert I\rvert - \sum_{i \in J} z_i
\\
&\geq 0,
\end{align*}
as desired.
\end{proof}

\begin{remark}
Fayers' proof of Proposition~\ref{proposition:weight_positive} uses Proposition~\ref{proposition:weight_min_two_sets} (from which the result is immediate). The proof we give here has the advantage to be more direct from the definition of $\wmc$ and $\ci{\mcharge}$ (a similar remark holds for the proof of Lemma~\ref{lemma:weight_sum_r=2}).
\end{remark}

Note that $\wmc(\blam) \geq r\sum_{j = 1}^r \wc[\charge_j](\lambda^{(j)})$. Indeed, by Propositions~\ref{proposition:we=wc} and~\ref{proposition:weight_positive} the inequality holds when $\blam$ is an $e$-multicore,  and we conclude by Lemma~\ref{lemma:weight_add_rim_hook}.

\bigskip
Finally, if $\blam$ is an $e$-multicore, for any $i \in \Ze$ and $j, k \in \{1, \dots, r\}$, as in~\cite{fayers:weights,fayers:core} we define
\[
\gamma_{i,jk}^{\mcharge}(\blam) \coloneqq \frac{\bfayers^{\mcharge}_{ij}(\blam) - \bfayers^{\mcharge}_{ik}(\blam)}{e} = \xrostam^\mcharge_{ij}(\blam) - \xrostam^\mcharge_{ik}(\blam) \in \Z.
\]
These integers depend on the multicharge $\mcharge$, however for any $i, l \in \Ze$ and $j, k \in \{1, \dots, r\}$ the set defined by the integers
\begin{equation}
\label{equation:gamma_iljk}
\gamma_{il,jk}(\blam) \coloneqq \gamma_{i,jk}^{\mcharge}(\blam) - \gamma_{l,jk}^{\mcharge}(\blam),
\end{equation}
does not.
%
%
%

\begin{proposition}[\protect{\cite[Proposition 3.8]{fayers:weights}}]
\label{proposition:weight_min_two_sets}
Assume that $r = 2$ and let $\blam$ be a bicore.  Assume that $\gamma_{il,12}(\blam) \leq 2$ for all $i, l \in \Ze$. Then $\wmc(\blam)$ is the smaller of the two integers
\[
\#\bigl\{ i \in \Ze : \gamma_{il,12}(\blam) = 2 \text{ for some } l \in \Ze\bigr\},
\]
and
\[
\#\bigl\{ l \in \Ze : \gamma_{il,12}(\blam) = 2 \text{ for some } i \in \Ze\bigr\}.
\]
\end{proposition}

Proposition~\ref{proposition:weight_min_two_sets} immediately implies that if $\blam$ is a bicore satisfying $\gamma_{il,12}(\blam) \leq 2$ for all $i, l \in \Ze$ then $\wmc(\blam) \leq e$. We will see that we can strengthen this inequality for a certain class of multipartitions. One on the aims of this paper is to study the consequences of such an inequality, namely, on the blocks of multipartitions.

\begin{remark}
\label{remark:weight_multicore_not_bounded}
By~\cite[Lemma 3.7]{fayers:weights}, the $\mcharge$-weight function on $e$-multicores is not bounded. For instance, if $e > 0$ assume  that $r = 2$ and $\mcharge = (0,0)$, and let $\lambda$ and $\mu$ be two bicores. By~\eqref{equation:yi_ci_ci+1} and Corollary~\ref{corollary:expression_ci_in_terms_of_yi} we have
\begin{align*}
\wmc(\lambda,\mu)
&=
2\ci[0]{0}(\lambda,\mu)  - \frac{1}{2}\sum_{i \in \Ze} \left(\ci{0}(\lambda,\mu) - \ci[i+1]{0}(\lambda,\mu)\right)^2
\\
&=
2\ci[0]{0}(\lambda) + 2\ci[0]{0}(\mu)  - \frac{1}{2}\sum_{i \in \Ze} \left(\ci{0}(\lambda) - \ci[i+1]{0}(\lambda) + \ci{0}(\mu) - \ci[i+1]{0}(\mu)\right)^2
\\
&=
\bigl\lVert \xrostamod^0(\lambda) \bigr\rVert^2 + \bigl\lVert \xrostamod^0(\mu)\bigr\rVert^2
- \frac{1}{2}\sum_{i \in \Ze} \left(\xrostamod^0_i(\lambda) + \xrostamod^0_i(\mu)\right)^2
\\
&=
\frac{1}{2}\bigl\lVert \xrostamod^0(\lambda) - \xrostamod^0(\mu)\bigr\rVert^2.
\end{align*}
Thus, if we chose the $e$-core $\mu$ so that $\xrostamod^0(\mu) = - \xrostamod^0(\lambda)$ (which satisfies the conditions of Proposition~\ref{proposition:1-1_correspondance_cores_abaci_y} indeed) we obtain
\[
\wmc(\lambda,\mu) = 2\bigl\lVert\xrostamod^0(\lambda)\bigr\rVert^2 = \ci[0]{0}(\lambda),
\]
which is not bounded since we can find $e$-cores with arbitrary  large number of $0$-nodes (see, for instance, Example~\ref{example:alpha=halpha0}). If $e = 0$ one can simply consider the partition $\lambda^{[h]} \coloneqq (h, \dots, h)$ where $h \geq 1$ is repeated $h$ times, and see that $\wmc(\lambda^{[h]},\emptyset) = \wc[0](\lambda^{[h]}) + \ci[0]{0}(\lambda^{[h]}) = 0+ h = h$. 
\end{remark}

\subsection{Blocks for multipartitions}
\label{subsection:blocks_multipartitions}

The combinatorics between an $r$-partition $\blam$ and $\alpha^\mcharge(\blam)$ is more intricate than in the case $r = 1$. For instance, the natural notion of $e$-multicore no more suffices to distinguish the blocks between two $r$-partitions.  However, as in the level $1$ case the quantities $\alpha^\mcharge(\blam)$ determine the blocks of the associated \emph{Ariki--Koike algebra} (see~\cite{lyle-mathas:blocks}).

\begin{example}
\label{example:block_bicore_notbicore}
Let $r = 2$, $e = 2$ and $\mcharge = (0, 1)$. The bipartitions $\blam = \bigl((1),(1)\bigr)$ and $\bmu = \bigl((2),\emptyset\bigr)$ lie in the same block $\alpha_0 + \alpha_1$, however $\blam$ is a bicore whereas $\bmu$ is not.
\end{example}

Again, we define
\[
\Qmcharge \coloneqq \left\{ \alpha^\charge(\blam) : \blam \text{ is an } r\text{-partition of } n \text{ for some } n \in \mathbb{N}\right\}.
\]

\begin{remark}
As we mentioned in~\textsection\ref{subsection:partitions}, the case $e=1$ is particularly easy to deal with, and indeed we have here $\Qmcharge=\N\alpha_0=\Qpos$. 
\end{remark}

 Since $\alpha^\mcharge(\blam)=  \sum_{j = 1}^r \alpha^{\charge_j}(\lambda^{(j)})$  we have $\Qmcharge = \sum_{j = 1}^r \Qcharge[\charge_j]$, however the sum is not direct. In particular, using the $1{:}1$-parametrisation of~\textsection\ref{subsection:blocks_partitions} for each $\Qcharge[\charge_j]$, we can give a naive parametrisation of $\Qmcharge$ but this parametrisation will not be $1{:}1$. We will rather aim at a generalisation of Lemma~\ref{lemma:cores_w=0} and Proposition~\ref{proposition:partitions_wgeq0}, looking for an implicit description of $\Qmcharge$. It  turns out that the following simple result, which comes from Lemmas~\ref{lemma:block_add_rim_hook_partition} and~\ref{lemma:remove_hook_abacus}, is the key of what follows.

\begin{proposition}
\label{proposition:Qmcharge_stable_add_const}
Let $\alpha \in \Qmcharge$ and $h \in \N$. Then $\alpha + h\constQ \in \Qmcharge$.
\end{proposition}
 
  We have the following (weak) generalisation of Corollary~\ref{corollary:partition_any_alpha+h_block} (see also~\cite[Theorem 4.7]{fayers:coreII}).
 
\begin{proposition}
\label{proposition:multipartition_any_alpha+h_block}
Assume $e > 0$. For any $\alpha \in \Q$ and any $h \geq -\max\bigl\{\wc[\charge_j](\alpha) : j \in \{1, \dots, r\}\bigr\}$ we have 
\[
\alpha + h\constQ \in \Qmcharge.
\]
\end{proposition}

\begin{proof}
Let $j \in \{1, \dots, r\}$. If $h \geq - \wc[\charge_j](\alpha)$ then by Corollary~\ref{corollary:partition_any_alpha+h_block} we have $\alpha + h\constQ \in \Qcharge[\charge_j]$. We conclude the proof since $\Qcharge[\charge_j] \subseteq \Qmcharge = \sum_{k = 1}^r \Qcharge[\charge_k]$.
\end{proof}

\bigskip

As we saw in Example~\ref{example:block_bicore_notbicore}, the raw notion of $e$-multicores is not adapted to the study of blocks. 

\begin{definition}[\cite{fayers:core}]
Let $\alpha \in \Q$. We say that $\alpha$ is a \emph{core block} if $\alpha \in \Qmcharge$ and either:
\begin{itemize}
\item $e = 0$;
\item $e > 0$ and $\alpha - \constQ \notin \Qmcharge$.
\end{itemize} 
\end{definition}
By Lemmas~\ref{lemma:block_add_rim_hook_partition} and \ref{lemma:remove_hook_abacus}, a block $\alpha \in \Qmcharge$ is a core block if and only if every $r$-partition $\blam$ such that $\alpha^\mcharge(\blam) = \alpha$ is an $e$-multicore.

\begin{remark}
\label{remark:coreblock_r=1}
In particular, if $r = 1$ then $\alpha^\charge(\lambda)$ is a core block if and only if $\lambda$ is an $e$-core, in which case $\lambda$ is the only $e$-core (and partition) that lies in $\alpha^\charge(\lambda)$ (see Lemma~\ref{lemma:block_core_unique_partition}).
\end{remark}

\begin{remark}
\label{remark:coreblock_e=1}
If $e = 1$ then the only $e$-multicore is the empty multipartition, thus there is only one core block, which is $0 \in \Qmcharge$.
\end{remark}

\begin{remark}
The notion of core block is related to the notion of \emph{maximal weight} for (irreducible) modules on Kac--Moody algebras: see, for instance,~\cite[\textsection 3.1]{jacon-lecouvey} for the connection with Kac--Moody algebras and~\cite[\textsection 20.3]{carter} for the notion of maximal weights.
\end{remark}

\begin{definition}[\protect{\cite{lyle-mathas:blocks}}]
A multipartition $\blam$ is an \emph{$\mcharge$-reduced $e$-multicore} if $\alpha^\mcharge(\blam)$ is a core block.
\end{definition}

Note that the term  \emph{core multipartition} has a different meaning, see~\cite{fayers:simultaneous,jacon-lecouvey}. Any $\mcharge$-reduced $e$-multicore is an $e$-multicore, and the converse holds when  $r = 1$ by Remark~\ref{remark:coreblock_r=1}.

\begin{lemma}
\label{lemma:reduced_multicore_restrict_two_comp}
Let $\blam = \bigl(\lambda^{(1)},\dots,\lambda^{(r)}\bigr)$  be an $\mcharge$-reduced multicore. Then for each $1 \leq j < k \leq r$, the bipartition $\bigl(\lambda^{(j)},\lambda^{(k)}\bigr)$ is an $(\charge_j,\charge_k)$-reduced bicore.
\end{lemma}

\begin{proof}
We prove the contraposition. Assume that $\bigl(\lambda^{(j)},\lambda^{(k)}\bigr)$ is not an $(\charge_j,\charge_k)$-reduced bicore. Then we can find a bipartition $\bigl(\mu^{(j)},\mu^{(k)}\bigr)$ in its block which is not a bicore, in particular we can remove an $e$-rim hook to $\mu^{(j)}$ or $\mu^{(k)}$. But then the multipartition
\[
\bigl(\lambda^{(1)},\dots,\lambda^{(j-1)},\mu^{(j)},\lambda^{(j+1)},\dots,\lambda^{(k-1)},\mu^{(k)},\lambda^{(k+1)},\dots,\lambda^{(r)}\bigr)
\]
lie in the same block as $\blam$ and is not an $e$-multicore, thus $\blam$ is not a reduced $\mcharge$-multicore.
\end{proof}

\begin{lemma}
\label{lemma:core_block_associated_with_any_alpha}
Assume that $e > 0$ and let $\alpha \in \Q$. There is a unique integer $h \in \Z$ such that $\alpha - h\constQ$ is a core block and we have $h = \max\bigl\{ k \in \Z : \alpha - k\constQ \in \Qcharge\bigr\}$. Moreover, we have $\alpha \in \Qmcharge \iff h \geq 0$.
\end{lemma}

\begin{proof}
Write $H \coloneqq \bigl\{ k \in \Z : \alpha - k\constQ \in \Qcharge\bigr\}$. The set $H$ is non-empty by Proposition~\ref{proposition:multipartition_any_alpha+h_block} and bounded above by~\eqref{equation:wmcalpha+h} and Proposition~\ref{proposition:weight_positive}. If $h \coloneqq \max H$, then by definition we have $\alpha - h\constQ \in \Qmcharge$ and $\alpha - (h+1)\constQ \notin \Qmcharge$ thus $\alpha-h\constQ$ is a core block. We now prove the unicity. If $h' \in \Z$  is such that $\alpha-h'\constQ$ is a core block then in particular $\alpha-h'\constQ \in \Qmcharge$ thus $h' \leq h$ by maximality of $h$. Now if $h' < h$ then
\[
\alpha - (h'+1)\constQ = (\alpha-h\constQ) + (h-h'-1)\constQ,
\]
but the left-hand side is not in $\Qmcharge$ since $\alpha-h'\constQ$ is a core block while the right-hand side is in $\Qmcharge$ by Proposition~\ref{proposition:Qmcharge_stable_add_const}, thus $h' = h$. Finally, we have $\alpha \in \Qmcharge \iff 0 \in H \iff h \geq 0$, which concludes the proof.
\end{proof}

\begin{remark}
\label{remark:infinite_number_core_block}
If 	$e > 0$, Lemma~\ref{lemma:core_block_associated_with_any_alpha} ensures that $Q / \langle \constQ\rangle \simeq \Z^{e-1}$ is in $1{:}1$-correspondence with the set of core blocks. In particular,  if $e \geq 2$ then given a multicharge there is an infinite number of core blocks. 
\end{remark}

\begin{remark}
\label{remark:h}
Let $\alpha \in \Q$ and let $h \in \Z$ such that $\hat\alpha\coloneqq\alpha-h\constQ$ is a core block. By~\eqref{equation:wcalpha+h} we have $h = \frac{\wmc(\alpha) - \wmc(\hat\alpha)}{r}$. Note that the quantity $\wmc(\hat\alpha)$ can be computed from $\alpha$, see~\cite{fayers:coreII}. However, the formula given in~\cite{fayers:coreII} does not seem suitable for the purpose of this paper.
\end{remark}

\begin{definition}
\label{definition:s-core}
Assume that $e > 0$ and let $\alpha \in \Q$. We say that the unique core block of the form $\alpha-h\constQ$ for $h \in \Z$  as in Lemma~\ref{lemma:core_block_associated_with_any_alpha} is the \emph{$\mcharge$-core} of $\alpha$.
\end{definition}

We will simply use \emph{core} instead of $\mcharge$-core when the multicharge $\mcharge$ is understood from the context. 
Given a multipartition $\blam$,  by Lemma~\ref{lemma:block_add_rim_hook_partition} to compute the core of $\alpha^\mcharge(\blam)$ we can proceed as follows. We first look at the $e$-multicore $\overline\blam$ of $\blam$. If every $r$-partition lying in the same block as $\overline\blam$  is an $e$-multicore then we are done (in which case $\alpha^\mcharge(\overline{\blam})$ is the core of $\alpha^\mcharge(\blam)$), otherwise we repeat the procedure with an $r$-partition in the block of $\overline\blam$ that is not an $e$-multicore.

\begin{remark}
\label{remark:score_ecore}
The terminology fits with the usual notion of $e$-core when $r = 1$: if $\lambda$ and $\mu$ are two partitions then $\alpha^\charge(\lambda)$ is the $\charge$-core of $\alpha^\charge(\mu)$ if and only if $\lambda$ is the $e$-core of $\mu$ (by Remark~\ref{remark:coreblock_r=1} together with Lemmas~\ref{lemma:block_add_rim_hook_partition} and~\ref{lemma:nakayamas_conjecture}).
\end{remark}

By Lemma~\ref{lemma:weight_add_rim_hook} and Proposition~\ref{proposition:weight_positive}, we have the following particular case of core block. The aim of this paper is to give a (weak) converse.

\begin{proposition}
Let $\alpha \in \Qmcharge$ with $\wmc(\alpha) \leq r-1$. Then $\alpha$ is a core block.
\end{proposition}


We say that a multicharge $\mcharge' \in \Z^e$ is \emph{compatible} with $\mcharge$ (or simply \emph{compatible} if $\mcharge$ is clear from the context) if $\mcharge' - \mcharge \in e\Z$. In particular, note that if $\mcharge'$ is a compatible multicharge then a multicore $\blam$ is $\mcharge$-reduced if and only if it is $\mcharge'$-reduced. We have the following characterisation of core blocks in terms of abaci.

\begin{proposition}[\protect{\cite[Theorem 3.1]{fayers:core}}]
\label{proposition:multicore_equivalent_statements}
Assume that $e > 0$ and let $\blam$ be an $e$-multicore. The following assertions are equivalent.
\begin{enumerate}[$(i)$]
\item
\label{item:reduced_multicore}
The $e$-multicore $\blam$ is $\mcharge$-reduced.
\item
\label{item:bfayers}
There exist a compatible multicharge $\mcharge'$  and integers $\bfayers_1, \dots, \bfayers_e \in \Z$ such that
\[
\bfayers^{\mcharge'}_{ij}(\blam) \in \{\bfayers_i, \bfayers_i + e\},
\]
for all $i \in \Ze$ and $j\in \{1, \dots, r\}$.
\item
\label{item:sfayers}
There exist a compatible multicharge $\mcharge'$ and 
integers $\sfayers_1,\dots,\sfayers_r \in \Z$ such that
\[
\gamma_{i,jk}^{\mcharge'}(\blam)\leq\sfayers_j - \sfayers_k + 1,
\]
for all $i \in \Ze$ and $j, k \in \{1, \dots, r\}$.
\item
For any compatible multicharge $\mcharge'$, there exist
integers $\sfayers_1,\dots,\sfayers_r \in \Z$ such that
\[
\gamma_{i,jk}^{\mcharge'}(\blam)\leq\sfayers_j - \sfayers_k + 1,
\]
for all $i \in \Ze$ and $j, k \in \{1, \dots, r\}$.
\end{enumerate}
\end{proposition}

\begin{corollary}
\label{corollary:gamma_leq_1}
Assume that $e > 0$. An $e$-multicore  $\blam$ is  $\mcharge$-reduced if and only if there exists a compatible multicharge $\mcharge'$ such that
\[
\gamma^{\mcharge'}_{i,jk}(\blam) \leq 1,
\]
for all $i \in \Ze$ and $j, k \in \{1, \dots, r\}$.
\end{corollary}

\begin{proof}
If the $e$-multicore $\blam$ is  $\mcharge$-reduced  then by Proposition~\ref{proposition:multicore_equivalent_statements}\ref{item:bfayers} we can write $\bfayers^{\mcharge'}_{ij}(\blam) = \bfayers_i + \epsilon_{ij}e$ for all $i \in \Ze$ and $j \in \{1, \dots, r\}$, where $\epsilon_{ij}\in\{0,1\}$ and $\mcharge'$ is a compatible multicharge. Thus, we obtain
\[
\gamma_{i,jk}^{\mcharge'}(\blam) = \frac{\bfayers^{\mcharge'}_{ij}(\blam) - \bfayers^{\mcharge'}_{ik}(\blam)}{e} = \epsilon_{ij} - \epsilon_{ik} \leq 1,
\]
for all $i \in \Ze$ and $j,k\in\{1,\dots,r\}$.
Conversely, if there exists a compatible multicharge $\mcharge'$ such that $\gamma^{\mcharge'}_{i,jk}(\blam) \leq 1$ for all $i \in \Ze$ and $j, k \in \{1, \dots, r\}$ then item~\ref{item:sfayers} of Proposition~\ref{proposition:multicore_equivalent_statements} is satisfied with $\sfayers_1 = \dots = \sfayers_r$ thus the $e$-multicore  $\blam$ is $\mcharge$-reduced.
\end{proof}

\section{Implicit description of the blocks}
\label{section:implicit_description_blocks}

We will give in this section the central results of the paper. We will first give a formula to compute the weight as a map from binary matrices. Let $r \geq 2$ and $e \geq 1$.

\subsection{Weight of a tuple of subsets}
\label{subsection:weight_binary}

We first begin with the following standard fact.
To an $r$-tuple of subsets $(E_1,\dots,E_r)$ of $\{1,\dots,e\}$  we can associate the $e\times r$ binary matrix $\mate$ such that the characteristic vector $\mathbf{1}_{E_j}$ of $E_j$ is the $j$-th column of $\mate$ for any $j \in \{1,\dots,r\}$.  Conversely, to an $e \times r$ binary matrix $\mate$ we can associate the $r$-tuple $(E_1,\dots,E_r)$ of subsets of $\{1,\dots,e\}$ such that for any $j \in \{1,\dots,r\}$ the $j$-th column of $\mate$ is exactly $\mathbf{1}_{E_j}$.

Now recall from Proposition~\ref{proposition:multicore_equivalent_statements}\ref{item:bfayers} that to any $\mcharge$-reduced $e$-multicore we can associate an $e\times r$ binary matrix. Conversely, given an $e\times r$ binary matrix $\mate = (\epsilon_{ij})$ and integers $\bfayers_1,\dots,\bfayers_e \in \Z$ such that $\bfayers_i = i \pmod{e}$ for each $i \in \{1, \dots, e\}$, we can construct an $\mcharge$-reduced $e$-multicore $\blam = \blam(\mate, (\bfayers_i))$, where the multicharge  $\mcharge$ is determined by~\eqref{equation:sumb_charge}. Note that the integers
\begin{equation}
\label{equation:gamma_diff_epsilon}
\gamma_{i,jk}^{\mcharge}(\blam) = \epsilon_{ij} - \epsilon_{ik},
\end{equation}
do not depend on the choice of $(\bfayers_i)$ so that, recalling Proposition~\ref{proposition:weight_min_two_sets}, the weight $\wmc(\blam)$ only depends on $\mate$. Now if $(E_1,\dots,E_r)$ is the $r$-tuple of  subsets of $\{1,\dots,e\}$ corresponding to $\mate$, we can thus write:
\[
\w(E_1,\dots,E_r) \coloneqq \wmc(\blam).
\]

\begin{example}
Take $e = 3$ and $r = 2$. We consider the following subsets of $\{1,2,3\}$:
\[
E_1 = \{1,2,3\}, \quad E_2 = \{3\},
\]
which correspond to the $3\times 2$ binary matrix $\mate = (\epsilon_{ij}) = \begin{pmatrix}
1&0\\1&0\\1&1
\end{pmatrix}$. We also consider the following integers:
\[
\bfayers_1 = 1, \quad \bfayers_2 = 5, \quad \bfayers_3 = -3,
\]
which satisfy the condition $\bfayers_i = i \pmod{e}$.
For $i \in \{1,2,3\}$ and $j \in \{1,2\}$ we define $\bfayers_{ij} = \bfayers_i + \epsilon_{ij}e$. We obtain:
\begin{align*}
\bfayers_{11} &= 1+3 = 4,& \bfayers_{12} &= 1+0=1,
\\
\bfayers_{21} &= 5+3 = 8, & \bfayers_{22} &= 5+0=5,
\\
\bfayers_{31} &=-3+3 = 0, & \bfayers_{32} &= -3 + 3 = 0.
\end{align*}
Recall that $\bfayers_{ij}$ is the largest element of $\bnumber{\charge_j}(\lambda^{(j)})$ congruent to $i$ modulo $e$. We obtain that the charged beta-number associated with $\lambda^{(1)}$ is:
\begin{align*}
\bnumber{\charge_1}(\lambda^{(1)}) &= \{4,1,-2,-5,-8,\dots\} \sqcup \{8,5,2,-1,-4,-7,\dots\} \sqcup \{0,-3,-6,-9,\dots,\}
\\
&= \{8,5,4,2,1,0,-1,-2,-3,-4,\dots\},
\end{align*}
thus the associated charged abacus is:
\cabacus{6}{3,1,1,0,0,0,0}{,}
thus $\lambda^{(1)} = (3,1,1)$. Similarly, we have:
\begin{align*}
\bnumber{\charge_2}(\lambda^{(2)}) &= \{1,-2,-5,-8,\dots\} \sqcup \{5,2,-1,-4,-7,\dots,\} \sqcup \{0,-3,-6,-9,\dots,\}
\\
&= \{5,2,1,0,-1,-2,-3,-4,\dots\},
\end{align*}
thus the associated charged abacus is:
\cabacus{4}{2,0,0,0,0,0}{,}
thus $\lambda^{(2)} = (2)$.
We can deduce the associated multicharge $\mcharge=(\charge_1,\charge_2)$, but we can also  use the formula
\[
\charge_j = \frac{3+1}{2} + \frac{1}{3}\sum_{i = 1}^3 \bfayers_{ij},
\]
for $j \in \{1,2\}$, thus obtaining $\mcharge = (6,4)$. Finally, to the pair $(E_1,E_2) = \bigl(\{1,2,3\},\{3\}\bigr)$ of subsets of $\{1,2,3\}$ we have associated the $(6,4)$-reduced $3$-multicore $\blam = \bigl((3,1,1),(2)\bigr)$, and the calculation of $\w(E_1,E_2)$ follows.
\end{example}

\begin{lemma}
\label{lemma:weight_subsets}
Assume that $r = 2$ and let $E_1, E_2 \subseteq \{1,\dots,e\}$. We have:
\[
\w(E_1,E_2) = \min(\lvert E_1\rvert,\lvert E_2\rvert) - \lvert E_1 \cap E_2\rvert.
\]
\end{lemma}

\begin{proof}
Let $\bfayers_1,\dots,\bfayers_e \in \Z$ with $\bfayers_i = i \pmod{e}$, let $\mcharge$ be the multicharge associated with $(\bfayers_i)$  and let $\blam$ be the $\mcharge$-reduced $e$-multicore associated with $(\bfayers_i)$ and pair  $(E_1,E_2)$ of subsets of $\{1,\dots,e\}$. By~\eqref{equation:gamma_diff_epsilon}, for any $i \in \Ze$ we have
\[
\gamma_{i,12}^\mcharge(\blam) = \mathbf{1}_1(i) - \mathbf{1}_2(i),
\]
so that $\gamma_{i,12}^\mcharge(\blam) = 1$ if and only if $i \in E_1 \setminus E_2$. Hence, recalling~\eqref{equation:gamma_iljk},
\begin{align*}
\gamma_{il,12}(\blam) = 2
&\iff
\gamma_{i,12}^\mcharge(\blam) = 1 \text{ and } \gamma_{l,12}^\mcharge(\blam) = -1
\\
&\iff i \in E_1 \setminus E_2 \text{ and } l \in E_2 \setminus E_1,
\end{align*}
so that
\[
A \coloneqq\#\left\{i \in \Ze : \gamma_{il,12}(\blam) = 2 \text{ for some } l \in \Ze \right\}
=
\begin{cases}
\lvert E_1 \setminus E_2\rvert, &\text{if } \lvert E_2 \setminus E_1\rvert \neq 0,
\\
0,&\text{otherwise},
\end{cases}
\]
and similarly
\[
B\coloneqq\#\left\{l \in \Ze : \gamma_{il,12}(\blam) = 2 \text{ for some } i \in \Ze \right\}
=
\begin{cases}
\lvert E_2 \setminus E_1\rvert, &\text{if } \lvert E_1 \setminus E_2\rvert \neq 0,
\\
0,&\text{otherwise}.
\end{cases}
\]
Looking whether $0$ is in $\bigl\{|E_1 \setminus E_2|, |E_2 \setminus E_1|\bigr\}$ or not, we obtain:
\[
\min(A,B) = \min\bigl(|E_1 \setminus E_2|, |E_2 \setminus E_1|\bigr).
\]
Recalling from Proposition~\ref{proposition:weight_min_two_sets} that $\w(E_1,E_2) = \min(A,B)$, we thus obtain:
\begin{align*}
\w(E_1,E_2) 
&= \min\Bigl(\lvert E_1\rvert - \lvert E_1 \cap E_2\rvert, \:\lvert E_2\rvert - \lvert E_1 \cap E_2\rvert \Bigr) 
\\
&=
\min(\lvert E_1\rvert, \lvert E_2\rvert) - \lvert E_1 \cap E_2 \rvert,
\end{align*}
as announced.
\end{proof}

Note that if $E_1,\dots,E_r \subseteq \{1,\dots,e\}$ then by Lemma~\ref{lemma:weight_sum_r=2} we have:
\begin{equation}
\label{equation:weight_binary_matrix}
\w(E_1,\dots,E_r) = \sum_{1 \leq j < k \leq r} \Bigl[ \min(\lvert E_j\rvert, \lvert E_k\rvert) - \lvert E_j \cap E_k\rvert\Bigr].
\end{equation}
We conclude this subsection by the following particular case of Lemma~\ref{lemma:weight_subsets}.
 
\begin{lemma}
\label{lemma:weight_set_same_empty}
For any $E \subseteq \{1,\dots,e\}$ we have:
\[
\w(E,E) = \w(E,\emptyset) = \w(E,\{1,\dots,e\}) = 0.
\]
\end{lemma}

\begin{remark}
\label{remark:weight_mate_e=1}
Recalling Remark~\ref{remark:coreblock_e=1}, if $e = 1$ then we recover the fact that $\w(E) = 0$ for any subset $E \subseteq \{1,\dots,e\}$.
\end{remark}

\subsection{Bound for the weight of core blocks}
\label{subsection:bound_weight_core_blocks}

 Recall from Remark~\ref{remark:infinite_number_core_block} that there is an infinite number of reduced $e$-multicores.

%

\begin{theorem}
\label{theorem:bound_core_blocks}
Let $r \geq 2$ and $e > 0$. There exists a constant $C \geq 0$ such that for all multicharge $\mcharge$ and each $r$-partition $\blam$ that is an $\mcharge$-reduced $e$-multicore we have
\[
\wmc(\blam) \leq C.
\]
\end{theorem}

Note that the statement is wrong if $e = 0$ (see Remark~\ref{remark:weight_multicore_not_bounded} or the coming Remark~\ref{remark:counter_example_superlevel}). 

\begin{proof}
We first prove the case $r = 2$, which is clear since by Proposition~\ref{proposition:weight_min_two_sets} and Corollary~\ref{corollary:gamma_leq_1} we have $\wmc(\blam) \leq e$ (we could also have use Lemma~\ref{lemma:weight_subsets}). We now assume that $r \geq 3$. Since $\blam$ is a multicore, by Lemma~\ref{lemma:weight_sum_r=2} we have
\[
\wmc(\blam) = \sum_{1 \leq j < k \leq r} \wmc[(\charge_j,\charge_k)]\bigl(\lambda^{(j)},\lambda^{(k)}\bigr).
\]
By Lemma~\ref{lemma:reduced_multicore_restrict_two_comp} we know that each $\bigl(\lambda^{(j)},\lambda^{(k)}\bigr)$ is a reduced $(\charge_j,\charge_k)$-bicore, thus using the case $r = 2$ we find
\[
\wmc(\blam) \leq e \binom{r}{2},
\]
which concludes the proof.
\end{proof}

\begin{definition}
\label{definition:bestboundcoreblock}
We denote by $\bestboundcoreblock$ the smallest possible constant $C$ as in Theorem~\ref{theorem:bound_core_blocks}, that is:
\[
\bestboundcoreblock = \mathrm{sup}\bigl\{ \wmc(\blam) : \mcharge \in \Z^r\text{ a multicharge and } \blam \text{ an } \mcharge\text{-reduced } e\text{-multicore} \bigr\} < \infty.
\]
\end{definition}

\begin{corollary}
\label{corollary:Qmcharge_supset_superlevel_weight}
Assume that $e > 0$.
We have
\[
\Qmcharge \supseteq \bigl\{\alpha \in \Q : \wmc(\alpha) > \boundcoreblock - r\bigr\}.
\]
\end{corollary}

\begin{proof}
Let $\alpha \in \Q$ and assume that $\alpha \notin \Qmcharge$. We will prove that $\wmc(\alpha) \leq \boundcoreblock - r$.
By Lemma~\ref{lemma:core_block_associated_with_any_alpha}, we can find $h \in \Z$ such that $\alpha - h\constQ$ is a core block, and since $\alpha \notin \Qmcharge$ we have $h < 0$. Now by Theorem~\ref{theorem:bound_core_blocks} we have $\wmc(\alpha - h\constQ) \leq \boundcoreblock$ thus by~\eqref{equation:wmcalpha+h} we have $\wmc(\alpha) \leq \boundcoreblock + rh \leq \boundcoreblock - r$. This concludes the proof.
\end{proof}

Note that the inclusion of Corollary~\ref{corollary:Qmcharge_supset_superlevel_weight} is not trivial, that is, the set $\bigl\{\alpha \in \Q : \wmc(\alpha) > \boundcoreblock - r\bigr\}$ is not empty, and in fact it is even infinite (by~\eqref{equation:wmcalpha+h}).

\begin{remark}
\label{remark:counter_example_superlevel}
The result of Corollary~\ref{corollary:Qmcharge_supset_superlevel_weight} does not hold if $e = 0$. For instance, assume that $r = 2$ and take $\mcharge = (0,\charge)$ for $\charge \geq 0$. For any $h \geq 1$, as in Remark~\ref{remark:weight_multicore_not_bounded} we consider the partition $\lambda^{[h]} = (h,\dots,h)$ where $h$ is repeated $h$ times and we define
\[
\alpha^{[h]} \coloneqq \alpha^0(\lambda^{[h]}) + \alpha_{h + \charge + 1}
= \sum_{\lvert i\rvert < h} (h - \lvert i \rvert) \alpha_i + \alpha_{h+\charge+1}.
\]
Let $(\cialpha)_{i \in \Z} \in \Z^{\Z}$ so that $\alpha^{[h]} = \sum_{i \in \Z}\cialpha\alpha_i$. Since $\cialpha[h+\charge] = 0 < \cialpha[h+\charge+1]$ and $\mcharge = (0,\charge)$, we have $\alpha^{[h]} \notin \Qmcharge$ (by~\eqref{equation:yi_ci_ci+1} and~\eqref{equation:einfinite_x01}), however
\begin{align*}
\wmc(\alpha^{[h]})
&=
2\cialpha[0] - \frac{1}{2}\sum_{i \in \Z} (\cialpha - \cialpha[i+1])^2
\\
&=
2\cialpha[0] - \frac{1}{2}\left(\sum_{i = -h}^{h-1} (\cialpha - \cialpha[i+1])^2 + (\cialpha[h+\charge] - \cialpha[h+\charge+1])^2 + (\cialpha[h+\charge+1] - \cialpha[h+\charge+2])^2\right)
\\
&=
2h - \frac{1}{2}\left(\sum_{i = -h}^{h-1} 1^2 + 1^2 + 1^2\right)
\\
&=
2h - \frac{1}{2}(2h + 2)
\\
&=
h - 1.
\end{align*}
 Note that this calculation  remains valid for $e \geq 2(h+1) + \charge$. This proves that we can find elements in $\Q\setminus\Qmcharge$  of arbitrarily large weight.
\end{remark}

We will give in~\textsection\ref{subsection:optimality} more information about the constant $\bestboundcoreblock$. In particular, in the following particular cases the situation is the same as in level one (cf. Proposition~\ref{proposition:partitions_wgeq0}).

\begin{proposition}
\label{proposition:Qmcharge_semialg_particular_case}
Assume that $(r, e) \in \bigl\{(2,2), (2,3), (3,2)\bigr\}$. Then $\Qmcharge = \bigl\{\alpha \in \Q : \wmc(\alpha) \geq 0\bigr\}$.
\end{proposition}

\begin{proof}
We will see in Proposition~\ref{proposition:bound_r=2} (respectively, Proposition~\ref{proposition:bound_r=3}) that  $\boundcoreblock = 1$ (resp. $\boundcoreblock = 2$) if $r = 2$ and $e \in \{2,3\}$ (resp. $r = 3$ and $e = 2$). In these three cases we have $\boundcoreblock \leq r - 1$, thus
\[
\left\{\alpha \in \Q : \wmc(\alpha) \geq  0\right\} \subseteq
\left\{\alpha \in \Q : \wmc(\alpha) \geq \bestboundcoreblock - r +1 \right\}.
\]
We conclude the proof by Proposition~\ref{proposition:weight_positive} and Corollary~\ref{corollary:Qmcharge_supset_superlevel_weight} since we then find
\[
\Qmcharge  \subseteq \{\alpha \in \Q : \wmc(\alpha) \geq 0\} \subseteq \left\{\alpha \in \Q : \wmc(\alpha) > \bestboundcoreblock - r  \right\} \subseteq \Qmcharge.
\]
\end{proof}

\subsection{Optimality}
\label{subsection:optimality}

We now assume that $e \geq 2$. The aim of this subsection is to give sharp bounds for $\bestboundcoreblock$. Note that we have seen during the proof of Theorem~\ref{theorem:bound_core_blocks} that $\bestboundcoreblock \leq \frac{er(r-1)}{2}$. In fact, we will see that  this first bound is roughly the value of $\bestboundcoreblock$ multiplied by a factor of $4$ (see Corollary~\ref{corollary:encadrement_final}). We will also compute $\bestboundcoreblock$ for small values of $e$ or $r$ (see~\textsection\ref{subsubsection:values_small_parameters} and~\textsection\ref{subsubsection:more_values}).  We first begin by gathering some elementary results.

\begin{lemma}
\label{lemma:binomial_coefficients}
Let $k \in \{1,\dots,e-1\}$.
\begin{enumerate}[$(i)$]
\item
\label{item:binomial}
We have:
\[
k\binom{e-1}{k} = (e-1)\binom{e-2}{k-1},
\]
and:
\[ (e-k)\binom{e}{k} = e\binom{e-1}{k}.\]
\item 
\label{item:k(e-k)}
The quantity $k(e-k)$  is maximal for $k = \lfloor\frac{e}{2}\rfloor$, the corresponding maximal value being $\lfloor\frac{e^2}{4}\rfloor$.
\end{enumerate}
\end{lemma}

\begin{proof}
\begin{enumerate}[$(i)$]
\item The first identity is standard. For the second one, using the identity $\binom{n}{m} = \binom{n}{n-m}$ twice and the first one we have:
\[
(e-k)\binom{e}{k} = (e-k)\binom{e}{e-k} = e\binom{e-1}{e-k-1} = e\binom{e-1}{k}.
\]
\item The first assertion is clear. For the second one, the result is clear if $e$ is even, and if $e$ is odd then $\lfloor\frac{e}{2}\rfloor = \frac{e-1}{2}$ and:
\[
\left\lfloor\frac{e}{2}\right\rfloor\left(e-\left\lfloor\frac{e}{2}\right\rfloor\right) = \frac{e-1}{2}\frac{e+1}{2} = \frac{e^2-1}{4} = \left\lfloor\frac{e^2}{4}\right\rfloor.
\]
\end{enumerate}
\end{proof}

\subsubsection{Values for small parameters}
\label{subsubsection:values_small_parameters}

We will here compute $\bestboundcoreblock$ when $r = 2$ or $e = 2$.

\begin{proposition}[Case $r = 2$]
\label{proposition:bound_r=2}
We have $\bestboundcoreblock[2] = \lfloor \frac{e}{2}\rfloor$.
\end{proposition}

\begin{proof}
By Lemma~\ref{lemma:weight_subsets}, it suffices to prove that for any $E, F \subseteq \{1,\dots,e\}$ we have $\min(\lvert E\rvert,\lvert F \rvert) - \lvert E \cap F \rvert \leq \lfloor \frac{e}{2}\rfloor$ and that equality happen for some $E, F$.  Writing $E = (E \cap F) \sqcup E'$ and $F = (E \cap F) \sqcup F'$, we have
\[
\min(\lvert E \rvert, \lvert F\rvert) - \lvert E \cap F\rvert = \min(\lvert E'\rvert,\lvert F'\rvert) \leq \left\lfloor \frac{e}{2}\right\rfloor,
\]
since $E' \cap F' = \emptyset$. We conclude the proof since equality holds when $\lvert E \rvert = \lfloor \frac{e}{2}\rfloor$ and $F = E^c$.
\end{proof}

\begin{proposition}[Case $e=2$]
\label{proposition:bound_e=2}
We have $\bestboundcoreblocke[2] = \left\lfloor\frac{r^2}{4}\right\rfloor$.
\end{proposition}

\begin{proof}
Let $E_1,\dots,E_r \subseteq\{1,2\}$. By Lemma~\ref{lemma:weight_set_same_empty}, if $E_j = \emptyset$ or $E_j = \{1,2\}$ for some $j \in \{1,\dots,r\}$ then $\w(E_1,\dots,E_r)$ does not decrease if we replace $E_j$ by $\{1\}$. Hence, we can assume that for all $j \in \{1,\dots,r\}$ we have $E_j = \{1\}$ or $E_j = \{2\}$. Let $s \in \{0,\dots,r\}$ so that, after reordering, we have $E_1 = \dots = E_s = \{1\}$ and $E_{s+1} = \dots = E_r = \{2\}$. Then by Lemmas~\ref{lemma:weight_sum_r=2} and~\ref{lemma:weight_set_same_empty} we have
\[
\w(E_1,\dots,E_r) = s(r-s)\,\w(\{1\},\{2\}) = s(r-s),
\]
and this concludes the proof by Lemma~\ref{lemma:binomial_coefficients}\ref{item:k(e-k)}.
\end{proof}

\subsubsection{Bounds}
\label{subsubsection:bounding_bounds}

 Recall that we always assume that $e \geq 2$ (recall from Remark~\ref{remark:weight_mate_e=1} that  $\bestboundcoreblocke[1] = 0$), with the exception of Lemma~\ref{lemma:superadditive}. We will first give a lower bound for $\bestboundcoreblock$.

\begin{lemma}
\label{lemma:superadditive}
Let $r \geq 2$. The sequence $(\bestboundcoreblock)_{e \geq 1}$ is superadditive, in other words:
\[
\bestboundcoreblocke[e+e'] \geq \bestboundcoreblock + \bestboundcoreblocke[e'],
\]
for any $e, e' \geq 1$.
\end{lemma}

\begin{proof}
Let $E_1,\dots, E_r \subseteq \{1,\dots,e\}$ and $F_1,\dots,F_r \subseteq \{1,\dots,e'\}$. 
For any $j \in \{1,\dots,r\}$ we define $F'_j \coloneqq \{ i + e : i \in F_j\} \subseteq \{e+1,\dots,e+e'\}$. Note that $E_j \cap F'_k = \emptyset$ for all $1 \leq j < k \leq r$. For any $j \in \{1,\dots,r\}$, we define $G_j \coloneqq E_j \sqcup F'_j \subseteq \{1,\dots,e+e'\}$. 
For any $1 \leq j < k \leq r$ we have
\begin{align*}
\min(\lvert G_j\rvert, \lvert G_k\rvert)
&=
\min(\lvert E_j\rvert + \lvert F_j\rvert, \lvert E_k\rvert + \lvert F_k\rvert)
\\
&\geq
\min(\lvert E_j\rvert, \lvert E_k\rvert) + \min(\lvert F_j\rvert, \lvert F_k\rvert),
\end{align*}
and, recalling that $E_j \cap F'_k = \emptyset$,
\begin{align*}
G_j \cap G_k
&=
(E_j \sqcup F'_j) \cap (E_k \sqcup F'_k)
\\
&=
(E_j \cap E_k) \sqcup (F'_j \cap F'_k),
\end{align*}
so that $\w(G_j,G_k) \geq \w(E_j,E_k) + \w(F_j,F_k)$. By Lemma~\ref{lemma:weight_sum_r=2} we obtain $\w(G_1,\dots,G_r) \geq \w(E_1,\dots,E_r) + \w(F_1,\dots,F_r)$ and this conclude the proof since $\bestboundcoreblocke[e+e'] \geq \w(G_1,\dots,G_r)$ and by taking the maximum on $E_1,\dots,E_r$ and~$F_1,\dots,F_r$.
\end{proof}

\begin{corollary}
\label{corollary:lower_bound_bestboundcoreblock}
For any $e \geq 2$ we have $\bestboundcoreblock \geq \bigl\lfloor\frac{e}{2}\bigr\rfloor \bigl\lfloor \frac{r^2}{4}\bigr\rfloor$.
\end{corollary}

\begin{proof}
Since $\bestboundcoreblocke[1] = 0$  we have
$\bestboundcoreblock \geq \bestboundcoreblocke[2\lfloor \frac{e}{2}\rfloor] \geq \lfloor\frac{e}{2}\rfloor\bestboundcoreblocke[2]$ by Lemma~\ref{lemma:superadditive} and we conclude by Proposition~\ref{proposition:bound_e=2}.
\end{proof}


We will now give an upper bound for $\bestboundcoreblock$.   The next proposition is the key of the next results.
 We write $\partse \coloneqq \mathcal{P}(\{1,\dots,e\})$ for the powerset of $\{1,\dots,e\}$, the set of subsets of $\{1,\dots,e\}$. We denote by $\lVert \cdot \rVert_1$ the $1$-norm on $\mathbb{R}^{\partse}$, given by $\lVert x \rVert_1 = \sum_{E \in \partse}\lvert x_E \rvert$ for all $x = (x_E)_{E \in \partse} \in \mathbb{R}^{\partse}$. For any $x, y \in \mathbb{R}^{\partse}$ we write $\langle x, y\rangle \coloneqq x^\transpose y$ for the canonical scalar product. We have $\langle x, x\rangle = \lVert x \rVert^2$, where $\lVert \cdot\rVert$ is the Euclidean norm as in Section~\ref{section:bacgroud_core_blocks}.
 Let $\matq = (a_{E, F})_{E, F \in \partse}$ be the matrix given by $a_{E,F} \coloneqq \w(E,F)$ for all $E, F \in \partse$. The matrix $\matq$ is symmetric of size $2^e$ with all its diagonal entries being $0$ and only non-negative entries. We denote by $\qe$ the quadratic form given by $\qe(x) \coloneqq \frac{1}{2} \langle x,\matq x\rangle$ for all $x \in \mathbb{R}^{\partse}$. 

\begin{proposition}
\label{proposition:boundcoreblock_maxqr}
\phantom{.}
\begin{enumerate}[$(i)$]
\item
\label{item:qe_pos}
If $x \in \N^{\partse}$ then $\qe(x) \in \N$.
\item
We have:
\[
\bestboundcoreblock =\max_{\substack{x \in \N^{\partse} \\ \lVert x\rVert_1 = r}} \qe(x).
\]
\end{enumerate}
\end{proposition}

\begin{proof}
We have a one-to-one correspondence between unordered $r$-tuples $(E_1,\dots,E_r)$ of subsets of $\{1,\dots,e\}$ and elements $x = (x_E)_{E \in \partse} \in \N^{\partse}$ with $\lVert x\rVert_1 = r$ given as follows:
\[
x_E = \#\bigl\{j \in \{1,\dots,r\} : E_j = E\}.
\]
For such elements, by Lemmas~\ref{lemma:weight_sum_r=2} and~\ref{lemma:weight_set_same_empty} we have
\begin{align*}
\w(E_1,\dots,E_r)
&=
\sum_{1\leq j < k \leq r} \w(E_j,E_k)
\\
&=
\sum_{1 \leq j < k \leq r} a_{E_j,E_k}
\\
&=
\frac{1}{2}\sum_{1 \leq j, k \leq r} a_{E_j,E_k}
\\
&=
\frac{1}{2}\sum_{E,F \in \partse} a_{E,F} x_E x_F
\\
&=
\qe(x),
\end{align*}
which concludes the proof.
\end{proof}

We will now study the quadratic form $\qe$. Our first aim is to prove that it suffices to study the restrictions of $\qe$ to the subspaces of $\mathbb{R}^{\partse}$ corresponding to subsets of $\{1,\dots,e\}$ of same cardinalities.

\begin{lemma}
\label{lemma:weight_complement}
Let $E, F \subseteq \{1,\dots,e\}$. We have $\w(E,F) = \w(E^c,F^c)$.
\end{lemma}

\begin{proof}
A simple calculation gives
\begin{align*}
\w(E^c,F^c)
&=
\min(\lvert E^c \rvert, \lvert F^c \rvert)
- \lvert E^c \cap F^c\rvert
\\
&=
\min(e-  \lvert E\rvert, e- \lvert F\rvert)
- \lvert  (E \cup F)^c \rvert
\\
&=
e - \max(\lvert E\rvert,\lvert F\rvert) - e + \lvert E \cup F\rvert
\\
&=
- \max(\lvert E\rvert,\lvert F\rvert) + \lvert E\rvert + \lvert F \rvert - \lvert E\cap F\rvert
\\
&=
\min(\lvert E\rvert, \lvert F\rvert) - \lvert E \cap F\rvert
\\
&=
\w(E,F).
\end{align*}
\end{proof}

\begin{lemma}
\label{lemma:add_elt_diff_w}
Let $E_1,\dots,E_r \subseteq \{1,\dots,e\}$. Let $m\coloneqq \min_{1\leq j \leq r} \lvert E_j\rvert$ and assume that $m < e$. Let $j_0$ such that $\lvert E_{j_0}\rvert = m$ and let $x \in \{1,\dots,e\}\setminus E_{j_0}$. For $j \in \{1,\dots,r\}$, we define
\[
\widetilde{E}_j \coloneqq \begin{cases}
E_j,&\text{if } j \neq j_0,
\\
E_{j_0} \cup \{x\}, &\text{if } j = j_0.
\end{cases}
\]
Then
\begin{multline*}
\w(\widetilde{E}_1,\dots,\widetilde{E}_r) = \w(E_1,\dots,E_r) + \#\bigl\{ j \in \{1,\dots,r\} : \lvert E_j \rvert > m \text{ and } x \notin E_j \bigr\}
\\
-  \#\bigl\{ j \in \{1,\dots,r\} : \lvert E_j \rvert = m \text{ and } x \in E_j\bigr\}.
\end{multline*}
\end{lemma}

\begin{proof}
Let $j \neq k \in \{1,\dots,e\}$. If $j, k \neq j_0$ then $\w(\widetilde{E}_j, \widetilde{E}_k) = \w(E_j, E_k)$. If $k = j_0$ then, recalling that $\mathbf{1}_j$ denotes the characteristic vector of $E_j$,
\begin{align*}
\w(\widetilde{E}_j,\widetilde{E}_{j_0})
&=
\w(E_j, E_{j_0} \sqcup\{x\})
\\
&=
\min(\lvert E_j \rvert, m + 1) - \lvert E_j \cap E_{j_0}\rvert - \mathbf{1}_j(x)
\\
&=
\begin{cases}
\w(E_j,E_{j_0}) - \mathbf{1}_j(x), &\text{if } \lvert E_j \rvert = m,
\\
\w(E_j, E_{j_0}) + 1 - \mathbf{1}_j(x), &\text{if } \lvert E_j \rvert > m.
\end{cases}
\end{align*}
Thus, by Lemma~\ref{lemma:weight_sum_r=2} we have
\begin{align*}
\w(\widetilde{E}_1,\dots,\widetilde{E}_r) - \w(E_1,\dots,E_r)
&=
\sum_{\substack{1 \leq j \leq r \\ j \neq j_0}} \bigl[ \w(\widetilde{E}_j,\widetilde{E}_{j_0}) - \w(E_j,E_{j_0})\bigr]
\\
&=
\#\bigl\{ j \in \{1,\dots, r \} : \lvert E_j \rvert > m\bigr\}
- \#\bigl\{j \in \{1,\dots,r\} : x \in E_j\bigr\}.
\end{align*}
Writing
\begin{multline*}
\bigl\{ j \in \{1,\dots, r \} : \lvert E_j \rvert > m\bigr\}
=
\bigl\{ j \in \{1,\dots,r\} : \lvert E_j \rvert > m \text{ and } x \in E_j \bigr\} 
\\
\sqcup \bigl\{j \in \{1,\dots,r\} : \lvert E_j \rvert >  m \text{ and }x \notin E_j\bigr\},
\end{multline*}
and
\begin{multline*}
\bigl\{j \in \{1,\dots,r\} : x \in E_j\bigr\}
 = \bigl\{ j \in \{1,\dots,r\} : x \in E_j \text{ and }\lvert E_j \rvert > m \bigr\}
 \\
 \sqcup \bigl\{ j \in \{1,\dots,r\}: x \in E_j \text{ and } \lvert E_j \rvert = m\bigr\},
\end{multline*}
gives
\[
\w(\widetilde{E}_1,\dots,\widetilde{E}_r) - \w(E_1,\dots,E_r)
=
\begin{multlined}[t]
\#\bigl\{ j \in \{1,\dots,r\} : \lvert E_j \rvert > m \text{ and } x \notin E_j \bigr\}
\\
-  \#\bigl\{ j \in \{1,\dots,r\} : x \in E_j \text{ and }  \lvert E_j \rvert = m\bigr\},
\end{multlined}
\]
as announced.
\end{proof}

\begin{proposition}
\label{proposition:max_w_equal_size}
There exist $E_1,\dots,E_r \subseteq \{1,\dots,e\}$  with $\lvert E_1 \rvert = \dots = \lvert E_r \rvert$ such that $\w(E_1,\dots,E_r) = \bestboundcoreblock$.
\end{proposition}

\begin{proof}
Let $E_1, \dots, E_r \subseteq \{1,\dots,e\}$. It suffices to prove that we can find $\widetilde{E}_1,\dots,\widetilde{E}_r \subseteq \{1,\dots,e\}$ satisfying $\lvert \widetilde{E}_1\rvert = \dots = \lvert \widetilde{E}_r\rvert$ such that $\w(\widetilde E_1,\dots,\widetilde E_r) \geq \w(E_1,\dots,E_r)$.

Let $m \coloneqq \min_{1\leq j \leq r} \lvert E_j\rvert$ and $M \coloneqq \max_{1 \leq j \leq r} \lvert E_j\rvert$.
If $m = M$ then we are done, thus we now assume that $m < M$. Let $j_m, j_M \in \{1,\dots,e\}$ such that $\lvert E_{j_m}\rvert = m$ and $\lvert E_{j_M} \rvert = M$. We have $e - m + M > e$ by assumption, thus $E_{j_m}^c \cap E_{j_M} \neq \emptyset$, in particular we can pick $x \in E_{j_m}^c \cap E_{j_M}$. We first assume that:
\begin{equation}
\label{equation:induction_nremax_same_size}
\#\bigl\{ j \in \{1,\dots,r\} : \lvert E_j \rvert > m \text{ and } x \notin E_j \bigr\}
\geq  \#\bigl\{ j \in \{1,\dots,r\} : \lvert E_j \rvert = m \text{ and } x \in E_j\bigr\}.
\end{equation}
We define:
\begin{align*}
N_m &\coloneqq \#\bigl\{ j \in \{1,\dots,r\} : \lvert E_j \rvert = m\},
\\
N_M &\coloneqq \#\bigl\{ j \in \{1,\dots,r\} : \lvert E_j \rvert = M\}.
\end{align*}
By Lemma~\ref{lemma:add_elt_diff_w} applied with the family $E_1,\dots,E_r$ and $x \notin E_{j_m}$,  we can construct a family $\widetilde{E_1},\dots,\widetilde{E}_r \subseteq \{1,\dots,e\}$ satisfying $\w(\widetilde E_1,\dots,\widetilde E_r) \geq \w(E_1,\dots,E_r)$ such that either
\begin{align*}
\max_{1 \leq j \leq r } \lvert \widetilde{E}_j \rvert &= M,
\\
\min_{1 \leq j \leq r } \lvert \widetilde{E}_j \rvert &= m + 1,
\\
\intertext{or}
\max_{1 \leq j \leq r } \lvert \widetilde{E}_j \rvert &= M,
\\
\min_{1 \leq j \leq r } \lvert \widetilde{E}_j \rvert &= m,
\\
\#\bigl\{ j \in \{1,\dots,r\} : \lvert \widetilde{E}_j \rvert = m \bigr\} &= N_m - 1,
\\
\#\bigl\{ j \in \{1,\dots,r\} : \lvert \widetilde{E}_j \rvert =M\bigr\} &= N_M,
\end{align*}
thus in both cases we conclude by induction on $(M-m,N_m + N_M) \in \N^2$.

Thus, we now assume that~\eqref{equation:induction_nremax_same_size} fails, that is,
\begin{equation}
\label{equation:induction_nremax_same_size_converse}
\#\bigl\{ j \in \{1,\dots,r\} : \lvert E_j \rvert > m \text{ and } x \notin E_j \bigr\}
<  \#\bigl\{ j \in \{1,\dots,r\} : \lvert E_j \rvert = m \text{ and } x \in E_j\bigr\}.
\end{equation}
Defining $F_j \coloneqq E_j^c$ for all $j \in \{1,\dots,e\}$, we have
\begin{align*}
\#\bigl\{ j \in \{1,\dots,r\} : \lvert F_j \rvert > e-M \text{ and } x \notin F_j \bigr\}
&=
\#\bigl\{ j \in \{1,\dots,r\} : \lvert E_j \rvert < M \text{ and } x \in E_j \bigr\}
\\
&\geq
\#\bigl\{ j \in \{1,\dots,r\} : \lvert E_j \rvert = m \text{ and } x \in E_j \bigr\}
\\
&>
\#\bigl\{ j \in \{1,\dots,r\} : \lvert E_j \rvert > m \text{ and } x \notin E_j \bigr\} \text{ (by~\eqref{equation:induction_nremax_same_size_converse})}
\\
&\geq 
\#\bigl\{ j \in \{1,\dots,r\} : \lvert E_j \rvert = M \text{ and } x \notin E_j \bigr\}
\\
&=
\#\bigl\{ j \in \{1,\dots,r\} : \lvert F_j \rvert = e-M \text{ and } x \in F_j \bigr\}.
\end{align*}
Since we have chosen $x$ such that $x \notin F_{j_M}$ with $\lvert F_{j_M}\rvert = e - M = \min_{1 \leq j \leq r} \lvert F_j\rvert$, we can apply Lemma~\ref{lemma:add_elt_diff_w} to the family $F_1,\dots,F_r$ and $x \notin F_{j_M}$. We thus find a family $\widetilde{F}_1,\dots,\widetilde{F}_r \subseteq \{1,\dots,e\}$ satisfying $\w(\widetilde F_1,\dots,\widetilde F_r) > \w(F_1,\dots,F_r)$, and this concludes the proof since $\w(F_1,\dots,F_r) = \w(E_1,\dots,E_r)$ by Lemmas~\ref{lemma:weight_sum_r=2} and~\ref{lemma:weight_complement}.
\end{proof}

For any $k \in \{0,\dots,e\}$, let $\qek$ be the restriction of $\qe$ to the subspace $\mathbb{R}^{\partsek}$, where $\partsek$ is the subset of $\partse$ given by the subsets $E \subseteq \{1,\dots,e\}$ of size $k$. We have $\qek(x) = \frac{1}{2}\langle x, \matqk x\rangle$ for all $x \in \mathbb{R}^{\partsek}$, where $\matqk = (a_{E,F})_{E, F \in \partsek}$ is the real symmetric matrix of size $\binom{e}{k}$ given by
\[
a_{E,F} = \w(E,F) = k - \lvert E \cap F\rvert,
\]
for all $E, F \in \partsek$. In particular, as for $\matq$, the matrix $\matqk$ has only $0$ entries in the diagonal.
By  Lemmas~\ref{lemma:weight_set_same_empty} and~\ref{lemma:weight_complement} and Proposition~\ref{proposition:max_w_equal_size}, we have
\begin{equation}
\label{equation:bestboundcoreblock_maxmax}
\bestboundcoreblock = \max_{1 \leq k \leq \lfloor\frac{e}{2}\rfloor} \max_{\substack{x \in \N^{\partsek} \\ \lVert x \rVert_1 = r}} \qek(x).
\end{equation}

The matrix $\matqk$ appears at some places in the literature (see, for instance,  \cite{ryser1,ryser2}). 
The result of Lemma~\ref{lemma:spectra_Aeq} is maybe well-known; we provide a proof for convenience.

\begin{lemma}
\label{lemma:spectra_Aeq}
Recall that $e \geq 2$ and let $k \in \{1,\dots,e-1\}$. The eigenvalues of $\matqk$ are the following:
\begin{align*}
k\binom{e-1}{k}&,&&\text{ with multiplicity } 1,
\\
0&,&&\text{ with multiplicity } \binom{e}{k}-e,
\\
-\binom{e-2}{k-1}&,&&\text{ with multiplicity } e-1.
\end{align*}
Moreover, the constant vector $\mathbf{1} \in \mathbb{R}^{\partsek}$ is an eigenvector for the eigenvalue $k \binom{e-1}{k}$.
\end{lemma}

\begin{proof}
We follow the computation of the eigenvalues for adjacency matrices of strongly regular graphs (see, for instance, \cite{godsil-royle}). For any $m, n \geq 1$, we denote by $J_{m,n}$ the $m \times n$ matrix filled with ones, and we define $J_n \coloneqq J_{n,n}$. We also write $I_n$ for the $n \times n$ identity matrix.

We first note that for any $E \in \partsek$ we have, 
\begin{align*}
\sum_{F \in \partsek} a_{E,F}
&= 
\sum_{F \in \partsek} (k-  \lvert E \cap F  \rvert)
\\
&=
\sum_{\ell = 0}^{k-1} (k - \ell)\binom{k}{\ell}\binom{e-k}{k-\ell}
\\
&=
\sum_{\ell = 1}^k \ell \binom{k}{\ell}\binom{e-k}{\ell}
\\
&=
\sum_{\ell = 1}^k k \binom{k-1}{\ell - 1}\binom{e - k}{e-k-\ell},
\end{align*}
thus, using Vandermonde's identity we obtain,
\begin{align*}
\sum_{F \in \partsek} a_{E,F}
&=
k \binom{e-1}{e-k-1}
\\
&= k \binom{e-1}{k},
\end{align*}
so that the constant vector $\mathbf{1}$ is an eigenvector of $\matqk$ with eigenvalue $k\binom{e-1}{k}$.

Now let $B = (b_{iE})_{\substack{1 \leq i \leq e \\ E \in \partsek}}$ be the $e\times \binom{e}{k}$ incidence matrix associated with the elements of $\{1,\dots,e\}$ and $\partsek$. In other words, for any $i \in \{1,\dots,e\}$ and $E \in \partsek$ we have $b_{iE} = \mathbf{1}_E(i)$, where $\mathbf{1}_E$ is the characteristic vector of $E \subseteq \{1,\dots,e\}$. The matrix $B^\transpose B$ is square with rows and columns indexed by $\partsek$ and for any $E, F \in \partsek$ we have
\[
(B^\transpose B)_{EF} = \sum_{i = 1}^e \mathbf{1}_E(i) \mathbf{1}_F(i) = \lvert E \cap F\rvert,
\]
so that $\matqk = k J_{\binom{e}{k}} - B^\transpose B$. We will now compute $\matqk^2$. First, note that the matrix $BB^\transpose$ is square of size $e$ and for any $i, j \in \{1,\dots,e\}$ we have
\[
(BB^\transpose)_{ij} = \sum_{E \in \partsek} \mathbf{1}_E(i) \mathbf{1}_E(j) = \begin{dcases}
\binom{e-1}{k-1}, &\text{if } i = j,
\\
\binom{e-2}{k-2}, &\text{if } i \neq j
\end{dcases}
\]
(with the convention $\binom{e-2}{k-2} = 0$ if $k = 1$),
so that
\[
BB^\transpose = \binom{e-2}{k-2} J_e + \left(\binom{e-1}{k-1} - \binom{e-2}{k-2}\right) I_e = \binom{e-2}{k-2} J_e +  \binom{e-2}{k-1} I_e.
\]
It is clear that $B^\transpose J_e$ is a scalar multiple of $J_{\binom{e}{k},e}$, similarly
\[
J_{\binom{e}{k}}^2,
\quad
J_{\binom{e}{k}} B^\transpose B \text{ and }
J_{\binom{e}{k},e} B,
\]
are scalar multiples of  $J_{\binom{e}{k}}$.
We deduce that
\begin{align}
\matqk^2
&=
\left(k J_{\binom{e}{k}} - B^\transpose B\right)^2,
\notag
\\
&=
\alpha J_{\binom{e}{k}} + B^\transpose (B B^\transpose) B
\notag
\\
&=
\beta J_{\binom{e}{k}} + \binom{e-2}{k-1} B^\transpose B
\notag
\\
&=
\gamma J_{\binom{e}{k}} - \binom{e-2}{k-1}\matqk,
\label{equation:matek2}
\end{align}
for some scalars $\alpha,\beta,\gamma$. Now let $x$ be any eigenvector of $\matqk$ orthogonal to $\mathbf{1}$ and let $\lambda$ be its associated eigenvalue. 
Since $\langle x, \mathbf{1}\rangle = 0$ we have $J_{\binom{e}{k}} x = 0$ and thus~\eqref{equation:matek2} gives
\[
\matqk^2 x = -\binom{e-2}{k-1} \matqk x,
\]
thus $\lambda^2 = - \binom{e-2}{k-1} \lambda$ thus $\lambda = 0$ or $\lambda = - \binom{e-2}{k-1}$. Now if $m$ (respectively $m_0$) denotes the multiplicity of the eigenvalue $-\binom{e-2}{k-1}$ (resp. $0$), what precedes proves that $k\binom{e-1}{k}$ is of multiplicity $1$ and we have
\begin{align*}
1 + m_0 + m &= \binom{e}{k},
\\
k\binom{e-1}{k} - m \binom{e-2}{k-1} &= \mathrm{tr}(\matqk) = 0,
\end{align*}
recalling that the diagonal entries of $\matqk$ are $0$. By Lemma~\ref{lemma:binomial_coefficients} we find that $m = k \binom{e-1}{k} / \binom{e-2}{k-1} = e-1$ and $m_0 = \binom{e}{k} - e$, which concludes the proof.
\end{proof}

\begin{proposition}
\label{proposition:upper_bonud_qek}
Let $k \in \{1,\dots,e-1\}$.
For any $x \in \mathbb{R}^{\partsek}$ we have
\[
\qek(x) \leq \frac{\lVert x \rVert_1^2}{2e}k(e-k).
\]
\end{proposition}

\begin{proof}
For any $x \in \mathbb{R}^{\partsek}$ we have
\[
x = \frac{\langle x, \mathbf{1}\rangle}{\lVert \mathbf 1 \rVert^2} \mathbf{1} + x_\perp,
\]
where $x_\perp \in \mathbb{R}^{\partsek}$ is orthogonal to $\mathbf{1}$.  By Lemma~\ref{lemma:spectra_Aeq}, the vector $x_\perp$ is a sum of eigenvectors for $\matqk$ with non-positive eigenvalues, thus $\qek(x_\perp) \leq 0$ and:
\[
\qek(x) = \frac{\langle x, \mathbf{1}\rangle^2}{\lVert \mathbf{1} \rVert^4} \qek(\mathbf{1}) + \qek(x_\perp) \leq \frac{\langle x, \mathbf{1}\rangle^2}{\lVert \mathbf{1} \rVert^4} \qek(\mathbf{1}) .
\]
Now $\lvert \langle x, \mathbf{1}\rangle\rvert \leq \lVert x \rVert_1$ by the triangle inequality and, again by Lemma~\ref{lemma:spectra_Aeq},
\[
\qek(\mathbf{1}) = \frac{1}{2}\langle \mathbf{1},\matqk\mathbf{1}\rangle = \frac{k}{2}\binom{e-1}{k} \lVert \mathbf 1 \rVert^2.
\]
Hence, we have
\[
\qek(x) \leq \frac{\lVert x \rVert_1^2}{\lVert \mathbf{1}\rVert^2} \frac{k}{2}\binom{e-1}{k}
=\frac{\lVert x \rVert_1^2}{\binom{e}{k}} \frac{k}{2}\binom{e-1}{k}.
\]
We obtain the announced result since $\frac{\binom{e-1}{k}}{\binom{e}{k}} = \frac{e-k}{e}$ by Lemma~\ref{lemma:binomial_coefficients}.
\end{proof}

\begin{corollary}
\label{corollary:sharpest_upper_bound_bestboundcoreblock}
We have
\[
\bestboundcoreblock \leq \frac{r^2}{2e}\left\lfloor\frac{e^2}{4}\right\rfloor.
\]
\end{corollary}

\begin{proof}
For any $k \in \{1,\dots,e-1\}$, by Proposition~\ref{proposition:upper_bonud_qek} we know that $\qek(x) \leq \frac{r^2}{2e}k(e-k)$ if $\lVert x\rVert_1 = r$. By~\eqref{equation:bestboundcoreblock_maxmax} this concludes the proof since $k(e-k) \leq \bigl\lfloor\frac{e^2}{4}\bigr\rfloor$ by Lemma~\ref{lemma:binomial_coefficients}\ref{item:k(e-k)}.
\end{proof}

\begin{definition}
\label{definition:idealbound}
For any $r \geq 0$ and $e \geq 1$ we define $\idealbound{r}{e} \coloneqq \Bigl\lfloor\frac{r^2}{2e}\bigl\lfloor\frac{e^2}{4}\bigr\rfloor\Bigr\rfloor$.
\end{definition}

Combining Corollaries~\ref{corollary:lower_bound_bestboundcoreblock} and~\ref{corollary:sharpest_upper_bound_bestboundcoreblock} yields the following final bounds.

\begin{corollary}
\label{corollary:encadrement_final}
For any $e, r \geq 2$ we have
\[
\left\lfloor\frac{e}{2}\right\rfloor \left\lfloor\frac{r^2}{4}\right\rfloor
\leq
\bestboundcoreblock
\leq 
\idealbound r e \leq \frac{er^2}{8}
 \]
 \end{corollary}
 
We have the following particular case.

\begin{proposition}
\label{proposition:bbcb_both_even}
Assume that $e, r \geq 2$ are both even. Then
\[
\bestboundcoreblock = \idealbound r e = \frac{er^2}{8}.
\]
\end{proposition}

We end this part by the following result, which will be useful later.

\begin{lemma}
\label{lemma:idealbound_lambda}
For any $r \geq 0$ and $e \geq 1$ we have:
\[\idealbound r e =  \left\lfloor \frac{r^2}{2\binom{e}{k}}\lambda\right\rfloor,\]
where $k \coloneqq \lfloor\frac{e}{2}\rfloor$ and $\lambda \coloneqq k\binom{e-1}{k}$ (the largest eigenvalue of $\matqk$).
\end{lemma}

\begin{proof}
By Lemma~\ref{lemma:binomial_coefficients}\ref{item:binomial} we have $\frac{\binom{e-1}{k}}{\binom{e}{k}} = \frac{e-k}{e}$. Thus, we have:
\begin{align*}
\frac{r^2}{2\binom{e}{k}}\lambda
&=
\frac{r^2}{2\binom{e}{k}}k\binom{e-1}{k}.
\\
&=
\frac{r^2 }{2e}k(e-k)
\\
&=
\frac{r^2}{2e}\left\lfloor\frac{e^2}{4}\right\rfloor,
\end{align*}
where the last equality follows from Lemma~\ref{lemma:binomial_coefficients}\ref{item:k(e-k)}, and this concludes the proof.
\end{proof}

\subsubsection{More values for small parameters}
\label{subsubsection:more_values}

Let $r, e \geq 2$.
We now aim to compute the value of $\bestboundcoreblock$ for $r \in \{3,4\}$ or $e \in \{3,\dots,6\}$. For these values, we will see that, except for $r = 3$, we have $\bestboundcoreblock = \idealbound r e$ (see Proposition~\ref{proposition:upper_bound_exact_small_cases}).

\paragraph{With small \texorpdfstring{$r$}{r}}

We begin by the following result, which will be in fact interesting only in the case $r = 3$.

\begin{lemma}
\label{lemma:borne_r_geq_3}
If $r \geq 3$ then $\bestboundcoreblock \leq \lfloor \frac{(r-1)re}{6}\rfloor$.
\end{lemma}

\begin{proof}
Let $E_1,\dots,E_r \subseteq \{1,\dots,e\}$ and define $a_{jk} \coloneqq \w(E_j,E_k) = \min(\lvert E_j \rvert, \lvert E_k\rvert) - \lvert E_j \cap E_k\rvert$ for all $1 \leq j < k \leq r$. By~\eqref{equation:weight_binary_matrix} and since $a_{jk} \in \N$, it suffices to show that
\[
\sum_{1 \leq j < k \leq r} a_{jk} \leq \frac{(r-1)re}{6}.
\]
Let $1 \leq j < k < l \leq r$; note that this is possible since $r \geq 3$. We have
\begin{align*}
a_{jk} + a_{kl} + a_{jl}
&=
\begin{multlined}[t]
\min(\lvert E_j \rvert, \lvert E_k\rvert)
+
\min(\lvert E_k \rvert, \lvert E_l\rvert) 
+
\min(\lvert E_j \rvert, \lvert E_l\rvert) 
\\
 - \lvert E_j \cap E_k\rvert
- \lvert E_k \cap E_l\rvert - \lvert E_j \cap E_l\rvert
\end{multlined}
\\
&\leq
\lvert E_j \rvert + \lvert E_k \rvert + \lvert E_l \rvert - \lvert E_j \cap E_k\rvert
- \lvert E_k \cap E_l\rvert - \lvert E_j \cap E_l\rvert
\\
&\leq
\lvert E_j \cup E_k \cup E_l \rvert - \lvert E_j \cap E_k \cap E_l \rvert
\\
&\leq
\lvert E_j \cup E_k \cup E_l \rvert
\\
&\leq e.
\end{align*}
Now, in the sum $\sum_{1 \leq j < k < l \leq r} (a_{jk}+a_{kl} + a_{jl})$ each couple $(j,k)$ for $1 \leq j < k \leq r$ appears exactly $r-2$ times, thus
\[
\sum_{1 \leq j < k < l \leq r} (a_{jk}+a_{kl} + a_{jl}) = (r-2) \sum_{1 \leq j < k \leq r} a_{jk}.
\]
Thus, we find that
\[
(r-2)\sum_{1 \leq j < k \leq r} a_{jk} \leq e\binom{r}{3},
\]
and thus
\[
\sum_{1 \leq j < k \leq r} a_{jk} \leq e\frac{r(r-1)}{6},
\]
whence the result.
\end{proof}

\begin{proposition}[Case $r = 3$]
\label{proposition:bound_r=3}
We have $\bestboundcoreblock[3] = e$.
\end{proposition}

\begin{proof}
By Corollary~\ref{corollary:lower_bound_bestboundcoreblock} and Lemma~\ref{lemma:borne_r_geq_3} we have
\begin{equation}
\label{equation:bounds_r=3}
2\left\lfloor\frac{e}{2}\right\rfloor \leq \bestboundcoreblock[3] \leq e,
\end{equation}
in particular $\bestboundcoreblock[3] = e$ if $e$ is even. (Note that if we use Corollary~\ref{corollary:lower_bound_bestboundcoreblock} instead of Lemma~\ref{lemma:borne_r_geq_3} 
we only obtain $\bestboundcoreblock \leq \frac{9e}{8}$.)  Thus, we can now assume that $e$ is odd and let us write $e = 2e' + 3$ with $e' \geq 0$. We first assume that $e' \geq 1$. By Lemma~\ref{lemma:superadditive} we have
\[
\bestboundcoreblockgen{3}{2e'} + \bestboundcoreblockgen{3}{3}  \leq \bestboundcoreblockgen{3}{e}.
\]
Since $2e'$ is even we have $\bestboundcoreblockgen{3}{2e'} = 2e'$, thus  we obtain, using~\eqref{equation:bounds_r=3} for the right inequality,
\[
e-3 + \bestboundcoreblockgen{3}{3} \leq \bestboundcoreblockgen{3}{e} \leq e.
\]
Thus, to conclude it suffices to prove that $\bestboundcoreblockgen{3}{3} = 3$. We already know by~\eqref{equation:bounds_r=3} that $\bestboundcoreblockgen{3}{3} \leq 3$, thus it suffices to find $E_1,E_2,E_3 \subseteq \{1,2,3\}$ such that $\w(E_1,E_2,E_3) = 3$. Using~\eqref{equation:weight_binary_matrix}, this equality is satisfied with $E_i \coloneqq \{i\}$ for any $i \in \{1,2,3\}$ and this concludes the proof.
\end{proof}

\begin{proposition}[Case $r = 4$]
\label{proposition:bound_r=4}
We have:
\[
\bestboundcoreblock[4] = \begin{cases}
2e,&\text{if } e \text{ is even},
\\
2e-1, &\text{if } e \text{ is odd}.
\end{cases}
\]
\end{proposition}

\begin{proof}
If $e$ is even then we know by Proposition~\ref{proposition:bbcb_both_even} that $\bestboundcoreblock[4] = \frac{e4^2}{8} = 2e$. We thus now assume that $e$ is odd. By Corollary~\ref{corollary:sharpest_upper_bound_bestboundcoreblock} we obtain
\[
\bestboundcoreblock[4] \leq \frac{8}{e}\left\lfloor\frac{e^2}{4}\right\rfloor < \frac{8}{e}\frac{e^2}{4} = 2e,
\]
thus $\bestboundcoreblock[4] \leq 2e-1$. If $e > 3$, as in the proof of Proposition~\ref{proposition:bound_r=3}, using Lemma~\ref{lemma:superadditive} we have
\[
\bestboundcoreblockgen{4}{e-3} + \bestboundcoreblockgen{4}{3} \leq \bestboundcoreblock[4],
\]
thus, since $e-3$ is even,
\[
2e-6 + \bestboundcoreblockgen{4}{3} \leq \bestboundcoreblock[4].
\]
Hence, to conclude it suffices to prove that $\bestboundcoreblockgen{4}{3} = 5$. We have seen that $\bestboundcoreblockgen{4}{3} \leq 2\cdot 3 - 1 = 5$, thus it suffices to find $E_1,\dots,E_4 \subseteq \{1,\dots,e\}$ with $\w(E_1,\dots,E_4) = 5$. Setting $E_i \coloneqq \{i\}$ for all $i \in \{1,2,3\}$ and $E_4 \coloneqq E_3$, we have
\[
\w(E_j,E_k) = \begin{cases}
1,&\text{if } j \in \{1,2\},
\\
0, &\text{otherwise},
\end{cases}
\]
for any $1 \leq j < k \leq 4$, thus $\w(E_1,\dots,E_4) = \sum_{1 \leq j < k \leq 4} \w(E_j,E_k) = 3 + 2 = 5$. This concludes the proof.
\end{proof}

\paragraph{With small \texorpdfstring{$e$}{e}}

The next results will mainly consist in using the fact that $\mathbf{1}$ is an eigenvector of $\matqk$.

\begin{proposition}[Case $e = 3$]
\label{proposition:bound_e=3}
We have $\bestboundcoreblocke[3] = \bigl\lfloor\frac{r^2}{3}\bigr\rfloor$.
\end{proposition}


\begin{proof}
By Corollary~\ref{corollary:sharpest_upper_bound_bestboundcoreblock} we have $\bestboundcoreblocke[3] \leq \frac{r^2}{6}\bigl\lfloor\frac{9}{4}\bigr\rfloor = \frac{r^2}{3}$, hence $\bestboundcoreblocke[3] \leq \bigl\lfloor\frac{r^2}{3}\bigr\rfloor$. Thus, by~\eqref{equation:bestboundcoreblock_maxmax} it suffices to find $x \in \N^{\partsek[1]}$ with $\lVert x \rVert_1 = r$ such that $\qek[1](x) = \bigl\lfloor\frac{r^2}{3}\bigr\rfloor$. To that extent, write $r = 3a + h$ with $a \geq 1$ and $h \in \{-1,0,1\}$ and consider $x \coloneqq a\mathbf{1} + h\mathbf{1}_E \in \Z^{\partsek[1]}$, where $\mathbf{1}_E$ is any characteristic vector of an element of $\partsek[1]$. Note that $x \in \N^{\partsek[1]}$ and $\lVert x \rVert_1 = a \binom{3}{1} + h = 3a+h=r$ indeed. We have, recalling that $a_{E,E} = 0$ and from Lemma~\ref{lemma:spectra_Aeq} that $\matqk[1]\mathbf{1} = \binom{2}{1}\mathbf{1} =  2\,\mathbf{1}$,
\begin{align*}
\qek[1](x)
&=
\frac{1}{2}\langle x,\matqk[1]x\rangle
\\
&=
\frac{1}{2}\left(a^2 \langle\mathbf{1},\matqk[1]\mathbf{1}\rangle + h^2 \langle \mathbf{1}_E,\matqk[1]\mathbf{1}_E\rangle + 2 ah \langle \mathbf{1},\matqk[1]\mathbf{1}_E\rangle\right)
\\
&=
\frac{1}{2}\left(2a^2 \lVert \mathbf{1}\rVert^2 + 4ah\right)
\\
&=
3a^2 + 2ah
\\
&=
\frac{r^2-h^2}{3}
\\
&=
\begin{dcases}
\frac{r^2}{3},&\text{if } r = 0\pmod{3},
\\
\frac{r^2-1}{3},&\text{if } r = \pm 1 \pmod{3},
\end{dcases}
\\
&=
\left\lfloor\frac{r^2}{3}\right\rfloor.
\end{align*}
This concludes the proof.
\end{proof}

\begin{proposition}[Case $e \in \{4, 6\}$]
\label{proposition:bound_e=4_6}
We have $\bestboundcoreblocke[4] = \bigl\lfloor\frac{r^2}{2}\bigr\rfloor$ and $\bestboundcoreblocke[6] = \bigl\lfloor\frac{3r^2}{4}\bigr\rfloor=  3\bigl\lfloor\frac{r^2}{4}\bigr\rfloor$.
\end{proposition}

\begin{proof}
If $r$ is even then the result is true by Proposition~\ref{proposition:bbcb_both_even}. Now if $r$ is odd, by Corollary~\ref{corollary:encadrement_final} we have
\begin{gather*}
2\left\lfloor\frac{r^2}{4}\right\rfloor \leq \bestboundcoreblocke[4] \leq \left\lfloor\frac{r^2}{2}\right\rfloor,
\\
3\left\lfloor\frac{r^2}{4}\right\rfloor  \leq \bestboundcoreblocke[6] \leq \left\lfloor\frac{3r^2}{4}\right\rfloor.
\end{gather*}
Now we have
\begin{align*}
\left\lfloor\frac{r^2}{2}\right\rfloor &= \frac{r^2-1}{2},
&
\left\lfloor\frac{r^2}{4}\right\rfloor &= \frac{r^2-1}{4},
\end{align*}
which concludes the proof for $e = 4$. For $e = 6$ we note that
\[
3\left\lfloor\frac{r^2}{4}\right\rfloor = 3 \frac{r^2-1}{4} = \frac{3r^2}{4} - \frac{3}{4} = \left\lfloor\frac{3r^2}{4}\right\rfloor,
\]
which concludes the proof.
\end{proof}

We will now deal with the case $e = 5$.

\begin{definition}
For any $r \geq 0$ and $e \geq 2$, we define
\[
\qbound r e \coloneqq \max_{\substack{x \in \N^{\partsek} \\ \lVert x \rVert_1 = r}} \qek(x),
\]
where $k \coloneqq \bigl\lfloor\frac{e}{2}\bigr\rfloor$.
\end{definition}

Recalling~\eqref{equation:bestboundcoreblock_maxmax} and Corollary~\ref{corollary:encadrement_final}, for any $r \geq 2$ and $e \geq 2$ we have
\begin{equation}
\label{equation:QleqNleqN'}
\qbound r e \leq \bestboundcoreblock \leq \idealbound r e.
\end{equation}

\begin{remark}
\label{remark:qbound=bestboundcoreblock}
We will see  that when $e = 5$ then $\qbound r e = \bestboundcoreblock$ (Lemma~\ref{lemma:Nre=maxq_e=5}). It is unclear whether this equality holds in full generality, in particular we state the next results with the quantity $\qbound r e$.
\end{remark}

Recall from Definition~\ref{definition:idealbound} that $\idealbound r e = \bigl\lfloor \frac{r^2}{2e}\bigl\lfloor\frac{e^2}{4}\bigr\rfloor\bigr\rfloor$.

\begin{lemma}
\label{lemma:idealbound_shift}
Let $r \geq 0$ and $e \geq 2$. Assume that $\qbound r e = \idealbound r e$. Then $\qbound{r+\binom e k} e = \idealbound{r+\binom{e}{k}}{e}$,  where $k \coloneqq \bigl\lfloor\frac{e}{2}\bigr\rfloor$.
\end{lemma}

\begin{proof}
By assumption, we can find $x \in \N^{\partsek} $ with $ \lVert x \rVert_1 = r$ such that $\qek(x)=   \idealbound r e  $ (with $k =\bigl\lfloor\frac{e}{2}\bigr\rfloor$).   We have,  where $\lambda \coloneqq k \binom{e-1}{k}$ is the highest eigenvalue of $\matqk$ (see Lemma~\ref{lemma:spectra_Aeq}),
\begin{align*}
\qek(x + \mathbf{1}) 
&= 
\qek(x) + \qek(\mathbf{1}) + \langle x, \matqk \mathbf{1}\rangle
\\
&=
\qbound r e + \frac{\lambda}{2}\binom{e}{k} + \lambda r,
\end{align*}
thus $\qbound{r+\binom{e}{k}} e \geq \qbound r e + \frac{\lambda}{2}\binom{e}{k} + \lambda r$. Note that $\frac{\lambda}{2}\binom{e}{k} = \qek(\mathbf{1})$ is an integer by Proposition~\ref{proposition:boundcoreblock_maxqr}\ref{item:qe_pos}.  Now, recalling the equality $\idealbound r e =  \left\lfloor \frac{r^2}{2\binom{e}{k}}\lambda\right\rfloor$ from Lemma~\ref{lemma:idealbound_lambda} we find:
\begin{align*}
\idealbound {r + \binom{e}{k}}{e}
&=
\left\lfloor \frac{r^2 +  \binom{e}{k}^2 + 2\binom{e}{k}r}{2\binom{e}{k}}\lambda\right\rfloor
\\
&=
\left\lfloor \frac{r^2}{2\binom{e}{k}}\lambda +  \frac{\lambda}{2}\binom{e}{k} + \lambda r\right\rfloor
\\
&=
\left\lfloor \frac{r^2}{2\binom{e}{k}}\lambda\right\rfloor +\frac{\lambda}{2}\binom{e}{k} + \lambda r.
\end{align*}

Finally, we have
\begin{align*}
\idealbound{r+\binom{e}{k}}{e}
&
\geq \qbound{r+\binom{e}{k}}e
\\
&= \qbound r e + \left(\qbound{r+\binom{e}{k}}e- \qbound r e\right)
\\
&\geq \idealbound r e + \left(\idealbound{r+\binom{e}{k}}{e} - \idealbound r e\right)
\\
&=  \idealbound{r+\binom{e}{k}}{e},
\end{align*}
which concludes the proof.
\end{proof}

\begin{lemma}
\label{lemma:idealbound_antishift}
Let $r \geq 0$ and $e \geq 2$. Let  $k \coloneqq \bigl\lfloor\frac{e}{2}\bigr\rfloor$ and   assume that there exists $x \in \{0,1\}^{\partsek}$ with $\lVert x \rVert_1 = r$ such that $\qek(x) = \idealbound r e$  (in particular, we have $\qbound r e = \idealbound r e$). Then $\qbound{\binom{e}{k}-r}e= \idealbound{\binom{e}{k}-r}{e}$.
\end{lemma}

\begin{proof}
Note that $\mathbf{1}-x \in \N^{\partsek}$ since the entries of $x$ are at most $1$. 
With similar calculations as in Lemma~\ref{lemma:idealbound_shift}, we thus have:
\begin{align*}
\qbound{\binom{e}{k} - r}{e} 
&\geq
\qek(\mathbf{1}-x)
\\
&=
\qek(x) + \qek(\mathbf{1}) - \langle x, \matqk\mathbf{1}\rangle
\\
&=
\qbound r e + \frac{\lambda}{2}\binom{e}{k} - \lambda r
\\
&=
\idealbound r e + \frac{\lambda}{2}\binom{e}{k} - \lambda r
\\
&=
\idealbound{\binom{e}{k} - r} e,
\end{align*}
thus we conclude by~\eqref{equation:QleqNleqN'}.
\end{proof}

\begin{remark}
\label{remark:e_leq_7}
Assume that $r = 1$. We have $\idealbound 1 e \leq \frac{e}{8}$ by Corollary~\ref{corollary:encadrement_final}, thus if $e \leq 7$ we have $\idealbound 1 e = 0$. In particular, the assumption $\qbound r e = \idealbound r e$ of Lemma~\ref{lemma:idealbound_shift} is satisfied. Moreover, since any  $x \in \N^{\partsek}$ with $\lVert x \rVert_1 = 1$ has entries in $\{0,1\}$, the assumption of Lemma~\ref{lemma:idealbound_antishift} is also satisfied (using~\eqref{equation:QleqNleqN'}).
\end{remark}

%

We give a last preliminary result, which is  a refinement of~\eqref{equation:bestboundcoreblock_maxmax} (see Remark~\ref{remark:qbound=bestboundcoreblock}).

\begin{lemma}
\label{lemma:Nre=maxq_e=5}
Assume that $e = 5$. For any $r \geq 2$ we have $\qbound r e = \bestboundcoreblock$.
\end{lemma}

\begin{proof}
By~\eqref{equation:bestboundcoreblock_maxmax}, Proposition~\ref{proposition:upper_bonud_qek} and Corollary~\ref{corollary:lower_bound_bestboundcoreblock}, it suffices to prove that
\[
\left\lfloor\frac{r^2}{2e}k'(e-k')\right\rfloor <
\left\lfloor\frac{e}{2}\right\rfloor\left\lfloor\frac{r^2}{4}\right\rfloor,
\]
for all $k' \in \bigl\{1,\dots,\bigl\lfloor \frac{e}{2}\bigr\rfloor-1\bigr\}$. In our setting, this reduces to proving
\begin{equation}
\label{equation:ineg_floors}
\left\lfloor\frac{2r^2}{5}\right\rfloor < 2\left\lfloor\frac{r^2}{4}\right\rfloor.
\end{equation}
If $r$ is even, we have $\bigl\lfloor\frac{2r^2}{5}\bigr\rfloor\leq \frac{2r^2}{5} < \frac{r^2}{2} = 2 \bigl\lfloor\frac{r^2}{4}\bigr\rfloor$ thus~\eqref{equation:ineg_floors} holds. If $r$ is odd we have $\bigl\lfloor\frac{r^2}{4}\bigr\rfloor = \frac{r^2-1}{4}$ and
\[
\frac{2r^2}{5} < \frac{r^2-1}{2}
\iff
4r^2 < 5r^2 - 5,
\]
thus~\eqref{equation:ineg_floors} holds since $r \geq 2$ is odd.
\end{proof}

\begin{proposition}[Case $e = 5$]
\label{proposition:bound_e=5}
For any $r \geq 2$ we have $\bestboundcoreblocke[5]
= \idealbound{r}{5} 
=  \bigl\lfloor\frac{3r^2}{5}\bigr\rfloor$.
\end{proposition}

\begin{proof}
By Lemma~\ref{lemma:idealbound_shift} and Lemma~\ref{lemma:Nre=maxq_e=5}, it suffices to prove the equality for $2 \leq r < \binom{5}{2}=10$, together with $\qbound r 5 = \idealbound r 5$ for $r \in \{0,1\}$. For $r = 0$ we have $\qbound r 5 = \idealbound r 5 = 0$,  and the equality also holds for $r= 1$ by Remark~\ref{remark:e_leq_7}. By Lemma~\ref{lemma:idealbound_antishift} and again Remark~\ref{remark:e_leq_7}, we thus have $\bestboundcoreblocke[5]   = \idealbound r 5$ for $r = 9$.

It remains to prove  $\bestboundcoreblocke[5] = \idealbound r 5$ for $r \in \{2,\dots,8\}$.
 It follows from Propositions~\ref{proposition:bound_r=2} and~\ref{proposition:bound_r=3} that the equality $\bestboundcoreblocke[5]=\idealbound r 5$ holds for $r \in \{2,3\}$. Moreover, the proofs of these propositions guarantee that Lemma~\ref{lemma:idealbound_antishift} can be applied for $r \in \{2,3\}$, so that the equality $\bestboundcoreblocke[5]=\idealbound r 5$ also holds for $r \in \{7,8\}$ (using Lemma~\ref{lemma:Nre=maxq_e=5}). The same can be said for $r = 4$ (Proposition~\ref{proposition:bound_r=4}) and thus $r = 6$, however we need to carefully check that the conditions of Lemma~\ref{lemma:idealbound_antishift} can be satisfied when $r = 4$. Following the proofs of Propositions~\ref{proposition:bound_e=2} and~\ref{proposition:bound_r=4}, taking
\begin{gather*}
E_1 \coloneqq \{1,3\}, \quad E_2 \coloneqq \{2,4\},
\\
E_3 \coloneqq \{1,5\}, \quad E_4 \coloneqq \{2,5\},
\end{gather*} 
by~\eqref{equation:weight_binary_matrix} we have $\w(E_1,\dots,E_4) = \sum_{1 \leq j < k \leq 4} \bigl(2 - \lvert E_j \cap E_k\rvert\bigr) = 9 =  \bestboundcoreblockgen 4 5$ indeed and the corresponding $x \in \N^{\partsekgen 5 2}$ has entries in $\{0,1\}$. We could have also noticed that $\idealbound 6 5 = \lfloor\frac{3\cdot 36}{5}\rfloor = \lfloor\frac{3\cdot 35}{5} + \frac{3}{5}\rfloor = 21$, and by Lemma~\ref{lemma:superadditive} we have
\[
\bestboundcoreblockgen 6 5 \geq \bestboundcoreblockgen 6 3 + \bestboundcoreblockgen 6 2,
\]
thus by Propositions~\ref{proposition:bound_e=2} and~\ref{proposition:bound_e=3} we have
\[
\bestboundcoreblockgen 6 5 \geq \left\lfloor\frac{6^2}{3}\right\rfloor + \left\lfloor\frac{6^2}{4}\right\rfloor = 2\cdot 6 + 3^2 = 21,
\]
thus $\bestboundcoreblockgen 6 5 = \idealbound 6 5$ by Corollary~\ref{corollary:encadrement_final} (note that this alternative proof does not work for $r \in \{7,8\}$).

It thus remains to treat the case $r = 5$. We define:
\begin{gather*}
E_1 \coloneqq \{1,2\}, \quad E_2 \coloneqq \{1,3\},
\\
\\
E_5 \coloneqq \{2,4\}, \quad\quad\quad\quad E_3 \coloneqq \{3,5\},
\\
\\
E_4 \coloneqq \{4,5\},
\end{gather*}
so that $\lvert E_j \cap E_{j+1}\rvert = 1$ for all $j \in \{1,\dots,5\}$ (with $E_6 \coloneqq E_1$), every other pairwise intersection being empty. By~\eqref{equation:weight_binary_matrix}, we deduce that
\[
\w(E_1,\dots,E_5) = 5 \cdot 1 + 5 \cdot 2  = 15,
\]
thus $\bestboundcoreblockgen{5}{5} \geq 15$. Since $\idealbound 5 5 = \bigl\lfloor\frac{3\cdot 5^2}{5}\bigr\rfloor = 15$, by Corollary~\ref{corollary:encadrement_final} we obtain that  $\bestboundcoreblockgen{5}{5} = \idealbound 5 5$ and this concludes the proof.
\end{proof}

\paragraph{All in one}

We now gather all the results we have proven in these small cases. Recall that $\bestboundcoreblock$ (resp. $\idealbound r e$) is defined in Definition~\ref{definition:bestboundcoreblock} (resp. Definition~\ref{definition:idealbound}).

\begin{proposition}
\label{proposition:upper_bound_exact_small_cases}
Let $e, r \geq 2$. We have
\[
\bestboundcoreblock
=
\idealbound r e,
 \]
  in (at least) the following  cases:
\[
r \in \{2,4\}
\quad \text{or}
\quad e \in \{2,\dots,6\}.
\]
\end{proposition}

Note that if $r = 3$ then $\bestboundcoreblock[3] = e$ by Proposition~\ref{proposition:bound_r=3}, however if for instance $e$ is even we have $\idealbound 3 e = \bigl\lfloor\frac{9e}{8}\bigr\rfloor = e + \bigl\lfloor\frac{e}{8}\bigr\rfloor > \bestboundcoreblock[3]$ as soon as $e \geq 8$.

\begin{proof}
First, note that the result holds if both $r$ and $e$ are even, by Proposition~\ref{proposition:bbcb_both_even}. We have
\[
\idealbound r  e =\begin{dcases}
\left\lfloor\frac{r^2e}{8}\right\rfloor, &\text{if } e \text{ is even},
\\
\left\lfloor\frac{r^2(e^2-1)}{8e}\right\rfloor, &\text{if } e \text{ is odd}.
\end{dcases}
\]
For $r = 2$, we deduce that (with $e \geq 2$ odd)
\[
\idealbound 2 e
=
\left\lfloor\frac{e^2-1}{2e}\right\rfloor
=
\left\lfloor\frac{e}{2} - \frac{1}{2e}\right\rfloor
=
\left\lfloor\frac{e}{2}\right\rfloor,
\]
thus $\idealbound 2 e = \bestboundcoreblock[2]$ by Proposition~\ref{proposition:bound_r=2}.
For $r = 4$, we have (with $e \geq 2$ odd)
\[
\idealbound 4 e = \left\lfloor\frac{2(e^2-1)}{e}\right\rfloor
=
2e - \left\lfloor\frac{2}{e}\right\rfloor = 2e-1,
\]
thus $\idealbound 4 e = \bestboundcoreblock[4]$ by Proposition~\ref{proposition:bound_r=4}. The equalities $\idealbound r e = \bestboundcoreblock$ for $e \in \{2,\dots,6\}$ are straightforward from Propositions~\ref{proposition:bound_e=2}, \ref{proposition:bound_e=3}, \ref{proposition:bound_e=4_6} and \ref{proposition:bound_e=5}.
\end{proof}
 
\subsubsection{Asymptotics}

In this section, we will describe the asymptotic behaviour  of the two sequences $(\bestboundcoreblock)_{e \geq 2}$ and $(\bestboundcoreblock)_{r \geq 2}$ for $r, e \geq 2$.
We first recall the following basic fact, which will be used without any further reference.

\begin{lemma}
Let $(u_n)_{n\geq 1}$ be a sequence of real numbers that converges to $+\infty$ and let $x \in \mathbb{R}$. Then the sequence $\bigl(\frac{\lfloor x u_n\rfloor}{u_n}\bigr)_{n\geq 1}$ converges to $x$.
\end{lemma}

We also recall the following standard result.

\begin{lemma}[Fekete's superadditive lemma]
\label{lemma:superadditive_lemma}
Let $(u_n)_{n\geq 1}$ be a superadditive sequence. Then $\bigl(\frac{u_n}{n}\bigr)_{n\geq 1}$ has a limit in $[-\infty,+\infty]$ as $n \to +\infty$.
\end{lemma}

By Lemmas~\ref{lemma:superadditive} and~\ref{lemma:superadditive_lemma}, we know that the sequence $\left(\frac{\bestboundcoreblock}{e}\right)_{e \geq 1}$ has a limit $\bestboundcoreblocke[\infty] \in [0,+\infty]$. By Corollary~\ref{corollary:encadrement_final}  we have
\begin{equation}
\label{equation:e_infinite}
\frac{1}{2}\left\lfloor\frac{r^2}{4}\right\rfloor\leq \bestboundcoreblocke[\infty] \leq \frac{r^2}{8},
\end{equation}
in particular if $r$ is even then $\bestboundcoreblocke[\infty] = \frac{r^2}{8}$.

\begin{remark}
By Propositions~\ref{proposition:bound_r=2}, \ref{proposition:bound_r=3} and \ref{proposition:bound_r=4},  we have $\bestboundcoreblockgen 2 \infty = \frac{1}{2}$, $\bestboundcoreblockgen 3 \infty = 1$ and $\bestboundcoreblockgen 4 \infty = 2$.
\end{remark}

Our aim is now to give a similar result when $r$ grows to infinity.

\begin{proposition}
\label{proposition:convergence_Nre_maxq}
The sequence $\left(\frac{\bestboundcoreblock}{r^2}\right)_{r \geq 2}$ converges to $\bestboundcoreblock[\infty] \coloneqq \max_{\substack{y \in \mathbb{R}^{\partse} \\ \lVert y \rVert_1 = 1}} \qe(y)$.
\end{proposition}

\begin{proof}
By Corollary~\ref{corollary:encadrement_final}, we know that the sequence $\bigl(\bestboundcoreblock/r^2\bigr)_{r \geq 2}$ is bounded. Hence, to prove that it converges to $\bestboundcoreblock[\infty]$ it suffices to prove that any converging subsequence converges to $\bestboundcoreblock[\infty]$. We thus consider a converging subsequence of $\bigl(\bestboundcoreblock/r^2\bigr)_{r \geq 2}$ and let $\ell \in \mathbb{R}$ be its limit. To avoid the notation being overloaded,  we still denote this subsequence by $\bigl(\bestboundcoreblock/r^2\bigr)_{r \geq 2}$. By Proposition~\ref{proposition:boundcoreblock_maxqr}, we have $\bestboundcoreblock / r^2 \leq \bestboundcoreblock[\infty]$ for any $r \geq 2$ thus $\ell \leq \bestboundcoreblock[\infty]$. We will now prove the reverse inequality.

Let $x = (x_E) \in \mathbb{R}^{\partse}$. Defining $\lvert x \rvert \coloneqq (\lvert x_E \rvert)_{E \in \partse}$ we have $\lVert \lvert x \rvert \rVert_1=  \lVert x \rVert_1$. Since the entries of $\matq$ are non-negative we have $\qe(x) \leq \qe(\lvert x \rvert)$. Hence, if $x \in \mathbb{R}^{\partse}$ is such that $\qe(x) = \max_{\lVert y \rVert_1 = 1} \qe(y)$ then we can assume that $x$ has only non-negative coordinates. In particular, we can take a sequence of rational vectors  with non-negative entries $(x_n)_{n \geq 1}$ with limit $x$ and satisfying $\lVert x_n\rVert_1 = 1$ for all $n \geq 1$, where the entries of $x_n$ have the same denominator~$b_n$. Noting  that $b_n$ is not necessarily coprime with its associated numerator, we can always assume that $b_n \to \infty$. Then by Proposition~\ref{proposition:boundcoreblock_maxqr} we have $\bestboundcoreblock[b_n] \geq \qe(b_n x_n) = b_n^2 \qe(x_n)$, thus since $\qe$ is continuous we obtain $\ell \geq \qe(x) = \bestboundcoreblock[\infty]$.
This concludes the proof.
\end{proof}

By Corollary~\ref{corollary:encadrement_final} we have
\[
\frac{1}{4}\left\lfloor\frac{e}{2}\right\rfloor
\leq
\bestboundcoreblock[\infty]
\leq
\frac{1}{2e}\left\lfloor\frac{e^2}{4}\right\rfloor.
\]
We will be able here to determine the exact value of the limit $\bestboundcoreblock[\infty]$.
%

\begin{corollary}
\label{corollary:r_infinite}
For any $e \geq 2$ we have
\[
\bestboundcoreblock[\infty] = \frac{1}{2e}\left\lfloor\frac{e^2}{4}\right\rfloor.
\]
\end{corollary}

\begin{proof}
Since $\qbound 0 e=\idealbound 0 e = 0$, we deduce from Lemma~\ref{lemma:idealbound_shift} and~\eqref{equation:QleqNleqN'} that  $\qbound r e = \bestboundcoreblock = \idealbound r e = \bigl\lfloor\frac{r^2}{2e}\lfloor\frac{e^2}{4}\rfloor\bigr\rfloor$ as soon as $r$ is of the form $r = \alpha\binom{e}{\lfloor\frac{e}{2}\rfloor}$ for some $\alpha \in \Z_{\geq 1}$. This proves that $(\bestboundcoreblock/r^2)_r$ has a subsequence that converges to  $\frac{1}{2e}\bigl\lfloor\frac{e^2}{4}\bigr\rfloor$ and this concludes the proof. 
\end{proof}

\section{Application in  level one}
\label{section:application}

Let $e \geq 2$. We will here apply the  implicit description of the set $\Qcharge$ given at Proposition~\ref{proposition:partitions_wgeq0} to study a shift operation on partitions, and relate it to a  shift operation on blocks that is naturally defined in higher levels.

In this whole section we assume that we are in level one, that is, we have $r = 1$. In particular, without loss of generality we can assume that the charge is zero and therefore we will often omit to write the corresponding superscripts. For instance, the weight of $\alpha = \sum_{i \in \Ze} \cialpha\alpha_i \in \Q$ is given by
\[
\w(\alpha) = \cialpha[0] - \frac{1}{2}\sum_{i \in \Ze} (\cialpha-\cialpha[i+1])^2,
\]
and, recalling Lemma~\ref{lemma:x_vide}, if $\lambda$ is an $e$-core and $i \in \{0,\dots,e-1\}$ then $\xrostamod_i(\lambda) = \xrostam_i(\lambda)$.
 The only case where we keep the superscript is for $\Qcharge[\chargez] = \{ \alpha(\lambda) : \lambda$ is a partition$\}$, to avoid confusion  with the ambient abelian group $\Q \supseteq \Qcharge[\chargez]$. Finally, we define
\begin{align*}
\Qc &\coloneqq \{ \alpha(\lambda) : \lambda   \text{ is an } e\text{-core}\}.
\end{align*}
We recall from Lemma~\ref{lemma:cores_w=0} that $\Qc = \bigl\{\alpha \in \Q : \w(\alpha) = 0\bigr\}$. Let $\ep \in\{1,\dots,e-1\}$ and let $\ordere \in \{2,\dots,e\}$ be the order of $\ep$ in $\Ze$.

\subsection{Shifting blocks}
\label{subsection:shifting_blocks}

\begin{definition}
\label{definition:sblock}
We define the $\mathbb{Z}$-linear map $\sblockprim : \Q \to \Q$ by $\sblock\alpha_i \coloneqq \alpha_{i - \ep}$ for all $i \in \mathbb{Z}/e\mathbb{Z}$.
\end{definition}

The map $\sblockprim$ is an automorphism of $\Q$ of order $\ordere$. As we saw in the introduction, the automorphism $\sblockprim$ appears together with a shift operation $\blam \mapsto \prescript{\sigma}{}\blam$ on $r$-partitions when $r \geq 2$. This shift operation on $r$-partitions is  defined by a (power of a) cyclic permutation of the components of $\blam$, and we have (under suitable conditions on the multicharge $\mcharge$):
\[
\alpha^{\mcharge}\bigl(\prescript{\sigma}{}{\blam}) = \sblock \alpha^\mcharge(\blam)
\] (see~\eqref{equation:intro_shift_block_parts}). In this Section~\ref{section:application} we are interested in the case $r = 1$. As we saw in Definition~\ref{definition:sblock}, the operation $\alpha \mapsto \sigma\cdot\alpha$ on $\Q$ is still defined, but it is unclear what the operation $\lambda \mapsto \prescript{\sigma}{}{\lambda}$ should be. Our aim is to study some properties of such an operation.

\medskip
 Note that :
\begin{equation}
\label{equation:sigma1}
\sblock\constQ=\constQ,
\end{equation}
recalling that $\constQ = \sum_{i \in \Ze} \alpha_i \in \Q$.

\begin{proposition}
\label{proposition:weight_shift}
Let $\alpha = \sum_{i\in\mathbb{Z}/e\mathbb{Z}} \cialpha\alpha_i \in \Q$. We have $\w(\sblock \alpha)=  \w(\alpha) + \cialpha[\ep] - \cialpha[0]$.
\end{proposition}

\begin{proof}
By definition, we have
\[
\sblock\alpha = \sum_{i \in \mathbb{Z}/e\mathbb{Z}} \cialpha[i+\ep] \alpha_i,
\]
thus
\begin{align*}
\w(\sblock\alpha)
&=
\cialpha[\ep] - \frac{1}{2}\sum_{i \in \mathbb{Z}/e\mathbb{Z}} \left(\cialpha[i+\ep] - \cialpha[i+\ep+1]\right)^2 
\\
&=
\cialpha[\ep] - \frac{1}{2}\sum_{i \in \mathbb{Z}/e\mathbb{Z}} \left(\cialpha[i] - \cialpha[i+1]\right)^2 
\\
&=
\w(\alpha) - \cialpha[0] + \cialpha[\ep].
\end{align*}
\end{proof}

Using Lemma~\ref{lemma:cores_w=0} and Proposition~\ref{proposition:partitions_wgeq0}, we deduce the following corollary.

\begin{corollary}
\label{corollary:shifting_blocks}
Let $\alpha = \sum_{i \in \mathbb{Z}/e\mathbb{Z}} \cialpha\alpha_i \in \Qp$. Then
\[\sblock\alpha \in \Qp \iff \cialpha[0]\leq   \cialpha[\ep] + \w(\alpha),\]
the right-hand side being an equality if and only if  $\sblock \alpha\in \Qc$. In particular, if $\alpha \in \Qc$ then
\[
\sblock \alpha\in \Qp \iff \cialpha[0] \leq \cialpha[\ep],
\]
 the right-hand side being an equality if and only if $\sblock\alpha \in \Qc$.
\end{corollary}

\subsection{Shifting partitions}
\label{subsection:shifting_partitions}

In the continuation of Corollary~\ref{corollary:shifting_blocks}, given a partition $\lambda$ such that $\sblock\alpha(\lambda) \in \Qp$ we show how to obtain the partitions that lie in $\sblock\alpha(\lambda)$. The constructions and results that we present in this subsection can be obtained using \emph{Scopes isometries} (see~\cite{scopes} and for instance~\cite{fayers:weights}), a Scopes isometry consisting in swapping two (adjacent) runners in the $e$-abacus.

\begin{definition}
\label{definition:shift_core}
Let $\lambda$ be an $e$-core. We denote by $\tspart\lambda$ the unique $e$-core such that
\[
\x(\tspart\lambda) = \bigl(\x[\ep](\lambda), \dots, \x[e-1](\lambda),\x[0](\lambda),\dots,\x[\ep-1](\lambda)\bigr).
\]
\end{definition}

Note that $\tspart\lambda$ is well-defined by Proposition~\ref{proposition:1-1_correspondance_cores_abaci_y} since $\x[\ep](\lambda) +  \dots + \x[e-1](\lambda) +\x[0](\lambda) +\dots+\x[\ep-1](\lambda) = 0$.  

\begin{proposition}
\label{proposition:alpha_tspartlambda}
Let $\lambda$ be an $e$-core and let  $\diffcont \coloneqq \cialpha[0](\lambda) - \cialpha[\ep](\lambda)$. We have
\[
\alpha(\tspart\lambda) = \sblock\alpha(\lambda) + \diffcont\constQ.
\]
\end{proposition}

\begin{proof}
By~\eqref{equation:yi_ci_ci+1}  we have $\x[0](\lambda) + \dots + \x[\ep-1](\lambda) = \diffcont$, moreover by Corollary~\ref{corollary:expression_ci_in_terms_of_yi} we have $\cialpha[0](\tspart\lambda) = \cialpha[0](\lambda)$. Hence, for any $i \in \mathbb{Z}/e\mathbb{Z}$ we have by~\eqref{equation:yi_ci_ci+1}
\begin{align*}
\cialpha(\tspart\lambda)
&=
\cialpha[0](\tspart\lambda) - \x[0](\tspart\lambda) - \dots - \x[i-1](\tspart\lambda)
\\
&=
\cialpha[0](\lambda) - \x[\ep](\lambda) - \dots - \x[\ep+i-1](\lambda)
\\
&=
\cialpha[0](\lambda) + \diffcont - \x[0](\lambda) - \dots - \x[\ep+i-1](\lambda)
\\
&=
\cialpha[\ep+i](\lambda) + \diffcont.
\end{align*}
We conclude since
\[
\alpha(\tspart\lambda) = \sum_{i\in\mathbb{Z}/e\mathbb{Z}} \cialpha(\tspart\lambda)\alpha_i = \sum_{i \in\mathbb{Z}/e\mathbb{Z}} \bigl(\cialpha[\ep+i](\lambda)+\diffcont)\alpha_i = \sum_{i \in\mathbb{Z}/e\mathbb{Z}} \cialpha[i](\lambda)\alpha_{i-\ep}  + \diffcont\constQ = \sblock \alpha(\lambda) + \diffcont\constQ.
\]
\end{proof}



The next result shows that Definition~\ref{definition:shift_core} fits with what we can expect between $\sblock\alpha(\lambda)$ and $\alpha(\spart\lambda)$ when $\lambda$ is an $e$-core.
Recall from Definition~\ref{definition:s-core} that any $\alpha \in \Q$ has an associated core.

\begin{corollary}
\label{corollary:core_shift}
Let $\lambda$ be a partition and let $\overline\lambda$ be its $e$-core. Then $\alpha(\spart\overline\lambda)$ is the core of $\sblock\alpha(\lambda)$. 
\end{corollary}

\begin{proof}
By definition, the partition $\spart\overline\lambda$ is an $e$-core thus the block $\alpha(\spart\overline\lambda)$ is a core block (recalling Remark~\ref{remark:coreblock_r=1}). The  assertion then follows from Lemma~\ref{lemma:core_block_associated_with_any_alpha} and  Proposition~\ref{proposition:alpha_tspartlambda}. 
\end{proof}

In particular, it follows from Remark~\ref{remark:score_ecore} that if $\lambda$ is a partition such that $\sblock\alpha(\lambda) \in \Qp$ then $\spart\overline{\lambda}$ is the common $e$-core of the partitions that lie in $\sblock\alpha(\lambda)$.
We now propose a generalisation of Definition~\ref{definition:shift_core} when $\lambda$ is not necessarily an $e$-core.

\begin{definition}
Let $\lambda$ be a partition and for any $i \in \{0,\dots,e-1\}$, let $R_i$ be the $i$-th runner of the $e$-abacus of $\lambda$. We define $\spart\lambda$ to be the partition whose $e$-abacus is obtained as follows: for any $i \in \{0,\dots,e-1\}$, the $i$-th runner is $R_{i+\ep}$ (where addition is modulo $e$).
\end{definition}

\begin{remark}
\label{remark:quotient_sigmalambda}
The partition $\spart\lambda$ is the (unique) partition with $e$-core $\spart\overline\lambda$ and $e$-quotient $\bigl(\lambda^{[\ep]},\dots,\lambda^{[e-1]},\lambda^{[0]},\dots,\lambda^{[\ep-1]}\bigr)$. In particular by~\eqref{equation:eweight_sum_length_equotient} we have $\we(\lambda) = \we\bigl(\spart\lambda\bigr)$. 
\end{remark}

\begin{corollary}
\label{corollary:case_spart_lambda_good}
Let $\lambda$ be a partition. We have
\[
\alpha\left(\spart\lambda\right) = \sblock\alpha(\lambda) \iff \lvert\spart\lambda \rvert = \lvert \lambda \rvert \iff \cialpha[0](\lambda) = \cialpha[\ep](\lambda).
\]
\end{corollary}

\begin{proof}
By Lemma~\ref{lemma:add_rim_hook_residues} we have $\cialpha[0](\lambda) - \cialpha[\ep](\lambda) = \cialpha[0](\overline\lambda) - \cialpha[\ep](\overline\lambda)$, thus from~\eqref{equation:sigma1}, Proposition~\ref{proposition:alpha_tspartlambda} and Remark~\ref{remark:quotient_sigmalambda} we obtain
\[
\alpha(\spart\lambda) = \sblock\alpha(\lambda) + \left(\cialpha[0](\lambda) - \cialpha[\ep](\lambda)\right)\mathbf{1}.
\]
Since $\lvert \sblock\alpha(\lambda)\rvert = \lvert \alpha(\lambda)\rvert = \lvert\lambda\rvert$, we deduce that $\lvert \spart\lambda\rvert = \lvert \lambda \rvert + \cialpha[0](\lambda) - \cialpha[\ep](\lambda)$ and this concludes the proof.
\end{proof}

\begin{remark}
Is is easy to construct an $e$-core (and thus, a partition)  that satisfies the equivalent conditions of Corollary~\ref{corollary:case_spart_lambda_good}. Indeed, if $\lambda$ is an $e$-core then by~\eqref{equation:yi_ci_ci+1} we have $\cialpha[0](\lambda) = \cialpha[\ep](\lambda)$ if and only if $\xrostamod_0(\lambda) + \dots + \xrostamod_{\ep-1}(\lambda) = 0$.
\end{remark}

\subsection{Some properties}
\label{subsection:some_properties}

In this subsection, we always assume that $\ep$ divides $e$. In particular, the order $\ordere$ of $\ep$ in $\Ze$ is $\ordere = \frac{e}{\ep}$. We will count both modulo $e$ and $\ep$, thus to avoid ambiguities we will add a prime when we compute modulo $\ep$. For instance, for any partition $\lambda$ and for any $\ip \in \{0,\dots,\ep-1\}$ we will denote by $\cialpha[\ip]'(\lambda)$ the number of nodes $\gamma$ of $\lambda$ such that $\res{}(\gamma) = \ip\pmod{\ep}$.
 The next lemma is immediate.

\begin{lemma}
\label{lemma:ci_ehat}
Let  $\ip \in \{0,\dots,\ep - 1\}$. For any partition $\lambda$ we have
\[
\cialpha[\ip]'(\lambda) = \sum_{k = 0}^{\ordere - 1} \cialpha[\ip + k\ep](\lambda).
\]
\end{lemma}

Let $\Q' = \oplus_{j \in \mathbb{Z}/\ep\mathbb{Z}} \mathbb{Z}\alphap_j$ be a free abelian group with basis $\{\alphap_j\}_{j \in\mathbb{Z}/\ep\mathbb{Z}}$. We consider the $\mathbb{Z}$-linear map $\pi : \Q \to \Q'$ determined by
\[
\pi(\alpha_i)\coloneqq \alphap_{i \bmod {\ep}},
\]
for all $i \in \mathbb{Z}/e\mathbb{Z}$. Note that $i \bmod \ep$ is well-defined since we have assumed that $\ep$ divides $e$. It follows from the definition of the automorphism $\sigma$ that
\[
\pi(\sblock\alpha) = \pi(\alpha) \in \Q',
\]
for any $\alpha \in \Q$.  We then deduce from Lemma~\ref{lemma:block_common_core} the following result.

\begin{proposition}
\label{proposition:shift_share_same_core}
Let $\alpha \in \Qp$ so that $\sblock\alpha \in \Qp$. If a partition $\lambda$ (respectively, $\mu$) lies in~$\alpha$ (resp. $\sblock\alpha$) then $\lambda$ and $\mu$ share the same $\ep$-core.
\end{proposition}


 Recall the classical result that any $\ep$-core is an $e$-core (cf. Lemma~\ref{lemma:removal_hook_beta_number}).

\begin{proposition}
\label{proposition:e_ehat_core}
Let $\lambda$ be an $e$-core and assume that $\sblock\alpha(\lambda) = \alpha(\lambda)$. Then $\lambda$ is  an $\ep$-core.
\end{proposition}

\begin{proof}
Since $\sblock\alpha(\lambda) = \alpha(\lambda)$ we have $\cialpha(\lambda) = \cialpha[i + \ep](\lambda)$ for any $i \in \{0,\dots,e-1\}$. We deduce that
\begin{align*}
\sum_{i = 0}^{e-1}\left(\cialpha(\lambda)-\cialpha[i+1](\lambda)\right)^2 &= \ordere \sum_{\ip=0}^{\ep-1} \left(\cialpha[\ip](\lambda)-\cialpha[\ip+1](\lambda)\right)^2,
\\
\intertext{and, by Lemma~\ref{lemma:ci_ehat},}
\cialpha[\ip]'(\lambda) &= \ordere \cialpha[\ip](\lambda),
\end{align*}
for any $\ip \in \{0,\dots,\ep-1\}$.
By Proposition~\ref{proposition:we=wc}, we obtain
\begin{align*}
\we[\ep](\lambda)
&=
\cialpha[0]'(\lambda) - \frac{1}{2}\sum_{\ip=0}^{\ep-1} \left(\cialpha[\ip]'(\lambda) - \cialpha[\ip+1]'(\lambda)\right)^2
\\
&=
\ordere \cialpha[0](\lambda) - \frac{\ordere^2}{2} \sum_{\ip =0}^{\ep-1} \left(\cialpha[\ip](\lambda) - \cialpha[\ip+1](\lambda)\right)^2
\\
&=
\ordere \cialpha[0](\lambda) - \frac{\ordere}{2}\sum_{i =0}^{e-1} \left(\cialpha(\lambda) - \cialpha[i+1](\lambda)\right)^2
\\
&=
\ordere \we(\lambda)
\\
&=
0,
\end{align*}
thus $\lambda$ is an $\ep$-core.
\end{proof}

The previous proposition involves a block $\alpha \in \Q$ satisfying $\alpha = \sblock\alpha$. We say that such a block is \emph{stuttering}. In a similar spirit as in~\cite{rostam:stuttering}, we will now study the relationship between stuttering blocks and stuttering partitions, that is, partitions $\lambda$ satisfying $\spart\lambda=\lambda$.

\begin{lemma}
\label{lemma:stuttering_partition}
Let $\lambda$ be a partition that satisfies $\spart\lambda =\lambda$. Then $\sblock\alpha(\lambda) = \alpha(\lambda)$ and $\ordere \mid \we(\lambda)$.
\end{lemma}

\begin{proof}
The first assertion is clear by Corollary~\ref{corollary:case_spart_lambda_good} since $\lvert \spart\lambda\rvert = \lvert\lambda\rvert$. For the second assertion, since $\spart\lambda=\lambda$ then we know by  Remark~\ref{remark:quotient_sigmalambda} that  the $e$-quotient of $\lambda$ is
\[
\left(\lambda^{[0]},\dots,\lambda^{[\ep-1]},\lambda^{[0]},\dots,\lambda^{[\ep-1]},\dots,\lambda^{[0]},\dots,\lambda^{[\ep-1]}\right),
\]
where the sequence $\lambda^{[0]},\dots,\lambda^{[\ep-1]}$ is repeated $\ordere$ times. Thus by~\eqref{equation:eweight_sum_length_equotient} we have
\[
\we(\lambda) = \ordere \sum_{i = 0}^{\ep-1} \bigl\lvert\lambda^{[i]}\bigr\rvert,
\]
which concludes the proof.
\end{proof}

The aim is now to give a converse statement for Lemma~\ref{lemma:stuttering_partition}, proving Theorem~\ref{theoremintro:stuttering} from the introduction.

\begin{proposition}
\label{proposition:stuttering_core_block}
Let $\lambda$ be an $e$-core. If $\sblock\alpha(\lambda) = \alpha(\lambda)$ then $\spart\lambda=\lambda$.
\end{proposition}

\begin{proof}
The assumption implies that  $\cialpha[0](\lambda) = \cialpha[\ep](\lambda)$. Hence, by Corollary~\ref{corollary:case_spart_lambda_good} we have $\alpha(\spart\lambda) =  \sblock\alpha(\lambda) = \alpha(\lambda)$, thus $\lambda$ and $\spart\lambda$ lie in the same block. This concludes since $\lambda$ is an $e$-core (recalling Lemma~\ref{lemma:block_common_core}).
\end{proof}

\begin{corollary}
\label{corollary:stuttering_blocks}
Let $\alpha \in \Qp$. If $\alpha = \sblock\alpha$  and $\ordere \mid \w(\alpha)$ then there exists $\lambda$ with $\alpha(\lambda) = \alpha$ satisfying $\spart\lambda = \lambda$.
\end{corollary}

\begin{proof}
Let $\mu$ be a partition with $\alpha(\mu) = \alpha$ and let $\overline\mu$ be the $e$-core of $\mu$. By Lemma~\ref{lemma:block_add_rim_hook_partition} and~\eqref{equation:sigma1}, since $\sblock\alpha(\mu) = \alpha(\mu)$  we have $\sblock\alpha(\overline\mu) = \alpha(\overline\mu)$. Thus, by Proposition~\ref{proposition:stuttering_core_block} we have $\spart\overline\mu = \overline\mu$. By assumption and~\eqref{equation:wclambda=wcalphalambda}, we can write $\we(\mu) = w \ordere$ with $w \in \N$.  Now let $\lambda$ be the partition whose $e$-abacus is obtained as follows: for each $i \in \{0,\dots,\ordere-1\}$, we slide $w$ times the rightmost bead on runner $i\ep$ of the $e$-abacus corresponding to $\overline\mu$. By Lemmas~\ref{lemma:block_add_rim_hook_partition} and~\ref{lemma:remove_hook_abacus}, the  partition~$\lambda$ satisfies $\alpha(\lambda) = \alpha(\overline\mu) + \ordere w \mathbf{1} = \alpha(\overline\mu) + \we(\mu)\mathbf{1} = \alpha(\mu)$, and by construction $\spart\lambda = \lambda$.
\end{proof}

Note that for any $h \in \N$, by~\eqref{equation:sigma1}  the block $\alpha = h\mathbf{1}\in\Qp$ satisfies $\alpha = \sblock\alpha$, however, we have $\ordere \mid \w(\alpha)$ if and only if $\ordere \mid h$.

\paragraph*{Acknowledgements} The author is thankful to Jérémy Le Borgne for many discussions about~\textsection\ref{subsubsection:bounding_bounds}, to Cédric Lecouvey for a discussion about Kac--Moody algebras and to Maria Chlouveraki and Nicolas Jacon for some suggestions. The author also thank the Centre Henri Lebesgue ANR-11-LABX-0020-0 and the referees for useful comments.

\paragraph*{Declarations}
\begin{itemize}
\item Funding:  the author did not receive support from any organization for the submitted work.
\item  Conflicts of interests/Competing interests: the author and has no relevant financial or non-financial interests to disclose.
\item Data availability statement: data sharing not applicable to this article as no datasets were generated or analysed during the current study.
\end{itemize}

\end{document}